\documentclass[11pt,a4paper, reqno]{amsart}

\usepackage[a4paper, left = 2cm,  hmargin={25mm,25mm},vmargin={25mm,25mm}]{geometry}
\usepackage[utf8]{inputenc}
\usepackage{geometry}
\usepackage[english]{babel}
\usepackage{graphicx}
\usepackage{color}
\usepackage{xcolor}
\usepackage{pdfpages}
\usepackage{amsbsy}
\usepackage{amssymb}
\usepackage{amsmath}
\usepackage{bm}
\usepackage{mathtools}
\usepackage{matlab-prettifier}
\usepackage{amsthm}
\usepackage{empheq}
\usepackage{tikz-cd}
\usepackage{amsfonts}
\usepackage{array}
\usepackage{hyperref}
\usepackage[OT2,T1]{fontenc}
\usepackage{csquotes}
\usepackage{mathabx}
\usepackage{bm}
\usepackage{bigints}

\numberwithin{equation}{section}

\newcommand\m[1]{\begin{pmatrix}#1\end{pmatrix}} 
\newcommand{\mf}{\mathfrak}

\DeclareSymbolFont{cyrletters}{OT2}{wncyr}{m}{n}
\DeclareMathSymbol{\Sha}{\mathalpha}{cyrletters}{"58}
\newcolumntype{P}[1]{>{\centering\arraybackslash}p{#1}}
\everymath{\displaystyle}

\newtheorem{theorem}{Theorem}[section]

\newtheorem{lemma}[theorem]{Lemma}

\newtheorem{defn}[theorem]{Definition}

\newtheorem{proposition}[theorem]{Proposition}

\newtheorem{remark}[theorem]{Remark}

\newtheorem*{notation}{Notation}
\newtheorem{corollary}[theorem]{Corollary}

\definecolor{lightgray}{gray}{0.5}
\usepackage[backend=biber, style=alphabetic,]{biblatex}
\begin{document}
\title{A Fourier-Jacobi Dirichlet series attached to modular forms of $\textup{SO}(2,4)$}
\author{Thanasis Bouganis and Rafail Psyroukis}
\address{Department of Mathematical Sciences\\ Durham University\\
South. Rd.\\ Durham, DH1 3LE, U.K..}
\email{athanasios.bouganis@durham.ac.uk, rafail.psyroukis@durham.ac.uk}
\subjclass[2020]{11F55, 11E88, 11E41, 11F66, 11M41}
\begin{abstract}
We consider a Dirichlet series $D(\bm{F},\bm{G};s)$ attached to two automorphic forms $\bm{F}$ and $\bm{G}$ of an orthogonal group of real signature $(2,4)$, involving their Fourier--Jacobi coefficients. When $\bm{F}$ is a Hecke eigenform and $\bm{G}$ a lift of a Jacobi-Poincaré series, our main result gives that $D(\bm{F},\bm{G};s)$ is equal to the standard $L$-function attached to $\bm{F}$, up to an explicit constant. To establish this, we use a correspondence between binary Hermitian forms and ideals of quaternion algebras, as established by Latimer, together with the fact that the even Clifford algebra of a three-dimensional definite quadratic space can be identified with a quaternion division algebra. Our work should be seen as a generalisation of a work of Kohnen and Skoruppa, whose result corresponds to the case of the orthogonal group of real signature $(2,3)$. 

\end{abstract}
\maketitle
\vspace{-1cm}
\section{Introduction}
Let $F$ and $G$ be two degree two Siegel cusp forms of even weight $k \geq 0$. In their groundbreaking work \cite{kohnen_skoruppa}, Kohnen and Skoruppa define and study the Dirichlet series 
\[
D_{F,G}(s):= \zeta(2s-2k+4) \sum_{N=1}^{\infty} \langle \phi_N, \psi_N \rangle N^{-s},
\]
where $\phi_N$ and $\psi_N$ are the $N$th Fourier--Jacobi coefficients of $F$ and $G$ respectively, and $\langle\,,\,\rangle$ is the Petersson inner product on Jacobi forms of weight $k$ and index $N$. The significance of their work derives mainly from the fact that if one further assumes that $F$ is a Hecke eigenform and $G$ is in the Maass space, then this Dirichlet series is proportional to $Z_F(s)$, the spinor $L$-function attached to $F$. Actually, as explained in their paper, it is sufficient to establish such a relation for the case where $G$ is taken to be the lift of a Jacobi-Poincaré series, as such lifts generate the Maass space in the symplectic setting. In particular, they show in \cite[Theorem 3]{kohnen_skoruppa},
\begin{equation}\label{kohnen-skoruppa-result}
    D_{F,\mathcal{P}_{k,D}}(s) = A(Q) Z_F(s), 
\end{equation}
where $Q$ denotes any quadratic form of discriminant $D$ representing $1$, and $A(Q)$ is 
the $Q$th Fourier coefficient of $F$. Here, $\mathcal{P}_{k,D}$ is a lift of a classical Poincare series associated to $k$ and $D$ (see \cite[page 552]{kohnen_skoruppa}). \newline

In this paper, we obtain a result analogous to Kohnen and Skoruppa's for modular forms of orthogonal groups of signature $(2,4)$. We will first describe our result and then explain how it fits within a broader framework, alongside Kohnen and Skoruppa's work.\newline

Let $V$ denote a two-dimensional quadratic space and let $\phi: V \times V \longrightarrow \mathbb{Q}$ denote a positive definite, symmetric, quadratic form. We assume that there is some $v\in V$ such that $\phi(v,v)=1$ (but see also Remark \ref{Condition on phi}). We then have that $(V,\phi)$ is isomorphic to $(K, N_{K/\mathbb{Q}})$, where $K$ is an imaginary quadratic field. For a fractional ideal $\mf{m}$ of $K$, we write $L_{\mf{m}}$ for the lattice in $V$ corresponding to the fractional ideal $\overline{\mf{m}}\, \mf{m}^{-1}$. It is known that $L_{\mf{m}}$ is a $\mathbb{Z}$-maximal lattice in $(V,\phi)$ and all $\mathbb{Z}$-maximal lattices of $(V,\phi)$ are of this form. We set $L:= L_{\mathcal{O}_K}$, and after fixing a $\mathbb{Z}$-basis for $L$ we may write $L = \mathbb{Z}^2$, $V=\mathbb{Q}^2$ and $S$ for a positive definite matrix such that
\begin{equation*}
    \phi(x,y) = \frac{1}{2}x^tSy,
\end{equation*}
for all $x,y \in \mathbb{Q}^2$. \newline

Consider now the quadratic spaces $(V_0,\phi_0)$ and $(V_1, \phi_1)$, which are represented by the matrices

\begin{equation*}
    S_0 := \m{&&1\\ &-S&\\1&&}, \,\,\, S_1 := \m{&&1\\ &-S_0&\\1&&},
\end{equation*}

with respect to the $\mathbb{Z}$-maximal lattices $L_0:=\mathbb{Z}^{4}$ and $L_1 := \mathbb{Z}^{6}$, respectively. We denote by $G_{\mathbb{Q}}^{*}$ and $G_\mathbb{Q}$ the special orthogonal groups associated with $S_0$ and $S_1$, respectively. \\

Let now $F$ denote an orthogonal cusp form of integer even weight $k \geq 0$ with respect to the subgroup $\Gamma(L_1):=\{\gamma \in G_{\mathbb{Q}} \mid \gamma L_1 = L_1\}$. If now $D:= \{g \in G_{\mathbb{A}} \mid gL_1 = L_1\}$, it is well-known that the class number of $G_{\mathbb{Q}}$ with respect to $D$ is one, and therefore $F$ corresponds uniquely to an automorphic cusp form of weight $k$ on $G_{\mathbb{A}}$, which we denote by $\bm{F}$.\\

Now, for each fractional ideal $\mf{m}$ of $K$, let $L_{\mf{m},1} := \mathbb{Z}^2 + L_{\mf{m}} + \mathbb{Z}^2$ denote the corresponding lattices in $V_1$. These are all $\mathbb{Z}$-maximal lattices in $V_1$, and because of the class number one property, we can pick isometries $\widetilde{h}_{\mf{m}} \in G_{\mathbb{Q}}$ such that $L_{\mf{m},1} = \widetilde{h}_{\mf{m}}L_1$. After picking these isometries in a particular way (see Lemma \ref{hm}), we use $\bm{F}$ to define modular forms $F_{\mf{m}}:= \bm{F}(\widetilde{h}_{\mf{m}};Z)$, which are then orthogonal modular forms with respect to $\Gamma(L_{\mf{m},1})$. Moreover, each $F_{\mf{m}}$ admits a Fourier-Jacobi expansion, and we write $\{\phi_{\mf{m},N}\}_{N=1}^{\infty}$ for these coefficients. In particular, they are Jacobi cusp forms of weight $k$ and lattice index $(L_{\mf{m}}, N\phi)$. \\

Furthermore, for each $\mf{m}$ and $D_{\mf{m}}\in \mathbb{Z}_{<0}$, we can define a Jacobi-Poincar\'{e} series $P_{\mf{m}} = P_{\mf{m}}(k,D)$ of weight $k$ and lattice index $L_{\mf{m}}$. We can then consider a lift of such Jacobi-Poincar\'{e} series, given by

\begin{equation*}
    \mathcal{P}_{\mf{m}}(\tau', z, \tau) := \sum_{N \geq 1}(V_{N}P_{\mf{m}})(\tau,z)e(N\tau'),
\end{equation*}
where $V_N$ is an index-raising operator acting on Jacobi forms of lattice index. This is, in fact, an element of the corresponding Maass space. Moreover, it can be shown that each such $\mathcal{P}_{\mf{m}}$ arises as $\bm{P}(\widetilde{h}_{\mf{m}}; Z)$, where $\bm{P}$ is the unique automorphic form corresponding to $\mathcal{P} := \mathcal{P}_{\mathcal{O}_K}$. \\

For any finite set of primes $\mathcal{Q}$, we then consider the Dirichlet series

\begin{equation*}
    D_{\mathcal{Q}}(F_{\mf{m}},\mathcal{P}_{\mf{m}};s) := \sum_{\substack{N=1\\(N,p)=1 \, \forall p \in \mathcal{Q}}}^{\infty} \langle \phi_{\mf{m},N}, V_N P_{\mf{m}}\rangle N^{-s},
\end{equation*}

where $\langle \ , \ \rangle$ denotes the inner product on the space of Fourier--Jacobi forms of weight $k$ and (lattice) index $(L_{\mf{m}}, N\phi)$. Moreover, after picking a set of representatives for $\textup{Cl}(K)$, we set 
\begin{equation*}
    D_{\mathcal{Q}}(\bm{F}, \bm{P};s):= \sum_{\mathfrak{m} \in \textup{Cl}(K)} \zeta_{\mf{m}, \mathcal{Q}}(s-k+3) D_{\mathcal{Q}}(F_{\mathfrak{m}},P_{\mathfrak{m}};s),
\end{equation*}
where $\zeta_{\mf{m}}(s)$ denotes the zeta function attached to the ideal class $\mathfrak{m}$ (see \eqref{ideal-class-zeta}). The main Theorem of this paper is the following.
\begin{theorem}\label{Theorem-Intro}
   
    Assume $K=\mathbb{Q}(\sqrt{-m})$ with $m$ square-free and $m \not \equiv 2 \pmod{4}$. Choose $\ell \in \mathbb{Z}_{\geq 1}$ so that it satisfies the conditions of Proposition \ref{stabilisers of lattices}, and set $\xi := (1,0,0, \ell)$. Assume $F$ is an eigenform for the corresponding Hecke algebra. 
    Then, for a finite set of primes $\mathcal{Q}$, we have

   \[
\zeta_{\mathcal{Q}}(2s-2k+4)D_{\mathcal{Q}}(\bm{F},\bm{P};s) = A(\xi)L_{\mathcal{Q}}(F;s-k+2),
\]

    where $L(F;s)$ is the standard $L$-function attached to $F$, and $A(\xi)$ is the $\xi^{th}$ Fourier coefficient of $F$.  In particular, if $A(\xi) \neq 0$, then $D_{\mathcal{Q}}(\bm{F},\bm{P};s)$ is proportional to $L_{\mathcal{Q}}(F;s-k+2)$.
\end{theorem}

We refer to Corollary \ref{Theorem for G} for a statement for the series $D(\bm{F},\bm{G};s)$ when $\bm{G}$ belongs to the space spanned by the forms $\bm{P}$ satisfying the conditions of the Theorem above. Here, $D(\bm{F},\bm{G};s)$ is defined similar to $D(\bm{F},\bm{P};s)$ (see \eqref{general-dirichlet-series}). \newline

We now discuss the main steps and the intuition behind our results. Our main goal has been to obtain an analogous result to that of Kohnen and Skoruppa. To that end, we consider a modular form $F$ for the full integral group $\Gamma(L_1)$, as well as a modular form $\mathcal{P}$, which is a lift of a Jacobi-Poincaré series, and hence an element of the corresponding Maass space (in fact, such lifts generate the Maass space). As we mentioned above, in the case $n=2$, the quadratic space $(V,\phi)$ is identified with $(K, N_{K/\mathbb{Q}})$, where $K$ is a quadratic imaginary field. Contrary to the result of Kohnen and Skoruppa, though, considering just the Dirichlet series $D(F, \mathcal{P};s)$ is not enough to obtain an Euler product. The reason for this is that, in general, there is not a single class of $\mathbb{Z}$-maximal lattices in $(V, \phi)$ (in the $n=1$ case, this is always true). It is, therefore, crucial to consider various Dirichlet series, parametrised by these classes. Denoting these classes by $L_{\mf{m}}$, where $\mf{m}$ runs over fractional ideals of $K$, our next step is to consider appropriate modular forms $F_{\mf{m}}$ corresponding to $L_{\mf{m}}$, which will be used to define these Dirichlet series. Because of the class number one property of the space $(V_1, \phi_1)$, $F$ can be identified with a unique automorphic form $\bm{F}$, which then can be used to produce the various $F_{\mf{m}}$. These are modular forms for the congruence subgroups of fixing the lattices $L_{\mf{m},1}:= \mathbb{Z}^{2} +  L_{\mf{m}} + \mathbb{Z}^2$ in $V_1$. By then considering the corresponding lifts $\mathcal{P}_{\mf{m}}$ of the Jacobi-Poincaré series for the lattices $L_{\mf{m}}$, we are able to define the Dirichlet series $D(F_{\mf{m}}, \mathcal{P}_{\mf{m}};s)$ for all $\mf{m}$. \\ 

The next key idea is that one can treat each $D(F_{\mf{m}}, \mathcal{P}_{\mf{m}};s)$ in an almost uniform manner.
For each $\mf{m}$, we consider the set of vectors $\xi$ in $V_0$, which satisfy $\phi_0[\xi]=\ell$ and $2\phi_0(\xi, L_{\mf{m},0}) = \mathbb{Z}$ for some $\ell > 0$. This set splits into a finite number of orbits under the action of $\Gamma(L_{\mf{m},0})$, and we write $\xi_{\mf{m},i}$ for a choice of representatives. Fixing $\xi := (1,0,0,\ell)$, we let $$H(\xi):=\{g \in G_{\mathbb{Q}}^{*} \mid g\xi = \xi\}.$$ This is a negative definite orthogonal group of rank $3$. Using Proposition \ref{first form of result}, we can rewrite $D(F_{\mf{m}}, \mathcal{P}_{\mf{m}};s)$ as an expression involving some zeta functions $\zeta(\xi_{\mf{{m},i}};s)$ counting the number of solutions of certain congruences, and Dirichlet series $D(F_{\mf{m}}, \xi_{\mf{m},i};s)$ involving the Fourier coefficients of $F$. Crucially, the $D(F_{\mf{m}}, \xi_{\mf{m},i};s)$ are independent of $\mf{m}$, and, thanks to work of Sugano \cite{sugano}, are related to the standard $L$-function attached to $F$, as well as to the standard $L$-functions of a basis of eigenforms $f_{j}$'s on $H(\xi)$ (see \eqref{expression3}). We note here that the work of Sugano plays in our paper the role that the classical work of Andrianov \cite{andrianov} plays in the work of Kohnen and Skoruppa.\\

Dealing with the $\zeta(\xi_{\mf{{m},i}};s)$ is at the heart of this paper. Our strategy is to exploit the isogeny between the unitary group $\textup{U}(2,2)$ and $\textup{SO}(2,4)$. In particular, we identify the $\xi_{i}$'s with classes of binary Hermitian forms of discriminant $\ell$, and hence relate the series $\zeta(\xi_{\mf{{m},i}};s)$ with $\mf{m}$-primitive representation numbers of the Hermitian form represented by $\xi_i$. Latimer, in \cite{Latimer1}, showed that binary hermitian forms correspond to specific classes of ideals in a certain quaternion algebra $B$, which allow us to identify each $\zeta(\xi_{\mf{{m},i}};s)$ as the partial zeta function associated with classes of ideals in $B$. We then identify such a quaternion algebra as an even Clifford algebra associated to $\textup{SO}(3)$, and interpret the $f_j$'s as lifts to modular forms on the corresponding Clifford group. The proof of our main result then follows using Shimura's work in \cite[Section 22.11]{shimura2004arithmetic}.\\

We now explain how the above results can be seen in a more general setting and how this paper is related to the recent work \cite{psyroukis_orthogonal} of one of the authors. It is a well-known fact that there is an accidental isogeny of the symplectic group $\textup{Sp}_2(\mathbb{R})$ and of $\textup{SO}(2,3)$, the special orthogonal group of signature $(2,3)$. Moreover, the spinor $L$-function attached to $F$ above can be identified with the standard $L$-function attached to a holomorphic modular form on an appropriate orthogonal group of real signature $(2,3)$, as for example is discussed in \cite[Section 25]{shimura2004arithmetic}. In fact, there are more accidental isogenies between classical and orthogonal groups. Examples are $\textup{U}(2,2)$ and $\textup{SO}(2,4)$, and $\textup{Sp}(2, \mathbb{H})$ and $\textup{SO}(2,6)$, where $\mathbb{H}$ denotes the skew field of real quaternions. Modular forms can be defined for orthogonal groups of any real signature $(2,n+2) $, $ n \geq 1$. 
It is therefore natural to ask whether the method of Kohnen and Skoruppa can be generalised to the orthogonal setting.\\

A first attempt at such a result was recently made by one of the authors in \cite{psyroukis_orthogonal}, although it differs in several fundamental respects from the present work. To start with, \cite{psyroukis_orthogonal} only considers a single Dirichlet series, corresponding to a modular form $F$ of full level. Moreover, the condition that the algebraic subgroup $H(\xi)$ has class number one must be imposed in order to establish a clear-cut Euler product. These conditions substantially restrict the range of admissible orthogonal groups.\\

This paper serves as a generalisation of the work \cite{psyroukis_orthogonal} for the case of signature $(2,4)$ in two crucial ways. First, we are able to remove the class number one assumption on the group $H(\xi)$ by exploiting the accidental isogeny between $\textup{U}(2,2)$ and $\textup{SO}(2,4)$. Moreover, we introduce--for the first time--the idea of considering multiple Dirichlet series, roughly corresponding to the number of classes in the genus of the definite group corresponding to $\phi$. This refinement is essential in order to be able to prove our result for number fields $K$ of any class number. We should mention here that most of the ideas discussed up until Section \ref{section-dirichlet-series} can be generalised for arbitrary $n \geq 1$. We expect that a similar approach, considering multiple Dirichlet series, should be the appropriate method to obtain a clear-cut Euler product in the general case.\\

Finally, our result generalises another classical result, namely that of Gritsenko \cite{gritsenko}. In particular, Gritsenko generalised the result of Kohnen and Skoruppa for the case of Hermitian modular forms of degree two over the imaginary quadratic field $K = \mathbb{Q}(i)$. Gritsenko’s approach, which relies on factorisation methods in parabolic Hecke rings, is very different from the techniques used in this paper. Indeed, our work is closer to the approach of Kohnen and Skoruppa \cite{kohnen_skoruppa}, in which the representation numbers of binary Hermitian forms play an analogue role to those of binary quadratic forms.
To this, we should add that it is not difficult to see that Gritsenko's method can be generalised to any number field $K$ of class number $1$. 
However, to the best of our knowledge, nothing is known, or at least published, for the case of a class number larger than one. In that sense, the present work may also be viewed as a generalisation of Gritsenko’s results to arbitrary number fields as in Theorem \ref{main theorem}, due to the accidental isogeny $\textup{U}(2,2)$ and $\textup{SO}(2,4)$.

\begin{notation}
    We denote the space of $m\times n$ matrices with coefficients in a ring $R$ with $\textup{M}_{m,n}(R)$. If $n=m$, we often use the notation $\textup{M}_{n}(R)$. By $1_n$, we denote the $n\times n$ identity matrix. For any matrix $M \in \textup{M}_{n}(R)$, we denote by $\det(M), \textup{ tr}(M)$ the determinant and trace of $M$ respectively. By $\textup{GL}_n(R)$, we denote the matrices in $M_{n}(R)$ with non-zero determinant and by $\textup{SL}_n(R)$ the matrices with determinant $1$. For any vector $v$, we denote by $v^{t}$ its transpose. We also use the bracket notation $A[B] := \overline{B}^{t}AB$ for suitably sized complex matrices $A,B$. By $\textup{diag}(A_1, A_2, \cdots, A_n)$, we will denote the block diagonal matrix with the matrices $A_1, A_2,\cdots, A_n$ in the diagonal blocks. For a complex number $z$, we denote by $e(z) := e^{2\pi i z}$. Finally, let $\zeta(s)$ denote the Riemann zeta function.
    
\end{notation}

\section{Preliminaries}\label{preliminaries}
Let $F$ denote the field $\mathbb{Q}$ or $\mathbb{Q}_p$ for any rational prime $p$. Let $V$ denote a finite-dimensional vector space over $F$ and let $\phi: V \times V \longrightarrow F$ be a non-degenerate symmetric bilinear form. We call $(V,\phi)$ a quadratic space. We put $\phi[x] := \phi(x,x)$ for all $x \in V$. We denote by $A(V,\phi)$ or simply by $A(V)$ the Clifford algebra associated to $(V,\phi)$. We define the canonical automorphism $x \longmapsto x'$ and the canonical involution $x \longmapsto x^{*}$ of $A(V)$ by the property $v = -v' = v^{*}$ for all $v \in V$. We then set
\begin{equation*}
    A^{+}(V) := \{\alpha \in A(V) \mid \alpha' = \alpha\},
\end{equation*}
\begin{equation*}
    A^{-}(V) := \{\alpha \in A(V) \mid \alpha' = -\alpha\}.
\end{equation*}
We also define the special orthogonal and Clifford group by
\begin{equation*}
    \textup{SO}^{\phi}(V) := \{\alpha \in \textup{GL}(V) \mid \phi[\alpha x] = \phi[x], \ \forall x \in V\},
\end{equation*}
\begin{equation*}
    G(V) :=\{\alpha \in A(V)^{\times} \mid \alpha V\alpha^{-1} = V\},
\end{equation*}
respectively, and put $G^{+}(V) := G(V) \cap A^{+}(V)$, the even Clifford group. We define the homomorphisms
\begin{align}\label{tau map}
    \begin{split}
        \tau : G^{+}(V) &\longrightarrow \textup{SO}^{\phi}(V)\\ 
        \alpha &\longmapsto \tau(\alpha)x := \alpha x\alpha^{-1}.
    \end{split}
\end{align}
and
\begin{align}\label{nu map}
    \begin{split}
        \nu : G^{+}(V) &\longrightarrow F^{\times}\\
        \alpha &\longmapsto \alpha \alpha^{*}.
    \end{split}
\end{align}
It is well known that $\tau$ is surjective with kernel $F^{\times}$. We also define the spinor norm 
\begin{align*}
    \sigma : \textup{SO}^{\phi}(V) \longrightarrow \mathbb{Q}^{\times}/ (\mathbb{Q}^{\times})^2
\end{align*}
as follows. If $\alpha \in \textup{SO}^{\phi}(V)$, we pick an element $\xi \in G^{+}(V)$ such that $\tau(\xi) = \alpha$ and denote by $\sigma(\alpha)$ the coset $\nu(\xi)(\mathbb{Q}^{\times})^2$, which is determined by $\alpha$.
\\

Let now $\mathfrak{g}$ denote a maximal order in $F$, i.e. $\mathfrak{g} = \mathbb{Z}$ or $\mathbb{Z}_p$, depending on $F = \mathbb{Q}$ or $\mathbb{Q}_p$. By a $\mathfrak{g}$-lattice $L$ in $V$, we mean a finitely generated $\mathfrak{g}$-module that spans $V$ over $F$. It is called $\mathfrak{g}$-integral if $\phi[x] \in \mathfrak{g}$ for all $x \in L$. We say $L$ is $\mathfrak{g}$-maximal if it is maximal among all integral $\mathfrak{g}$-lattices. We also define its dual lattice via
\begin{equation*}
    L^{*} := \{x \in V \mid 2\phi(x, y) \in \mathfrak{g}, \ \forall y \in L\}.
\end{equation*}
The level $q$ of $L$ is defined as the least positive integer such that $q \phi[x] \in \mathfrak{g}, \ \forall x \in L^{*}$. A vector $v \in L$ is called primitive if
    \begin{equation*}
        \forall k \in \mathbb{Z}, w \in L: (v = kw \implies k=\pm 1).
    \end{equation*}
We also define the integral orthogonal group via
\begin{equation*}
    \textup{SO}^{\phi}(L) := \{\alpha \in \textup{SO}^{\phi}(V) \mid \alpha L = L\}.
\end{equation*}
By $A(L)$ we denote the subring of $A(V)$ which is generated by $L$ and $\mathfrak{g}$. We also put $A^{+}(L) := A(L) \cap A^{+}(V)$. \\

For any $q \in \mathbb{Q}^{\times}$ and $\mathfrak{b}$ a fractional ideal of $\mathbb{Q}$, we define
\begin{equation}\label{shimura set}
    L[q, \mathfrak{b}] := \{x \in V \mid \phi[x] = q \textup{ and } \phi(x, L) = \mathfrak{b}\}. 
\end{equation}
Finally, given two lattices $L, M$ in $V$, we denote by $[L/M]$ the fractional ideal generated by $\det \alpha$ for all $\alpha \in \textup{GL}(V)$ such that $\alpha L \subset M$.
\section{Modular Forms on $\textup{SO}(2,n)$}\label{modular-forms-section}
Let now $V := \mathbb{Q}^{n}$ and $L := \mathbb{Z}^{n}$ with $n \geq 1$ (both identified with the standard basis). Assume $S$ is an even integral positive definite symmetric matrix of rank $n$. Here, even means $S[x]:=x^t Sx \in 2\mathbb{Z}$ for all $x \in L$. We define
\begin{equation*}
    S_0 := \m{&&1\\ &-S&\\1&&}, \ S_1 := \m{&&1\\&S_0&\\1&&}
\end{equation*}
of real signatures $(1,n+1)$ and $(2,n+2)$, respectively.

Let also $V_0:= \mathbb{Q}^{n+2}$ and $V_1 := \mathbb{Q}^{n+4}$ and consider the quadratic spaces $(V_0, \phi_0), \ (V_1, \phi_1)$, where
\begin{align*}
    \phi_i : V_i \times V_i &\longmapsto \mathbb{Q}\\
    (x,y) &\longmapsto \frac{1}{2}x^{t}S_iy,
\end{align*}
for $i= 0, 1$. \\

We then have that $\phi:=\phi_{0} \mid_{V\times V}$ is just $(x, y) \longmapsto -x^{t}Sy/2$, and we make the assumption that $L = \mathbb{Z}^{n}$ is a maximal $\mathbb{Z}$-lattice with respect to $\phi$. \\

From \cite[Lemma 6.3]{shimura2004arithmetic}, we then obtain that $L_0:= \mathbb{Z}^{n+2}$ is a $\mathbb{Z}$-maximal lattice in $V_0$.\\

If now $R$ is a  $\mathbb{Q}$-algebra, the corresponding orthogonal groups of $R$-rational points are given by 
\begin{equation*}
    G_R^{*} = \textup{SO}^{\phi_0}(V_0 \otimes R) := \{g \in \textup{SL}_{n+2}(R) \mid g^{t}S_0g = S_0\},
\end{equation*}
\begin{equation*}
    G_R = \textup{SO}^{\phi_1}(V_1 \otimes R):= \{g \in \textup{SL}_{n+4}(R) \mid g^{t}S_1g = S_1\}.
\end{equation*}

In particular, for each prime $p$, we will abbreviate by $G_p$ the group $G_{\mathbb{Q}_p}$. We will write $G_{\mathbb{A}}$ for the adelisation of $G$, $G_{\mathbb{R}}$ for the infinite part of $G_{\mathbb{A}}$ and $G_{\mathbb{A},f}$ for its finite part. We use analogous notation for $G^{*}$. \\

We now collect some standard facts about automorphic forms on $G$ (see for example \cite{sugano} or \cite{krieg_jacobi}). We let $\mathcal{H}_{S}$ denote one of the connected components of
\begin{equation*}
    \{Z \in V_0 \otimes_{\mathbb{Q}} \mathbb{C} \mid \phi_0[\textup{Im}Z] > 0\}.
\end{equation*}
In particular, if we let 
\begin{equation*}
    \mathcal{P}_{S} := \{y'=(y_1,y,y_2) \in \mathbb{R}^{n+2} \mid y_1 >0, \,\, \phi_0[y']>0\},
\end{equation*}
we choose
\begin{equation}\label{symmetric space}
    \mathcal{H}_S := \{z=x+iy \in V_0 \otimes_{\mathbb{R}} \mathbb{C} \mid y \in \mathcal{P}_S\}.
\end{equation}
For a matrix $g \in M_{n+4}(\mathbb{R})$, we write it as
\begin{equation*}
    g = \m{\alpha & a^t &\beta \\ b&A&c\\ \gamma&d^t&\delta},
\end{equation*}
with $A \in \textup{M}_{n+2}(\mathbb{R}), \alpha,\beta,\gamma,\delta \in \mathbb{R}$ and $a,b,c,d$ real column vectors. Now the map
\begin{equation}\label{action_orthogonal}
    Z \longmapsto g\langle Z\rangle = \frac{-\frac{1}{2}S_0[Z]b+AZ+c}{-\frac{1}{2}S_0[Z]\gamma+d^tZ+\delta}
\end{equation}
gives a well-defined transitive action of $G_{\mathbb{R}}^{+}$ on $\mathcal{H_{S}}$, where $G_{\mathbb{R}}^{+}$ denotes the connected component of the identity of $G_{\mathbb{R}}$. The denominator of the above expression is the factor of automorphy
\begin{equation*}
    j(g, Z) := -\frac{1}{2}S_0[Z]\gamma+d^tZ+\delta.
\end{equation*}
Let $Z_0$ denote any point of $\mathcal{H}_{S}$ with real part $0$ and denote by $D_{\mathbb{R}}$ its stabiliser in $G_{\mathbb{R}}^{+}$. Then, we have that $G_{\mathbb{R}}^{+}/D_{\mathbb{R}} \cong \mathcal{H}_{S}$.\\

In the following, we will need the notion of an automorphic form. For each prime number $p < \infty$, let $D_{p} := G_{p} \cap \textup{SL}_{n+4}(\mathbb{Z}_{p})$ and let 
    \begin{equation*}
        D_{f} := \prod_{p < \infty} D_{p}.
    \end{equation*}

\begin{defn}\label{automorphic form}
    Let $k \geq 0$. A function $\bm{F} : G_{\mathbb{A}} \longrightarrow \mathbb{C}$ is called a holomorphic cusp form of weight $k$ with respect to $D_{f}$ if the following conditions hold:
    \begin{enumerate}
        \item $\bm{F}(\gamma g u) = \bm{F}(g) \textup{ }\forall \gamma \in G_{\mathbb{Q}}, u \in D_{f}$.
        \item For any $g=g_{\infty}g_{f}$ with $g_{\infty} \in G_{\mathbb{R}}^{+}$ and $g_{f} \in G_{\mathbb{A},f}$, $\bm{F}(g_{\infty}g_{f})j(g_{\infty},Z_0)^{k}$ depends only on $g_{f}$ and $Z = g_{\infty}\langle Z_0\rangle$ and is holomorphic on $\mathcal{H}_{S}$ as a function of $Z$.
        \item $\bm{F}$ is bounded on $G_{\mathbb{A}}$.
    \end{enumerate}
    Denote the above space by $\mathfrak{S}_{k}(D_{f})$.
\end{defn}

Let now $\Gamma$ denote a congruence subgroup of $G_{\mathbb{Q}}\cap G_{\mathbb{R}}^{+}$, in the sense of \cite[p. 62]{Shimura_Euler_Product}. We then have the following definition of an orthogonal modular form in the classical sense.

\begin{defn}\label{modular_form_defn}
A holomorphic function $F : \mathcal{H}_S \longrightarrow \mathbb{C}$ is called a modular form of weight $k \geq 0$ with respect to $\Gamma$ if it satisfies the equation
\begin{equation*}
    (F|_k \gamma)(Z) := j(\gamma, Z)^{-k}F(\gamma\langle Z\rangle) = F(Z)
\end{equation*}
for all $\gamma \in \Gamma$ and $Z \in \mathcal{H}_S$. We will denote the set of such forms by $M_{k}(\Gamma)$.
\end{defn}

Given an $\bm{F}$ as in Definition \ref{automorphic form} and an $g_{f} \in G_{\mathbb{A},f}$, we can associate an orthogonal modular form. Indeed for $Z \in \mathcal{H}_{S}$, we define
\begin{equation}\label{adelic F}
    \bm{F}(g_{f}; Z) := \bm{F}(g_{\infty}g_{f})j(g_{\infty}, Z_0)^{k},
\end{equation}
where $g_{\infty} \in G^{+}_{\mathbb{R}}$ is chosen so that $Z = g_{\infty}\langle Z_0\rangle$. Let now
\begin{equation*}
    \Gamma(g_{f}) = G_{\mathbb{Q}} \cap (G_{\mathbb{R}}^{+} \times g_{f}D_{f}g_{f}^{-1}),
\end{equation*}
which is a discrete subgroup of $G_{\mathbb{R}}^{+}$. We then have
\begin{equation*}
    \bm{F}(g_{f}; \gamma \langle Z\rangle) = j(\gamma, Z)^{k}\bm{F}(g_{f}; Z)
\end{equation*}
for all $\gamma \in \Gamma(g_{f})$ and $Z \in \mathcal{H}_S$. In particular, $\bm{F}(g_f; Z) $ is an orthogonal modular form of weight $k$ with respect to $\Gamma(g_f)$, in the sense of Definition \ref{modular_form_defn}. \\

From \cite[p. 29]{sugano_lifting}, we can identify $\mathfrak{S}_k(D_f)$ with $S_k(\Gamma(id))$ by sending $\bm{F} \longmapsto F(id; Z)$.\\

Now, if $X \in V_0$, define the element $\gamma_{X} \in G$ by
\begin{equation*}
    \gamma_{X} := \m{1&-X^{t}S_0 & -\frac{1}{2}S_0[X]\\ 0&1_{n+2}&X\\0&0&1}.
\end{equation*}
The holomorphic function $\bm{F}(g_{f};Z)$ is invariant under $Z \longmapsto Z+X$ for $X$ in the lattice
\begin{equation}\label{lattice-sugano}
    L(g_{f}):=\left\{X \in V_{0} \mid \gamma_{X} \in \Gamma(g_f)\right\}.
\end{equation}
Hence, as it is explained in \cite{sugano}, every such function then admits a Fourier expansion of the form
\begin{equation}\label{fourier-expansion-classical-orthogonal}
    \bm{F}(g_{f}; Z) = \sum_{\substack{r \in L(g_{f})^{*}\\ir \in \mathcal{H}_S}}a(g_{f}; r)e(2\phi_0(r, Z)),
\end{equation}
where 
\begin{equation*}
    L(g_{f})^{*}:=\left\{X \in V_0 \mid 2\phi_0(X,Y) \in \mathbb{Z} \textup{ for all } Y\in L(g_f)\right\}
\end{equation*}
is the dual lattice of $L(g_f)$ with respect to $\phi_0$.\\

Finally, let us introduce adelic Fourier coefficients. Let $\chi = \prod_{p \leq \infty}\chi_p$ be a character of $\mathbb{Q}_{\mathbb{A}}$ such that $\chi_{\mid \mathbb{Q}}:=1$ and $\chi_{\infty}(x) := e(x)$ for all $x \in \mathbb{R}$. For $\eta \in V_0$ and $g \in G_{\mathbb{A}}$, we define
\begin{equation}\label{adelic-fourier-coefficient}
    \bm{F}_{\chi}(g;\eta) := \int_{V_0\backslash V_{\mathbb{A}}} \bm{F}(\gamma_{\chi}g)\chi\left(-2\phi_0(\eta,X)\right) \textup{d}X.
\end{equation}
Now, for $g_{\infty} \in G_{\infty}^{0}$ and $g_{f} \in G_{\mathbb{A},f}$ we have
\begin{equation}\label{property of F_x}
    \bm{F}_{\chi}(g_{\infty}g_{f}; \eta) = a(g_{f};\eta)j(g_{\infty};Z_0)^{-k}e(2\phi_0(\eta, g_{\infty}\langle Z_0\rangle)).
\end{equation}
\section{The $n = 2$ case}\label{dirichlet series start}

We will now focus on the $n=2$ case, i.e. $\dim V = 2$. Assume there is a vector $v \in V$ such that $\phi[v] = 1$. Since $\phi$ is positive definite, hence anisotropic, from \cite[Section 7.2]{shimura2004arithmetic}, we obtain that $(V,\phi)$ is isomorphic with $(K, N_{K/\mathbb{Q}})$, where $K/\mathbb{Q}$ is a quadratic imaginary extension. Let $K := \mathbb{Q}(\sqrt{-m})$, where $m$ is a positive square-free integer and fix a basis $K = \mathbb{Q} + \mathbb{Q}\omega$, where 
\begin{equation}\label{omega}
    \omega = 
    \begin{cases}
    \sqrt{-m} \,\,\,\,\textup{ if } m\equiv 1,2 \pmod 4 \\
    \frac{1+\sqrt{-m}}{2} \,\,\,\,\textup{ if } m\equiv 3 \pmod 4
\end{cases}.
\end{equation}
This means that there is a linear isomorphism
\begin{align}\label{theta map}
    \theta : (V, \phi) &\longrightarrow (K, N_{K/\mathbb{Q}})\\\nonumber
    \m{a \\ b} &\longmapsto a+b\omega,
\end{align}
which satisfies $N_{K/\mathbb{Q}}(\theta(v)) = \phi[v], \  \forall v \in V$. In particular, we have that with respect to the standard basis of $V$, $\phi$ is represented by the matrix
\begin{equation}\label{matrix S}
    S := 
    \begin{cases}
    \m{2&0\\0&2m} \textup{ if } m\equiv 1,2 \pmod 4 \\
    \m{2&1\\1 & (m+1)/2} \textup{ if } m\equiv 3 \pmod 4
\end{cases}.
\end{equation}
\begin{remark}\label{Condition on phi}
    The condition $\phi[v]=1$ is not too restrictive. In particular, if this condition does not hold, then $(V,\phi)$ would be isomorphic to $(K, cN_{K/\mathbb{Q}})$, where $c \in \mathbb{Q}^{\times}$. But since $\textup{SO}^{\phi}(V) = \textup{SO}^{c^{-1}\phi}(V)$, we still get the isomorphism of the corresponding orthogonal groups.
\end{remark}

We would now like to consider the $\mathbb{Z}$-maximal lattices in $(V,\phi)$. We use the identification above. From \cite[Section 9.25]{shimura2004arithmetic}, the genus of the $\mathbb{Z}$-maximal lattices in $(K, N_{K/\mathbb{Q}})$ consists of all the $\mathcal{O}_K$-fractional ideals $\mathfrak{m}$ with the property that $\mathfrak{m}\, \overline{\mathfrak{m}} = \mathcal{O}_K$. We now describe the classes within this genus. We write $\textup{Cl}(K)$ for the class group of $K$ and $\textup{Cl}(K/\mathbb{Q})$ for the subgroup of $\textup{Cl}(K)$ generated by all principal ideals and the ideals $\mathfrak{a}$ such that $\mathfrak{a} = \overline{\mathfrak{a}}$. Then, as it is stated in \cite[Section 9.25]{shimura2004arithmetic} we have a bijection between $\textup{Cl}(K)/\textup{Cl}(K/\mathbb{Q})$ and the set of classes in the genus of the $\mathbb{Z}$-maximal ideals, which is given by sending an ideal $\mathfrak{m}$ to $\mathfrak{m}^{-1} \overline{\mathfrak{m}} = N(\mathfrak{m})^{-1} \overline{\mathfrak{m}}^2$. In particular, there are (see \cite[equation (9.16)]{shimura2004arithmetic}
\[
[\textup{Cl}(K):\textup{Cl}(K/\mathbb{Q})]= (\# \textup{Cl}(K)) \,2^{2-\kappa} 
\]
classes, where $\kappa$ is the number of ramified places, including the archimedean. For each such class, we will write $L_{\mathfrak{m}}$ for the $\mathbb{Z}$-maximal lattice in $V$ corresponding to the ideal $N(\mathfrak{m})^{-1}\overline{\mathfrak{m}}^2$. From the discussion in \cite[Section 9.27]{shimura2004arithmetic}, to the class of lattices represented by $L_\mathfrak{m}$ we may attach a matrix $S_{\mathfrak{m}}$ in the genus of $S$. In particular, let $A_\mathfrak{m} \in \textup{SL}_2(\mathbb{Q})$ such that $A_{\mathfrak{m}} L_{\mathcal{O}_K} = L_\mathfrak{m}$. We then set $S_{\mathfrak{m}}:= A_{\mathfrak{m}}^{t}SA_{\mathfrak{m}}$. \newline

We now consider the six-dimensional quadratic space $(V_1,\phi_1)$ as defined in Section \ref{modular-forms-section}, and we fix a Witt decomposition corresponding to the matrix $S_1$. That is, we may write
\[
V_1 = V + \sum_{i=1}^2 (\mathbb{Q}e_i + \mathbb{Q}f_i)
\]
where the vectors $e_1$ and $e_2$ span a maximal totally isotropic space in $V_1$ and so do $f_1$ and $f_2$. By \cite[Lemma 6.3]{shimura2004arithmetic}, we may consider the $\mathbb{Z}$-maximal lattices in $(V_1,\phi_1)$ defined as
\[
L_{\mathfrak{m},1} = L_{\mathfrak{m}} + \sum_{i=1}^2 (\mathbb{Z}e_i + \mathbb{Z}f_i),
\]
where $L_{\mathcal{O}_K,1} = \mathbb{Z}^6$. As in Section \ref{modular-forms-section}, we just write $L_1$ for this lattice. Similarly, we consider the lattices
\[
L_{\mathfrak{m},0} = \mathbb{Z}e_2 + L_{\mathfrak{m}} + \mathbb{Z}f_2,
\]
where this time we view $e_2, f_2$ as elements of $V_0$. Again $L_{\mathcal{O}_K, 0} = L_0$. For each $\mathfrak{m}$, we define the groups
\begin{equation*}
    \Gamma(L_{\mathfrak{m},0}):=\{g \in G^{*}_{\mathbb{Q}} \mid gL_{\mathfrak{m},0}=L_{\mathfrak{m},0}\}, \,\,\,\,\,\, \Gamma^{+}(L_{\mathfrak{m},0}) = \Gamma(L_{\mathfrak{m},0}) \cap G_{\mathbb{R}}^{*, +}
\end{equation*}
\begin{equation*}
    \Gamma(L_{\mathfrak{m},1}) := \{g \in G_{\mathbb{Q}} \mid gL_{\mathfrak{m},1} = L_{\mathfrak{m},1}\}, \,\,\,\,\,\, \Gamma^{+}(L_{\mathfrak{m},1}) = \Gamma(L_{\mathfrak{m},1}) \cap G_{\mathbb{R}}^{ +}
\end{equation*}

Finally, we define the discriminant kernels by
\begin{equation*}
    \widehat{\Gamma}(L_{\mathfrak{m},0}):= \{M \in \Gamma^{+}(L_{\mathfrak{m},0}) \mid (M -1_{4})L_{\mathfrak{m},0}^{*} \subset L_{\mathfrak{m},0}\}.
\end{equation*}
\begin{equation*}
    \widehat{\Gamma}(L_{\mathfrak{m},1}):= \{M \in \Gamma^{+}(L_{\mathfrak{m},1}) \mid (M -1_{6})L_{\mathfrak{m},1}^{*} \subset L_{\mathfrak{m},1}\}.
\end{equation*}

We now have the following Lemma.
\begin{lemma}\label{hm}
    For each fractional $\mathcal{O}_K$-fractional ideal $\mathfrak{m}$, there is $h_{\mathfrak{m}} \in G_{\mathbb{Q}}^{*} \cap G_{\mathbb{R}}^{*,+}$ such that $L_{\mathfrak{m},0} = h_{\mathfrak{m}}L_0$ and $\sigma(h_{\mathfrak{m}}) = N(\mathfrak{m})(\mathbb{Q}^{\times})^{2}$.
\end{lemma}
\begin{proof}
    Since the lattices $L_{\mathfrak{m}}$ belong in the genus of $L=\mathbb{Z}^2$, there is $g_{\mathfrak{m}} \in \textup{SO}^{\phi}(\mathbb{A})$ such that $L_{\mathfrak{m}} = g_{\mathfrak{m}}L$. Without loss of generality, we may choose all the $g_{\mf{m}}$ to be trivial at infinity. Clearly 
    \begin{equation}\label{lattices-relation-adelic-element}
        L_{\mathfrak{m},0} = \textup{diag}(1, g_{\mathfrak{m}},1)L_{\mathfrak{m}}.
    \end{equation}

We will now show that we can choose $h_{\mf{m}}$ with the right conditions. For each prime $p<\infty$, we set $D_p^* :=G_p^* \cap \textup{SL}_{4}(\mathbb{Z}_p)$ and we let
\begin{equation*}
    D^* := G^*_{\mathbb{R}}\prod_{p<\infty}D^*_p.
\end{equation*}

We will choose the global element $h_{\mf{m}}$ as a ``good'' approximation to $\textup{diag}(1, g_{\mf{m}}, 1)$. That is, we claim that we may define $h_{\mf{m}} \in G^*_{\mathbb{Q}}$ such that $h_{\mf{m}} = \textup{diag}(1, g_{\mf{m}}, 1) d$ where $d \in D^*$ (in particular $d L_0 = L_0$), such that for the finite part of its spinor norm we have that 
\begin{equation*}
    \sigma(d)_f \in \left(\prod_p \mathbb{Z}_p^{\times}\right) \left(\mathbb{Q}^{\times}_{\mathbb{A}_f}\right)^2.
\end{equation*}
To see this, we set  $\widetilde{D} := \tau(J)\subseteq D^*$ with $J$ as defined in \cite[page 487]{shimura2006quadratic}. Then, we have that $G^*_{\mathbb{A}} = G^*_{\mathbb{Q}}\widetilde{D}$. 
This follows from \cite[Lemma 9.21 (iv) and Lemma 9.22 (iii)]{shimura2004arithmetic}, 
and the fact that $\nu(J_p) = \mathbb{Z}_p^{\times}$ for every finite prime $p$ (\cite[page 496]{shimura2006quadratic}). Therefore, there are $h_{\mathfrak{m}} \in G_{\mathbb{Q}}^{*}$ and $d \in \widetilde{D}$ such that $\textup{diag}(1, g_{\mf{m}}, 1)^{-1} = h_{\mf{m}}^{-1}d$, i.e. $h_{\mathfrak{m}} = \textup{diag}(1, g_{\mf{m}}, 1)d$, with the properties of $\sigma(d)_{f}$ as claimed.\\

In particular, $\sigma(h_{\mf{m}}) = \sigma\left(\textup{diag}(1, g_{\mf{m}},1) \right) \sigma(d)$. From \cite[Lemma 4.9]{shimura2004arithmetic}, we have that $\sigma\left(\textup{diag}(1, g_{\mf{m}},1) \right) = \sigma(g_{\mf{m}})$, and we will use the isometry $\theta: V \rightarrow K$ to calculate the spinor norm of $g_{\mf{m}}$. We write $\widetilde{\theta}: \textup{SO}^{\phi}(V) \rightarrow K^1$ for the isomorphism of the orthogonal groups, where $K^1$ denotes the elements of $K$ of norm $1$.\\

Recall $\theta(L_{\mf{m}}) = \overline{\mf{m}}^2/N(\mf{m}) = \overline{\mf{m}}/\mf{m}$. Now, for each prime $p$, the fractional ideal $\mf{m}_p$ is principal in $K_p$, say $\mf{m}_p = (\alpha_p)$ with some $\alpha_p \in K_p$ (note that $K_p$ may or may not be a field). We then have $\widetilde{\theta}(g_{\mf{m},p}) = \overline{\alpha}_p/\alpha_p$ in $K_p^{1}$. From \cite[Section 7.2]{shimura2004arithmetic}, we have that the element $\alpha_p$ is a lift of $\overline{\alpha}_p/\alpha_p$ in the corresponding Clifford group under $\tau$, and satisfies $\nu(\alpha_p) = N_{K_{p}/\mathbb{Q}_p}(\alpha_p) = N_{K_{p}/\mathbb{Q}_p}(\mf{m}_p)$. By putting these together, we get $\sigma(g_{\mf{m}})_{f} = N(\mf{m}) (\mathbb{Q}_{\mathbb{A},f}^{\times})^2$.\\

Going back to the equation 
\[
\sigma(h_{\mf{m}}) = \sigma\left(\textup{diag}(1, g_{\mf{m}},1) \right) \sigma(d),
\]
we now have that $\sigma(h_{\mf{m}}) = N(\mf{m})_f \sigma(d) (\mathbb{Q}^{\times}_{\mathbb{A},f})^2$. Since $\sigma(h_{\mf{m}}) \in \mathbb{Q}^{\times}$, and the selection of $d$, we may conclude that for all finite primes we have that $v_p(\sigma(h_{\mf{m}})) = v_p(N(\mf{m})) \pmod{2}$, where $v_p$ is the standard valuation associated to $p$. In particular, we have that $\sigma(h_{\mf{m}}) = \pm N(\mf{m}) \,(\mathbb{Q}^{\times})^2$. To conclude the statement, it is enough to determine the sign of $\sigma(h_{\mf{m}})$. But the sign of $\sigma(h_{\mf{m}})$ is $\pm 1$ according to whether $h_{\mf{m}}$ is in $G^{*, +}_{\mathbb{Q}}$ or not (see \cite[Theorem 14.2 (ii)]{shimura2004arithmetic}. In particular, since $\sigma\left(\textup{diag}(-1,1,1,-1) \right) = -1 (\mathbb{Q}^{\times})^2$, we can always multiply with such an element to have $h_{\mf{m}} \in G_{\mathbb{R}}^{*, +}$, which will then given $\sigma(h_{\mf{m}})= N(\mf{m}) (\mathbb{Q}^{\times})^{2}$, as required.
\end{proof}
Given $h_{\mathfrak{m}}$, we define $\widetilde{h}_{\mathfrak{m}} := \textup{diag}(1, h_{\mathfrak{m}}, 1) \in G_{\mathbb{Q}}$. We now have the following couple of Lemmas.
\begin{lemma}\label{conjugation}
    Let $g \in G_{\mathbb{R}}^+$ and $h \in G_{\mathbb{R}}$. Then $hg h^{-1} \in G_{\mathbb{R}}^+$. Similarly, the same holds for $G^*_{\mathbb{R}}$ and $G_{\mathbb{R}}^{*,+}$.
\end{lemma}
\begin{proof}
We use the fact that $G_{\mathbb{R}}^+$ (resp. $G_{\mathbb{R}}^{*,+}$) is of index two in $G_{\mathbb{R}}$ (resp. $G_{\mathbb{R}}^{*}$) (see, for example, \cite[p. 140]{shimura2004arithmetic}), which holds for indefinite quadratic forms. This then shows that $G_{\mathbb{R}}^{+}$ (resp. $G_{\mathbb{R}}^{*,+}$) is a normal subgroup in $G_{\mathbb{R}}$ (resp. $G_{\mathbb{R}}^{*}$), from which the result follows.

\end{proof}

\begin{lemma} \label{isomorphisms} The map $g \longmapsto h_{\mathfrak{m}}g h_{\mathfrak{m}}^{-1}$ induces the isomorphisms
\begin{equation*}
    \Gamma(L_0) \cong \Gamma(L_{\mathfrak{m},0}), \,\,\,\ \Gamma^+(L_0) \cong \Gamma^+(L_{\mathfrak{m},0}), \,\,\,\ \widehat{\Gamma}(L_0) \cong \widehat{\Gamma}(L_{\mathfrak{m},0}).
\end{equation*}
Similarly, the map $g \longmapsto \widetilde{h}_{\mathfrak{m}}g \widetilde{h}_{\mathfrak{m}}^{-1}$ induces the isomorphisms  
\begin{equation*}
    \Gamma(L_1) \cong \Gamma(L_{\mathfrak{m},1}), \,\,\,\, \Gamma^+(L_1) \cong \Gamma^+(L_{\mathfrak{m},1}), \,\,\,\,\widehat{\Gamma}(L_1) \cong \widehat{\Gamma}(L_{\mathfrak{m},1}).
\end{equation*}
\end{lemma}
\begin{proof}
The first isomorphism $\Gamma(L_0) \cong \Gamma(L_{\mathfrak{m},0})$ is clear, since $h_{\mathfrak{m}}$ of Lemma \ref{hm}, satisfies $L_{\mathfrak{m}} = h_{\mathfrak{m}}L_0$. The second isomorphism $\Gamma^+(L_0) \cong \Gamma^+(L_{\mathfrak{m},0})$ follows from Lemma \ref{conjugation}. Finally, from \cite[Section 6.1]{shimura2004arithmetic}, we have $L^*_{\mathfrak{m},0} = h_{\mathfrak{m}} L^*_0$. Hence, if $g \in \widehat{\Gamma}(L_0)$, then
\begin{equation*}
    (h_\mathfrak{m}gh_\mathfrak{m}^{-1}-1)L_{\mathfrak{m},0}^{*} = h_\mathfrak{m}gL_{0}^{*}-h_{\mathfrak{m}}L_0^{*} = h_\mathfrak{m}(g-1)L_0^{*} \subset h_\mathfrak{m}L_0 = L_{\mathfrak{m},0}
\end{equation*}
and therefore the same map induces an isomorphism $\widehat{\Gamma}(L_0) \cong \widehat{\Gamma}(L_{\mathfrak{m},0})$. Similarly, we obtain the rest isomorphisms, since $\widetilde{h}_{\mathfrak{m}}$ satisfies $L_{\mathfrak{m}, 1} = \widetilde{h}_{\mathfrak{m}}L_1$. \qedhere

\end{proof}

We take now $\bm{F} \in \mathfrak{S}_{k}(D_{f})$. For each fractional ideal $\mathfrak{m}$, we define $F_{\mathfrak{m}}(Z) := \bm{F}(\widetilde{h}_{\mathfrak{m}}; Z)$ using \eqref{adelic F}. We then have that $F_{\mathfrak{m}}(\gamma \langle Z\rangle) = j(\gamma, Z)^{k}F_{\mathfrak{m}}(Z)$ for all $\gamma \in \Gamma(\widetilde{h}_{\mathfrak{m}}) = G_{\mathbb{Q}} \cap (G_{\mathbb{R}}^{+} \times \widetilde{h}_{\mathfrak{m}} D \widetilde{h}_{\mathfrak{m}}^{-1})$. We note that $\Gamma(\widetilde{h}_{\mathfrak{m}}) = \{\gamma \in  G_{\mathbb{Q}} \cap G_{\mathbb{R}}^{+} \mid \gamma L_{\mathfrak{m},1} = L_{\mathfrak{m},1}\}=\Gamma^{+}(L_{\mathfrak{m},1})$. In particular, we have that  $F_{\mathfrak{m}}$ is an orthogonal modular form for $\Gamma^{+}(L_{\mathfrak{m},1})$ in the sense of Definition \ref{modular_form_defn} (i.e., with $\Gamma$ equal to $\Gamma^{+}(L_{\mathfrak{m},1})$). When $\mathfrak{m} = \mathcal{O}_K$, we will write $F := F_{\mathcal{O}_K}$, which is invariant under $\Gamma^{+}(L_1)$.

\begin{lemma} With notation as above, we have that $L_{\mathfrak{m},0} = L(\widetilde{h}_m)$, where $L(\tilde{h}_m)$ is the lattice defined in \eqref{lattice-sugano}.

\end{lemma}

\begin{proof}
The element $\widetilde{h}_{\mathfrak{m}}$ may be written as
$\widetilde{h}_{\mathfrak{m}} = \textup{diag}(1_2, g_{\mathfrak{m}}, 1_2)k$, where $g_{\mathfrak{m}} \in \textup{SO}^{\phi}(\mathbb{A})$ and $k \in D$.\\
Indeed, since $L_{\mathfrak{m},1} = \widetilde{h}_{\mathfrak{m}}L_1$ and $L_{\mathfrak{m,1}} = \textup{diag}(1_2, g_{\mathfrak{m}}, 1_2) L_1$, we have that $\textup{diag}(1_2, g_{\mathfrak{m}}, 1_2)^{-1}\widetilde{h}_{\mathfrak{m}} \in D$.\\

We first show that $L(\widetilde{h}_\mathfrak{m}) \subseteq L_{\mathfrak{m},0}$. We consider $X \in L\left(\widetilde{h}_\mathfrak{m}\right)$. From the definition of the group $\Gamma\left(\widetilde{h}_{\mathfrak{m}}\right)$ we see
\[
\Gamma\left(\widetilde{h}_{\mathfrak{m}}\right) = \textup{SO}^{\phi_1}(V_1) \cap \widetilde{h}_{\mathfrak{m}} D \widetilde{h}_{\mathfrak{m}}^{-1}= \textup{SO}^{\phi_1}(V_1) \cap \m{1_2 & 0 & 0 \\ 0 & g_{\mathfrak{m}} & 0 \\ 0 & 0 & 1_2} D \m{1_2 & 0 & 0 \\ 0 & g_{\mathfrak{m}} & 0 \\ 0 & 0 & 1_2}^{-1}.
\]
We now calculate:
\begin{equation}\label{calculations}
\m{1_2 & 0 & 0 \\ 0 & g_{\mathfrak{m}} & 0 \\ 0 & 0 & 1_2}^{-1} \gamma_X \m{1_2 & 0 & 0 \\ 0 & g_{\mathfrak{m}} & 0 \\ 0 & 0 & 1_2} = \m{ 1 & -X^t S_0 g_{\mathfrak{m},0} & -\frac{1}{2}S_0[X] \\ 0 & 1_4 & g_{\mathfrak{m},0}^{-1}X \\ 0 & 0 & 1},
\end{equation}
where we have set $g_{\mathfrak{m},0} := \m{1 & 0 & 0 \\ 0 & g_{\mathfrak{m}} & 0 \\ 0 & 0 & 1 } \in \textup{SO}^{\phi_0}(\mathbb{A})$. Moreover we may write $-X^t S_0 g_{\mathfrak{m},0} = -(g_{\mathfrak{m},0}^{-1}X)^t S_0$ since $g_{\mathfrak{m},0}^t S_0 g_{\mathfrak{m},0} = S_0$. Now we check the condition
\[
\m{ 1 & (g_{\mathfrak{m},0}^{-1}X)^t S_0 & -\frac{1}{2}S_0[g_{\mathfrak{m},0}^{-1}X] \\ 0 & 1_4 & g_{\mathfrak{m},0}^{-1}X \\ 0 & 0 & 1} \in D.
\]
That is, we require that $g_{\mathfrak{m},0}^{-1}X \in \prod_{p<\infty} L_{0,p}$, and so $X \in g_{\mathfrak{m},0}\prod_{p < \infty} L_{0,p}$. But $X$ is a global element and hence we have that $X \in \mathbb{Q}^4 \cap g_{\mathfrak{m},0}\prod_{p < \infty} L_{0,p} = L_{\mathfrak{m},0}$.\newline

We now show the other inclusion, namely $L_{\mathfrak{m},0} \subseteq L(\widetilde{h}_\mathfrak{m})$. We consider an $X \in L_{\mathfrak{m},0}$. In particular, since $L_{\mathfrak{m},0} = h_{\mathfrak{m}}L_0$, we may write $X = h_{\mathfrak{m}} v$ for $v \in L_0$. From the calculations \eqref{calculations}, it is enough to check whether 
\[
\m{ 1 & (g_{\mathfrak{m},0}^{-1}X)^t S_0 & -\frac{1}{2}S_0[g_{\mathfrak{m},0}^{-1}X] \\ 0 & 1_4 & g_{\mathfrak{m},0}^{-1}X \\ 0 & 0 & 1} \in D
\]
But $g_{\mathfrak{m},0}^{-1}X \in g_{\mathfrak{m},0}^{-1} h_{\mathfrak{m}} L_0 = L_0$, and hence all the upper triangular entries are in $\mathbb{Z}$. This shows that the matrix is in $D$.
\end{proof}
In particular, the Fourier expansion \eqref{fourier-expansion-classical-orthogonal} of the modular forms $F_{\mathfrak{m}}$ can be written over the dual of the lattice $L_{\mathfrak{m},0}$. That is,
\[
F_\mathfrak{m}(Z) = \sum_{\eta \in L_{\mathfrak{m},0}^{*}} a(\widetilde{h}_{\mathfrak{m}};\eta) e(S_0(\eta,Z)).
\]

In the case when $\mf{m} = \mathcal{O}_K$, we will write $A(\eta) := a(id; \eta)$ for all $\eta \in L_0^{*}$ for the Fourier coefficients of $F = F_{\mathcal{O}_K}$. We are now able to relate the Fourier coefficients of the various $F_{\mathfrak{m}}$ with those of $F$, using the following Lemma.

\begin{lemma}\label{fourier-coeffs-different-m}
    For the Fourier coefficients of $F_{\mathfrak{m}}$, we have
    \begin{equation*}
        a(\widetilde{h}_{\mathfrak{m}};\eta) = a\left(id; h_{\mathfrak{m}}^{-1}\eta\right) = A(h^{-1}_{\mf{m}}\eta),
    \end{equation*}
\end{lemma}
where $id$ denotes the identity element in $G_{\mathbb{A},f}$.
\begin{proof}
    
    Using the properties in \cite[eq. (1.14)] {sugano} on the adelic Fourier coefficients of $\bm{F}$ (see \eqref{adelic-fourier-coefficient}), we obtain
    \begin{equation*}
        \bm{F}_{\chi}(\widetilde{h}_{\mathfrak{m}}; \eta)  = \bm{F}_{\chi}(id; h_{\mathfrak{m}}^{-1}\eta) = a(id; h_{\mathfrak{m}}^{-1}\eta)e(\eta^t h_{\mathfrak{m}}^{-t} S_0 Z_0).
    \end{equation*}
    But from \eqref{property of F_x}, we also have
    \begin{multline*}
        \bm{F}_{\chi}(\widetilde{h}_{\mf{m}};\eta) = a(\widetilde{h}_{\mathfrak{m}}; \eta)j(\widetilde{h}_{\mathfrak{m}}, Z_0)^{-k}e(\eta^{t}S_0\widetilde{h}_{\mathfrak{m}}\langle Z_0\rangle) = a(\widetilde{h}_{\mathfrak{m}}; \eta)e(\eta^tS_0h_{\mathfrak{m}} Z_0) =\\= a(\widetilde{h}_{\mathfrak{m}};\eta)e(\eta^{t}h_{\mathfrak{m}}^{-t}S_0 Z_0),
    \end{multline*}
    using the fact that $h_{\mathfrak{m}}^{t}S_0 h_{\mathfrak{m}} = S_0$. By comparing these two equations, the Lemma follows.
\end{proof}
We now have the following Lemma about the dual lattice of $L_{\mathfrak{m},0}$. 

\begin{lemma}\label{dual-Lm}

   With notation as above, we have that
    \begin{equation*}
        L_{\mathfrak{m},0}^{*} = \mathbb{Z}e_2 + L_{\mathfrak{m}}^{*} + \mathbb{Z}f_{2},
    \end{equation*}
    where $L_{\mathfrak{m}}^{*}$ is the dual lattice of $L_\mathfrak{m}$ with respect to $\phi$. 
\end{lemma}
\begin{proof}
    We will show the result by showing both inclusions. 
    \begin{itemize}
        \item $L_{\mathfrak{m},0}^{*} \supseteq \mathbb{Z}e_2 + L_{\mathfrak{m}}^{*} + \mathbb{Z}f_{2}$. If $v = (a, b, c)^{t}$ with $a, c \in \mathbb{Z}$ and $b \in L_{\mathfrak{m}}^{*}$, we can compute for any $v' = (a',b',c') \in L_{\mathfrak{m}}$, with $a',c' \in \mathbb{Z}$, $b' \in L_{\mathfrak{m}}$,
        \begin{equation*}
            vS_0v' = a'c + c'a - b^{t}Sb' \in \mathbb{Z},
        \end{equation*}
        because $a,a',c,c' \in \mathbb{Z}$ and $b^t S b \in \mathbb{Z}$. Hence $v \in L_{\mathfrak{m},0}^{*}$.
        \item $L_{\mathfrak{m},0}^{*} \subseteq \mathbb{Z}e_2 + L_{\mathfrak{m}}^{*} + \mathbb{Z}f_{2}$. Suppose now $v = (a, b, c)^{t} \in L_{\mathfrak{m},0}^{*}$ with $a, c \in \mathbb{Q}$ and $b \in V$. Then $v^tS_0 y \in \mathbb{Z}$ for all $y \in L_{\mathfrak{m},0}$ implies that if we choose $y=e_2$ and $y =f_{2}$, then $a ,c \in \mathbb{Z}$. Moreover, if $y = (0, z, 0)^{t}$ with $z \in L_{\mathfrak{m}}$, we have $v^t S_0 y = b^t Sz \in \mathbb{Z}$ for all $z \in L_{\mathfrak{m}}$. This implies $b \in L_{\mathfrak{m}}^{*}$, and hence the result follows. \qedhere
    \end{itemize}
\end{proof}

From the lemma above, for any $\eta \in L^*_{\mathfrak{m},0}$, we may write $\eta = (N,\lambda,n)$ with $N,n \in \mathbb{Z}$ and $\lambda \in L^*_{\mathfrak{m}}$. Then we have
\[
F_\mathfrak{m}(Z) = \sum_{N=1}^{\infty}\phi_{\mathfrak{m},N}(z,\tau) e(N \tau'),
\]
where 
\[
\phi_{\mathfrak{m},N}(z,\tau)=\sum_{\substack{\lambda \in L^*_{\mathfrak{m}},\  n \in \mathbb{Z},\\ (N,\lambda,n) \in L^*_{\mathfrak{m},0}}} a(\widetilde{h}_{\mathfrak{m}}; (N,\lambda,n)) e(n \tau + \lambda^t S z).
\]
Now, since $F_{\mf{m}} \in S_k(\Gamma(L_{\mf{m},1})) \leq S_{k}(\widehat{\Gamma}(L_{\mf{m},1}))$, it can be shown that for each $N \geq 1$, $\phi_{\mf{m}, N}$ is a Jacobi cusp form of weight $k$ (as of $F_{\mathfrak{m}})$ and index $(L_{\mathfrak{m}}, N\phi)$ in the sense of \cite[Definition 1.23, 1.26]{jacobi_lattice}. We denote this space by $S_k(L_{\mathfrak{m}}, N\phi)$. We now recall that we have defined $A_\mathfrak{m} \in \textup{SL}_2(\mathbb{Q})$ such that $A_{\mathfrak{m}} L_{\mathcal{O}_K} = L_\mathfrak{m}$, and we have set $S_{\mathfrak{m}} = A_{\mathfrak{m}}^{t} S A_{\mathfrak{m}}$. With this notation we obtain that $L^*_{\mathfrak{m}} = S^{-1}A_{\mathfrak{m}}^{-t}\mathbb{Z}^{2}$.  Using this, we obtain,
\begin{multline}\label{expansion-of-jacobi}
\phi_{\mathfrak{m},N}(z,\tau) = \sum_{\lambda \in \mathbb{Z}^2, n \in \mathbb{Z}}a(\widetilde{h}_{\mathfrak{m}}; (N, S^{-1}A_{\mathfrak{m}}^{-t}\lambda,n))e(n \tau + \lambda^t A^{-1}_{\mathfrak{m}} z) =\\ =\sum_{\lambda \in \mathbb{Z}^2, n \in \mathbb{Z}}a(\widetilde{h}_{\mathfrak{m}}; (N, A_{\mathfrak{m}}S_{\mathfrak{m}}^{-1}\lambda,n))e(n \tau + \lambda^t A^{-1}_{\mathfrak{m}} z).
\end{multline}

\section{Poincar\'e Series}
We will now define a special class of Fourier--Jacobi forms, called Poincar\'e series, associated to the lattices $L_{\mathfrak{m}}$ above. To begin with, the support of the lattice $L_{\mathfrak{m}}$ in $V$ is
    \begin{equation*}
        \textup{supp}(L_{\mathfrak{m}}, \phi) := \left\{(D_{\mathfrak{m}},r_{\mathfrak{m}}) \mid D_{\mathfrak{m}} \in \mathbb{Q}_{\leq 0}, r_{\mathfrak{m}} \in L_{\mathfrak{m}}^{*}, D_{\mathfrak{m}} \equiv \frac{1}{2}S[r_{\mathfrak{m}}] \pmod {\mathbb{Z}}\right\}.
    \end{equation*}
Now, since $L_{\mathfrak{m}}^{*} = S^{-1}A_{\mathfrak{m}}^{-t}\mathbb{Z}^2$, we write $r_{\mathfrak{m}} \longmapsto S^{-1}A_{\mathfrak{m}}^{-t}r_{\mathfrak{m}}$, with $r_{\mathfrak{m}} \in \mathbb{Z}^2$. Hence, the congruence becomes
\begin{equation*}
    q_{\mathfrak{m}}D_{\mathfrak{m}} \equiv \frac{1}{2}q_{\mathfrak{m}}S_{\mathfrak{m}}^{-1}[r_{\mathfrak{m}}] \pmod {q_{\mathfrak{m}}\mathbb{Z}},
\end{equation*}
where $q_{\mathfrak{m}}$ is the level of the lattice $L_{\mathfrak{m}}$. This then implies $D_{\mathfrak{m}} \in \frac{1}{q_{\mathfrak{m}}}\mathbb{Z}$. Hence, by writing $D_{\mathfrak{m}} \longmapsto D_{\mathfrak{m}}/q_{\mathfrak{m}}$, we get that equivalently 
\begin{equation*}
    \textup{supp}(L, \phi) = \left\{(D_{\mathfrak{m}}/q_{\mathfrak{m}},S^{-1}A_{\mathfrak{m}}^{-t}r_{\mathfrak{m}}) \mid D_{\mathfrak{m}} \in \mathbb{Z}_{\leq 0}, r_{\mathfrak{m}} \in \mathbb{Z}^2, D_{\mathfrak{m}} \equiv \frac{1}{2}q_{\mathfrak{m}}S_{\mathfrak{m}}^{-1}[r_{\mathfrak{m}}] \pmod {q_{\mathfrak{m}}}\right\}.
\end{equation*}
Let then
\begin{equation}\label{support poincare}
    \widetilde{\textup{supp}}(L_{\mathfrak{m}}, \phi) :=  \{(D_{\mathfrak{m}},r_{\mathfrak{m}}) \in \mathbb{Z}_{\leq 0}\times \mathbb{Z}^2 \mid (D_{\mathfrak{m}}/q_{\mathfrak{m}}, S^{-1}A_{\mathfrak{m}}^{-t}r_{\mathfrak{m}}) \in \textup{supp}(L_{\mathfrak{m}}, \phi)\}.
\end{equation}
and in the following, this is the set we will use.\\

If now $(D_{\mf{m}},r_{\mf{m}}) \in  \widetilde{\textup{supp}}(L_{\mf{m}}, \phi)$, we define the Poincar\'e series $P(k,D_{\mf{m}},r_{\mf{m}})$ as in \cite[Definition 2.2]{jacobi_lattice}.  If $k > n/2+2$, $P(k,D_{\mf{m}},r_{\mf{m}})$ is absolutely and uniformly convergent on compact subsets of $\mathbb{H} \times (L_{\mf{m}} \otimes \mathbb{C})$ ($\mathbb{H}$ the usual upper half plane) and defines an element of $S_k(L_{\mf{m}}, \phi)$ (see \cite[Theorem 2.3, (i)]{jacobi_lattice}).\\

For $N \geq 1$, we denote by $V_N: S_k(L_{\mf{m}}, \phi) \longrightarrow S_k(L_{\mf{m}}, N\phi)$ the index-raising operator defined in

\cite[Definition 4.25]{jacobi_lattice}. In particular, for $(D_{\mf{m}},r_{\mf{m}}) \in \widetilde{\textup{supp}}(L_{\mf{m}}, \phi)$, we define
\begin{equation}\label{poincare in maass}
    \mathcal{P}_{\mf{m}}(\tau',z,\tau):=\mathcal{P}(k,D_{\mf{m}},r_{\mf{m}})(\tau', z,\tau) := \sum_{N \geq 1}(V_{N}P(k,D_{\mf{m}},r_{\mf{m}}))(\tau,z)e(N\tau').
\end{equation}
Here, by abusing notation, we write $P(k,D_{\mf{m}},r_{\mf{m}})$ to actually denote the Poincar\'e series $P(k,D_{\mf{m}},r_{\mf{m}})/\lambda_{k,D_{\mf{m}}}$, where $\lambda_{k,D}$ is defined in \cite[Theorem 2.3, (i)]{jacobi_lattice}.\\

Our first aim is to show that $\mathcal{P}_{\mf{m}} \in S_{k}(\widehat{\Gamma}(L_{\mf{m},1}))$. This is done in \cite{sugano_lifting} (see also \cite{krieg_jacobi}) for the lattice $L_1$ and here we simply explain how these ideas can be adapted to extend the lifting to all lattices $L_{\mf{m},1}$. We will first give a Lemma regarding some useful properties of primitive vectors (see Section \ref{preliminaries}).

\begin{lemma}\label{primitive} Let $\Lambda$ denote a $\mathbb{Z}$-lattice in some vector space $V$. The following properties hold:
\begin{enumerate}
\item Let $w \in \Lambda$. Then there exists a primitive $v \in \Lambda$ such that $w = l v$ with $l \in \mathbb{Z}$.   
\item Let $g \in \textup{GL}(V)$, $\Lambda_1$ and $
\Lambda_2$ lattices in $V$ such that $g \Lambda_1 =\Lambda_2$, and $v$ primitive in $\Lambda_1$. Then $gv$ is a primitive element in $\Lambda_2$.
\end{enumerate}
\end{lemma}
\begin{proof}
Let us first show (1). If $w$ is primitive, there is nothing to prove. Let us assume that this is not the case. We let $\mathfrak{a}_w := \{ k \in \mathbb{Q} \,\,\,| \,\,\, kw \in \Lambda\}$. This is a fractional ideal in $\mathbb{Q}$, and hence it can be written as $\mathfrak{a}_w = m \mathbb{Z}$ for some $m \in \mathbb{Q}$ with $m \neq \pm 1$. We write $m = \frac{m_1}{m_2}$ with $m_1,m_2 \in \mathbb{Z}$ and $gcd(m_1,m_2) = 1$. Since clealry, $1 \in \mathfrak{a}_{w}$ we have that $m_1 = 1$. That is $m = \frac{1}{m_2}$. We then claim that $v:= m_2^{-1} w \in \Lambda$ is a primitive element. Indeed, if we assume otherwise, then there exists $u \in \Lambda$ such that $ku = v$ for some integer $k \neq \pm 1$. That is, $k u = m_2^{-1} w$, and so $k^{-1}m^{-1}_2 \in \mathfrak{a}_w$. This is a contradiction, as $\mathfrak{a}_{w}$ is generated by $m_2^{-1}$. Hence, $v$ is primitive and clearly $w = m_2 v$. \newline

For (2), if we assume that $gv \in \Lambda_2$ is not primitive, then there exists a $w \in \Lambda_2$ such that $kw = gv$ for some $k \neq \pm1$. But then $k g^{-1} w = v$ and $g^{-1}w \in \Lambda_1$, contradicting the primitivity of $v$.
\end{proof}
We now define the integral Jacobi group of with respect to $L_{\mf{m}}$ to be $J_{S,\mf{m}} := \textup{SL}_2(\mathbb{Z}) \ltimes (L_{\mf{m}} \times L_{\mf{m}})$. We can embed this in $\widehat{\Gamma}(L_{\mf{m},1})$ as follows:
\begin{equation*}
    g = \m{a & b \\ c& d} \in \textup{SL}_2(\mathbb{Z}) \longmapsto \iota(g):=  \m{a & -b \\-c & d \\ &&1_{2}\\ &&&a&b\\ &&&c&d} \in \widehat{\Gamma}(L_{\mf{m},1}).
\end{equation*}
\begin{equation*}
    [\xi, \eta] \in L_{\mf{m}} \times L_{\mf{m}} \longmapsto \m{1 & 0 & \eta^{t}S& \xi^{t}S\eta & S[\eta]/2\\0&1 & \xi^t S&S[\xi]/2 & 0\\&&1_2&\xi & \eta\\ &&&1&0\\&&&0&1} \in \widehat{\Gamma}(L_{\mf{m},1}).
\end{equation*}
Let also 
\begin{equation*}
    M := \m{&&&&-1\\&&&-1\\&&1_2\\&-1\\-1}.
\end{equation*}
We then have the following Lemma.
\begin{lemma}\label{generators-lattice-m}
    $\widehat{\Gamma}(L_{\mf{m},1})$ is generated by $J_{S, \mf{m}}$, $M$ and $\widehat{\Gamma}(L_{\mf{m},0})$.
\end{lemma}
\begin{proof}
    Let $\Delta$ denote the subgroup of $\widehat{\Gamma}(L_{\mf{m},1})$ generated by the above elements. We first claim that if $\xi$ is a primitive element in $L_{\mf{m},1}^{*}$, then there exists $\gamma \in \Delta$ such that $\gamma \xi = (0, \mu,0)^{t}$, with $\mu = (a, \alpha, 1)^{t} \in L_{\mf{m},0}^{*}$. This is the analogous statement of \cite[Lemma 6.4]{sugano_lifting}. By writing $\xi = (x, u, v^t, w, z)^{t}$, such that $x, u, w, z \in \mathbb{Z}$ and $v \in L_{\mf{m}}^{*}$, we can follow Sugano's proof and obtain that there is $\delta \in \Delta$, so that if $\delta \xi = (x', u', (v')^{t}, w', z')$, then $x'=z'=0$. Now, since $\xi$ is primitive in $L_{\mf{m},1}^{*}$, we also get that the vector $\zeta := (u', (v')^{t}, w')^{t}$ is primitive in $L_{\mf{m},0}^{*}$.\\ 
    
    We now claim that that there is an $X \in L_{\mf{m},0}$ such that $2\phi_0(\zeta, X) = 1$. Indeed, if we assume otherwise, then the ideal of $\mathbb{Z}$ generated by $2\phi_0(\zeta,Y)$ with $Y \in L_{\mf{m},0}$ is not equal to $\mathbb{Z}$. Let us write $\beta \neq \pm 1$ for a generator of this ideal. That is, for all elements $Y \in L_{\mf{m},0}$ we have $2\phi_0(\zeta, Y) \in \beta \mathbb{Z}$, or equivalently $2\phi_0(\beta^{-1}\zeta,Y) \in \mathbb{Z}$ for all $Y \in L_{\mf{m},0}$. Hence $\beta^{-1} \zeta \in L^*_{\mf{m},0}$, which contradicts the fact that $\zeta$ is primitive in $L^*_{\mf{m},0}$.  Since now $X \in L_{\mf{m},0}$, we can finish the proof of the claim as in Sugano.\\

    Let now $\gamma \in \widehat{\Gamma}({L_{\mf{m},1})}$ and let $\xi$ denotes its first column. Then, $\xi$ will be primitive in $L_{\mf{m},1}^{*}$. Indeed, since the vector $e_1 = (1,0,0,0,0,0)^t \in L^*_{\mf{m},1}$, we have that $\xi e_1$ is also primitive in $L^*_{\mf{m},1}$. This is nothing else than the first column of $\xi$. By the claim we established above, there is $\delta \in \Delta$ such that 
    $\delta \xi = (0 ,\mu, 0)^{t}$ with $\mu = (a, \alpha, 1)^{t} \in L_{\mf{m},0}^{*}$. But since $\xi \in L_{\mf{m},1}$, we have that $\mu \in L_{\mf{m},0}$. Now, if $R_{K_{\lambda}}$ is defined as in \cite[Proposition 1, (e)]{krieg_integral_orthogonal}, we have that if $\lambda = -\alpha$, then $R_{K_{-\alpha}} \in \Delta$ and $R_{K_{-\alpha}} \mu = (0, 0, 1)^{t}$ since $S_1[\mu]=0$. By then multiplying by $\iota\left(\m{0&1\\-1&0}\right)M$, we obtain that $\iota\left(\m{0&1\\-1&0}\right)MR_{K_{-\alpha}}\delta \gamma \in \Gamma_{\infty}$, where $\Gamma_{\infty}$ denotes the parabolic fixing the standard zero-dimensional cusp. But now, similar to \cite[Proposition 3]{krieg_integral_orthogonal}, we have that any element of $\Gamma_{\infty}$ can be written in the form $\pm R_{K}M_{c}$, where $R_K$ and $M_c$ are defined as in \cite[Proposition 1]{krieg_integral_orthogonal}. In particular, in our setting, we have $K \in \widehat{\Gamma}(L_{\mf{m},0})$ and $c \in L_{\mf{m},0}$, and therefore both of these elements are in  $\Delta$. Hence, $\gamma \in \Delta$, from which the Lemma follows.
\end{proof}
Using the above Lemma, we can then show that $\mathcal{P}_{\mf{m}} \in S_k(\widehat{\Gamma}(L_{\mf{m},1}))$.
\begin{proposition}
    $\mathcal{P}_m$ is a modular form of weight $k$ with respect to $\widehat{\Gamma}(L_{\mf{m},1})$.
\end{proposition}
\begin{proof} We recall from \eqref{poincare in maass},
 \[  \mathcal{P}_{\mf{m}}(\tau',z,\tau)= \sum_{N \geq 1}(V_{N}P_{\mf{m}})(\tau,z)e(N\tau').
\]
By Lemma \ref{generators-lattice-m}, to show that $\mathcal{P}_{\mf{m}}$ is a modular form of weight $k$ with respect to $\widehat{\Gamma}(L_{\mf{m},1})$ it is enough to check the modularity property with respect to $M$ and the elements of $J_{S, \mf{m}}$ and $\widehat{\Gamma}(L_{\mf{m},0})$. But this follows exactly as in \cite[Theorem 6.2]{sugano_lifting}.
\end{proof}

We also note that since $\mathcal{P}_m$ arises as a lift of $P_{\mathfrak{m}}$ via the operator $V_N$, we have $\mathcal{P}_{\mf{m}}$ will be in the corresponding Maass space, i.e. it will satisfy the conditions of \cite[eq. (5.12), (5.13)]{sugano_lifting} for the lattice $L_{\mf{m},1}$.\\

Our next aim is to show that when we make a compatible choice of the $D_{\mf{m}}$, namely $D_{\mf{m}} = -\ell q_{\mf{m}}$ for some positive integer $\ell$, then the  $\mathcal{P}_{\mf{m}}$'s, are related, in the sense that they are arise from the same adelic automorphic form $\bm{P}$, in the way described just after Definition \ref{modular_form_defn}. We do this by using the explicit knowledge of the Fourier coefficients of the $P_{\mf{m}}$ and the way we construct $\mathcal{P}_m$ using the $V_N$ operator. In particular, this choice of $D_{\mf{m}}$ forces $r_{\mf{m}} \in L_{\mf{m}}$, and therefore, we can take $r_{\mf{m}}=0$, because of \cite[Theorem 2.3]{jacobi_lattice}. \\

For each $P_{\mf{m}}$ as above, we write its Fourier expansion as
\begin{equation*}
    P_{\mf{m}}(\tau, z) = \sum_{a \in \mathbb{Z}, \alpha \in L_{\mf{m}}^{*}}c_{\mf{m}}(a, \alpha) e(a\tau + \alpha^t S z).
\end{equation*}

It is well known (see for example \cite[Prop 1.25]{jacobi_lattice}) that $c_{\mf{m}}(a,\alpha) = c_{\mf{m}}(a',\alpha')$ if $\alpha \equiv \alpha' \pmod{L_{\mf{m}}}$ and $a-\phi[\alpha] = a'-\phi[\alpha']$. In particular, making the usual change of variables and setting $C_{\mf{m}}(\widetilde{D},\alpha):= c_{\mf{m}}(\phi[\alpha]-\widetilde{D},\alpha)$, we have that
\[
P_{\mf{m}}(\tau, z) = \sum_{(\widetilde{D},\alpha) \in \textup{supp}(L_{\mf{m}},\phi)}C_{\mf{m}}(\widetilde{D}, \alpha) e((\phi[\alpha]-\widetilde{D})\tau + \alpha^t S z),
\]
and $C_{\mf{m}}(\widetilde{D}_1, \alpha_1)=C_{\mf{m}}(\widetilde{D}_2, \alpha_2)$ if $\widetilde{D}_1 = \widetilde{D_2}$ and $\alpha_1 \equiv \alpha_2 \pmod{L_{\mf{m}}}$.\\

In \cite[Theorem 2.3]{jacobi_lattice} the coefficients $C_{\mf{m}}(\widetilde{D},\alpha)$ are determined.

\begin{theorem}\label{Fourier Coefficients of Poincare}
The Fourier coefficients $C_{\mf{m}}(\widetilde{D},\alpha)$ are given as (recall the weight $k$ is even)
\begin{multline*}
\lambda_{k,D_{\mf{m}}}C_{\mf{m}}(\widetilde{D},\alpha) = 2\delta_{(-\ell,0)}(\widetilde{D},\alpha) + \frac{2\pi i ^k}{\textup{det}(L_{\mf{m}})^{1/2}} \left(\frac{-\widetilde{D}}{\ell}\right)^{\frac{k-2}{2}}2\sum_{c \geq 1}c^{-2}J_{k-2}\left(\frac{4\pi(-\ell \widetilde{D})^{1/2}}{c}\right)\times\\\times H_{L_{\mf{m}},c}(\widetilde{D},\alpha),
\end{multline*}
where
\begin{equation*}
    \delta_{(-\ell,0)}(\widetilde{D},\alpha) := \begin{cases}
    1, \,\, \textup{if} \,\,\widetilde{D}=-\ell,\,\,\textup{and} \,\,\alpha \in L_{\mf{m}}\\
    0,\,\,\textup{otherwise}
\end{cases},
\end{equation*}
and
\[
H_{L_{\mf{m},c}}(\widetilde{D},\alpha):= \sum_{d \in \left(\mathbb{Z}/c\mathbb{Z}\right)^{\times}}\sum_{\lambda \in L_{\mf{m}}/cL_{\mf{m}}} e_c\left((\phi[\lambda]+ \ell)d^{-1} + (\phi[\alpha]-\widetilde{D})d + 2\phi(\alpha, \lambda)\right),
\]
with $e_c(x) := \textup{exp}\left(\frac{2\pi x}{c}\right), \, x\in \mathbb{R}$ and $\lambda_{k,D_{\mf{m}}}$ the normalising factor of \cite[Theorem 2.3, (i)]{jacobi_lattice}. The function $J_{k-2}$ is the $J$-Bessel function of index $k-2$ (see \cite[(1.1)]{jacobi_lattice}). 
\end{theorem}

\begin{remark}
   In the above theorem it is understood that for $c=1$ we have $H_{L_{\mf{m}},1}(\widetilde{D},\alpha) = \textup{exp}\left(2 \pi i (\ell+\phi[\alpha]-\widetilde{D}) \right)$. 
\end{remark}

By \cite[Theorem 4.26]{jacobi_lattice}, the coefficients $c_{\mf{m}}(a,\alpha)$ of $P_{\mf{m}}$ determine those of $V_N \,P_{\mf{m}}$. Using this, each $\mathcal{P}_{\mf{m}}$ will have a Fourier expansion of the form (cf. \cite[eq. (6.2)]{sugano_lifting})
\begin{equation*}
    \mathcal{P}_{\mf{m}}(Z) = \sum_{\eta \in L_{\mf{m},0}^{*}} A_{\mf{m}}(\eta) e(\eta^{t}S_0 Z) = \sum_{a, b \in \mathbb{N}, \alpha \in L_{\mf{m},0}^{*}}\sum_{r} r^{k-1} c_{\mf{m}}\left(\frac{ab}{r^2}, \frac{\alpha}{r}\right)e(b\tau' + \alpha^t Sz + a \tau),
\end{equation*}
where $r$ runs over all positive integers $r \mid \gcd(a,b)$ and $\alpha/r \in L_{\mf{m},0}^{*}$.\\

In particular, if $\eta = (a, \alpha, b) \in L_{\mf{m},0}^{*}$, then 
\begin{equation}\label{relation-fourier-with-fj}
    A_{\mf{m}}(\eta) = \sum_{r} r^{k-1} c_{\mf{m}}\left(\frac{ab}{r^2}, \frac{\alpha}{r}\right),
\end{equation}
using the fact that $c_{\mf{m}}(n,r) = c_{\mf{m}}(n,-r)$, since we assume the weight $k$ is even.\\

From the discussion above we know that $c_{\mf{m}}\left(\frac{ab}{r^2}, \frac{\alpha}{r}\right) =C_{\mf{m}}(\widetilde{D},\frac{\alpha}{r})$, with $\widetilde{D}=\frac{ab}{r^2}- r^{-2}\phi[\alpha]$, depends only on the class of $r^{-1}\alpha \in L^*_{\mf{m}}$ modulo $L_{\mf{m}}$. Hence in order to understand how the Fourier coefficients of $\mathcal{P}_{\mf{m}}$, for the different choices of $\mf{m}$, are related, we need to understand how the discriminant group $L^*/L$ is related to $L^*_{\mf{m}}/L_{\mf{m}}$. We do this next. \newline

We first note that the element $h_{\mf{m}}$ induces a map $L^*/L \longrightarrow L^*_{\mf{m}}/L_{\mf{m}}$. Indeed, by Lemma \ref{dual-Lm} we have the bijections $L_0^*/L_0 \longrightarrow L^*/L$ and $L^*_{\mf{m},0}/L_{\mf{m},0} \longrightarrow L^*_{\mf{m}}/L_{\mf{m}}$. Since $L_{\mf{m},0} = h_\mf{m} L_0$, the claim follows. We then simply write $h_{\mf{m}}: L^*/L \longrightarrow L^*_{\mf{m}}/L_{\mf{m}}$ for this map. \\

We now want to relate the above map to another map which is induced by the adelic elements $g_{\mf{m}}$. We start a little more generally.\newline

Let $X$ be a finite-dimensional $\mathbb{Q}$-vector space and write $\Lambda$ for a $\mathbb{Z}$-lattice in it. Selecting a basis accordingly, we may identify $X$ with $\mathbb{Q}^n$ and $\Lambda$ with $\mathbb{Z}^n$, where $n$ is the dimension of $X$. We now write
\begin{equation*}
    X_{\mathbb{A}}:=\{ (x_p) \in X_p \,\,|\, x_p \in \Lambda_p,\,\,\textup{for almost all $p$}\},
\end{equation*}
where $X_p = X \otimes_{\mathbb{Q}} \mathbb{Q}_p$ and $\Lambda_p = \Lambda \otimes_{\mathbb{Z}} \mathbb{Z}_p$. 
For an element $g=(g_p) \in \textup{GL}_n(X_{\mathbb{A}})= \textup{GL}_n(X)_{\mathbb{A}}$ and $x \in X_{\mathbb{A}}$ we can define $gx \in X_{\mathbb{A}}$ by $(gx)_p = g_p x_p$. We have the following Lemma.

\begin{lemma} Let $\Lambda_{\mathbb{A}} = \prod_p {\Lambda_p}$. The map $x + \Lambda \mapsto (x_p) + \Lambda_{\mathbb{A}}$
gives an isomorphism of abelian groups $X/\Lambda \rightarrow X_{\mathbb{A}}/\Lambda_{\mathbb{A}}$.\end{lemma}

\begin{proof}
It is clear that the map is well-defined and injective, so what is left to be shown, is that the map is also surjective. That is, given any class $x + \Lambda_{\mathbb{A}} \in X_{\mathbb{A}}/\Lambda_{\mathbb{A}}$ we may find a representative in $X$. If we write $x = (x_p)$ then for all but finitely many $p$, we have $x_p \in \Lambda_p$. We write $\mathbf{s}$ for the finite set of primes where $x_p=(x_{1,p}, \ldots,x_{n,p}) \not \in \Lambda_p$. Since 
\begin{equation*}
    X_p/\Lambda_p \cong \mathbb{Q}_p^n/\mathbb{Z}^n_p \cong \left(\mathbb{Q}_p/\mathbb{Z}_p\right)^n \cong \left(\bigcup_{m\geq 1} p^{-m}\mathbb{Z}/\mathbb{Z}\right)^n,
\end{equation*}
for each $p \in \mathbf{s}$ we may write $x_p +\Lambda_p = w_p + \Lambda_p$ with $w_p = (w_{1,p}, \ldots, w_{n,p}) \in \left( \bigcup_{m\geq 1} p^{-m}\mathbb{Z}\right)^n \subset \mathbb{Q}^n$. We then set $z:= \sum_{p \in \mathbf{s}} w_p \in X$. Then we have $x+\Lambda_{\mathbb{A}} = z + \Lambda_{\mathbb{A}}$ since for any $q \neq p$ we have $w_p \in \Lambda_q$.
\end{proof}

Given an element $g \in \textup{GL}_n(X)_{\mathbb{A}}$, we obtain a map $g : X_{\mathbb{A}}/\Lambda_{\mathbb{A}}\rightarrow X_{\mathbb{A}}/ (g\Lambda)_{\mathbb{A}}$, and hence by the Lemma above a map $g :X/\Lambda\rightarrow X/ g\Lambda $. We also note that if $g \in \textup{GL}_n(X)$, then the above map is nothing else than the map induced by $g :X/\Lambda\rightarrow X/ g\Lambda $ since $(gx)_p = (g_p x_p)$. \newline

By the discussion above, the adelic element $g_{\mf{m}}$ induces a map $L^*/L \longrightarrow L^*_{\mf{m}}/L_{\mf{m}}$, and similarly the element $\textup{diag}( 1, g_{\mf{m}}, 1)$ induces a map $L^*_0/L_0 \longrightarrow L^*_{\mf{m},0}/L_{\mf{m},0}$. It is also clear that if we identify $L^*_0/L_0$ with $L^*/L$ and $L_{\mf{m},0}^*/L_{\mf{m},0}$ with $L_{\mf{m}}^*/L_{\mf{m}}$, the two maps agree. As above, we simply write $g_{\mf{m}} :L^*/L \rightarrow L^*_{\mf{m}}/L_{\mf{m}}$ for this map. The following Lemma will allow us to relate the map induced by $h_{\mf{m}}$ to the map induced by $g_{\mf{m}}$ as long as one is a good ``approximation'' of the other.
\begin{lemma}\label{h_m and g_m} The elements $h_{\mf{m}}$ of Lemma \ref{hm} have the property that the two maps $h_{\mf{m}} :L^*/L \rightarrow L^*_{\mf{m}}/L_{\mf{m}}$
and $g_{\mf{m}} :L^*/L \rightarrow L^*_{\mf{m}}/L_{\mf{m}}$ are the same.   
\end{lemma}

\begin{proof} We write $\widehat{D}^*$ for the discriminant kernel subgroup of $D^*$, that is 
\begin{equation*}
    \widehat{D}^*= \prod_{p} \widehat{D}^*_p,
\end{equation*}
where $\widehat{D}_p^* =\{ g_p \in D_p^*\,\,\,|\,\,\, (g_p-1)L^*_{0,p}\subseteq L_{0,p} \}$. Recall that $h_{\mf{m}}$ of Lemma \ref{hm} is chosen in the form $h_{\mf{m}} = \textup{diag}(1,g_{\mf{m}},1)d$, with $d \in \tau(J)$. But from \cite{paper}, we have that $\widehat{D}^{*} = \tau(J)$, so in fact $d \in \widehat{D}^{*}$.

We now consider the maps $L^*/L \longrightarrow L^*_{\mf{m}}/L_{\mf{m}}$, induced by $h_{\mf{m}}$ and $g_{\mf{m}}$. These two maps differ by the map induced by the element $d : L^*/L \longrightarrow L^*/L$. But this map is the identity map since $d$ is in the discriminant kernel subgroup, and hence the maps induced by $h_{\mf{m}}$ and $g_{\mf{m}}$ agree.   
\end{proof}
We now have the following Lemma.
\begin{lemma}\label{discriminant-modules}
   Let $\beta \in L^*_{\mf{m}}$ and $\beta' \in L^*$ such that $g_{\mf{m}}^{-1}(\beta) = \beta'$, with respect to the mapping $g_{\mf{m}}^{-1}: L_{\mf{m}}^*/L_{\mf{m}} \longrightarrow L^*/L$. Then for any $\widetilde{D}$ we have
   $C_{\mf{m}}(\widetilde{D}, \beta) = C(\widetilde{D}, \beta')$.
\end{lemma}

\begin{proof}
We use Theorem \ref{Fourier Coefficients of Poincare} and the explicit description of the Fourier coefficients of the Poincaré series. We first note that we have $\textup{det}(L) = \textup{det}(L_{\mf{m}})$ as the lattices are in the same genus. Actually, these are equal to $\# L^*/L = \# L^*_{\mf{m}}/L_{\mf{m}}$. Hence, in order to establish the Lemma it is enough to show that $H_{L,c}(\widetilde{D},g_{\mf{m}}^{-1}(\beta)) = H_{L_{\mf{m}},c}(\widetilde{D},\beta)$ for all $c \in \mathbb{N}$. But (recall we set $r_{\mf{m}}=0)$,
\[
H_{L,c}(\widetilde{D},g_{\mf{m}}^{-1}(\beta)) = 
\sum_{d \in (\mathbb{Z}/c\mathbb{Z})^{\times}} \sum_{\lambda \in L/cL} e_c((\phi[\lambda] +\ell)d^{-1} + (\phi[g^{-1}_{\mf{m}}(\beta)] -\widetilde{D})d+2\phi(g^{-1}_{\mf{m}}(\beta),\lambda)).
\]
Now we note that the expressions above within the exponential are all $\pmod{c}$ for $c \geq 1$. For $c \geq 2$, the element $g_{\mf{m}}$ induces an isometries $L/cL \rightarrow L_{\mf{m}}/cL_{\mf{m}}$ and $L^*/cL^* \rightarrow L^*_{\mf{m}}/cL^*_{\mf{m}}$. That is,
\[
H_{L,c}(\widetilde{D},g_{\mf{m}}^{-1}(\beta))=\sum_{d \in (\mathbb{Z}/c\mathbb{Z})^{\times}} \sum_{\lambda \in L/cL} e_c((\phi[g_{\mf{m}}(\lambda)] +\ell)d^{-1} + (\phi[\beta] -\widetilde{D})d+\phi(\beta,g_{\mf{m}}(\lambda))=
H_{L_\mf{m},c}(\widetilde{D}, \beta).
\]
For $c=1$ we have
\[
H_{L,1}(\widetilde{D},g_{\mf{m}}^{-1}(\beta))= e(\ell + \phi[g^{-1}_{\mf{m}}(\beta)] -\widetilde{D}) =e(\ell + \phi[\beta'] -\widetilde{D}).
\]
But $(\widetilde{D},\beta) \in \textup{supp}(L_{\mf{m}})$ and $(\widetilde{D},\beta')\in \textup{supp}(L)$. That is, $\phi[\beta] \equiv \widetilde{D} \equiv \phi[\beta'] \pmod{\mathbb{Z}}$. And so,
\[
H_{L,1}(\widetilde{D},g_{\mf{m}}^{-1}(\beta))=e(\ell + \phi[\beta'] -\widetilde{D})=1 = e(\ell + \phi[\beta] -\widetilde{D})=H_{L_{\mf{m}},1}(\widetilde{D},\beta).\qedhere
\]
\end{proof}
We are now ready to prove the following Proposition.
\begin{proposition}\label{adelic-poincare}
    Let $\bm{P} \in S_k(\widehat{D}_{f})$ corresponding to $\mathcal{P}=\mathcal{P}(k,-\ell q,0) \in S_k(\widehat{\Gamma}(L_1))$, for some $\ell \in \mathbb{N}$. Then, for each fractional ideal $\mathfrak{m}$, and with $D_{\mf{m}} = -\ell q_{\mf{m}}$ , and $r_{\mf{m}}=0$, we have
    \begin{equation*}
        \mathcal{P}_{\mf{m}}(Z) = \bm{P}(\widetilde{h}_{\mf{m}}; Z) , \ \forall Z \in \mathcal{H}_S.
    \end{equation*}
\end{proposition}

\begin{proof}
By Lemma \ref{fourier-coeffs-different-m}, with $F_{\mf{m}}(Z)$ there taken as  $\bm{P}(\widetilde{h}_{\mf{m}}; Z)$ it is enough to show that $A_{\mf{m}}(\eta) = A(h_{\mf{m}}^{-1}\eta)$ for all $\eta \in L_{\mf{m},0}^{*}$, where $A_{\mf{m}}$ are the Fourier coefficients of $\mathcal{P}_{\mf{m}}$.\\

By Lemma \ref{primitive}, we can write $\eta = d \zeta$, with $\zeta \in L_{\mf{m},0}^{*}$ primitive, and $d \in \mathbb{Z}_{\geq 1}$. Using \cite[eq. (5.12)]{sugano_lifting}, we can assume $\zeta = (a, \alpha, 1)$. We also write $h_{\mf{m}}^{-1}\zeta = (a', \alpha', b')$, which is now a primitive vector in $L_{0}^{*}$ by Lemma \ref{primitive}. We then have, using \cite[eq. (5.13)]{sugano_lifting},
\begin{equation*}
    A_{\mf{m}}(\eta) = A_{\mf{m}}(d\zeta) = \sum_{r \mid d} r^{k-1}A_{\mf{m}}\left(\m{ad^2r^{-2}\\\alpha dr^{-1}\\1}\right) = \sum_{r \mid d} r^{k-1} c_{\mf{m}}(ad^2r^{-2}, \alpha dr^{-1}).
\end{equation*}
 
On the other hand, we have 
\begin{multline*}
    A(h_{\mf{m}}^{-1}\eta) = A(dh_{\mf{m}}^{-1} \zeta) = A\left(d\m{a'\\\alpha'\\b'}\right) = A\left(d \m{a'b'\\\alpha' \\1}\right) = \sum_{r \mid d}r^{k-1}A\left(\m{a'b'd^2r^{-2}\\ \alpha' dr^{-1}\\ 1}\right)  =\\= \sum_{r\mid d} r^{k-1}c(a'b'd^2r^{-2}, \alpha' d r^{-1}).
\end{multline*}
We will now show that for each $r \mid d$, we have $c_{\mf{m}}(ad^2r^{-2}, \alpha dr^{-1}) = c(a'b'd^2r^{-2}, \alpha' d r^{-1})$. But, $\phi_0[\zeta] = \phi_0[h_{\mf{m}}^{-1}\zeta]$, so $a - \phi[\alpha] = a'b'-\phi[\alpha']$, which then implies $ad^2r^{-2}-\phi[\alpha d r^{-1}]  = a'b'd^2/r^2 - \phi[\alpha' dr^{-1}]$. Let then $-\widetilde{D} := ad^2r^{-2}- \phi[\alpha dr^{-1}]= a'b'd^2/r^2 - \phi[\alpha' dr^{-1}]$, then $C_{\mf{m}}(\widetilde{D},\alpha d/r) = c_{\mf{m}}(\phi[\alpha dr^{-1}]-\widetilde{D}, \alpha dr^{-1})$. Similarly, $C(\widetilde{D},\alpha'd/r) = c(\phi_0[\alpha'd/r]-\widetilde{D}, \alpha' d/r)$, we need to show $C_{\mf{m}}(\widetilde{D}, \alpha d/r) = C(\widetilde{D}, \alpha'd/r)$. By Lemma \ref{h_m and g_m}, we have $\alpha' = g_{\mf{m}}^{-1}(\alpha) \pmod {L}$, we get $\alpha'd/r = g_{\mf{m}}^{-1} (\alpha d/r) \pmod L$. But the coefficient $C(\widetilde{D}, \alpha' d/r)$ depends only on $\alpha'd/r \pmod L$ and hence $C(\widetilde{D}, \alpha' d/r) = C(\widetilde{D}, g_{\mf{m}}^{-1}(\alpha d/r))$. Therefore, we need to show $C_{\mf{m}}(\widetilde{D}, \alpha d/r) = C(\widetilde{D}, g_{\mf{m}}^{-1}(\alpha d /r))$ which follows from Lemma \ref{discriminant-modules}.
\end{proof}

\section{Dirichlet Series}\label{section-dirichlet-series}
Assume now $\bm{F}\in S_k(D_{f})$ and $\bm{G} \in S_k(\widehat{D}_f)$. For each fractional ideal $\mathfrak{m}$ of $\mathcal{O}_K$, consider the modular forms $F_{\mf{m}}, G_{\mf{m}}$ defined by $F_{\mf{m}}(Z):= \bm{F}(\widetilde{h}_{\mf{m}};Z)$ and $G_{\mf{m}}(Z):= \bm{G}(\widetilde{h}_{\mf{m}};Z)$ for $Z \in \mathcal{H}_S$, with corresponding Fourier-Jacobi coefficients $\{\phi_{\mf{m}, N}\}_{N=1}^{\infty}$ and $\{\psi_{\mf{m}, N}\}_{N=1}^{\infty}$. For any finite set of primes $\mathcal{Q}$, we define the $\mathfrak{m}$-th Dirichlet series
\begin{equation*}
    D_{\mathcal{Q}}(F_{\mf{m}}, G_{\mf{m}};s) := \sum_{\substack{N=1\\ (N,p)=1 \ \forall p \in \mathcal{Q}}}^{\infty} \langle \phi_{\mathfrak{m}, N}, \psi_{\mf{m}, N}\rangle N^{-s},
\end{equation*}
where $\langle \ , \ \rangle$ is the inner product in the space of Fourier-Jacobi forms of weight $k$ and index $(L_{\mathfrak{m}}, N\phi)$. By comparison with \cite[Lemma 6.1]{psyroukis_orthogonal}, this Dirichlet series converges absolutely for $\textup{Re}(s)>k+1$ and represents a holomorphic function on this domain. We then set 
\begin{equation}\label{general-dirichlet-series}
    D_{\mathcal{Q}}(\bm{F}, \bm{G};s):= \sum_{\mathfrak{m} \in \textup{Cl}(K)} \zeta_{\mf{m}, \mathcal{Q}}(s-k+3) D_{\mathcal{Q}}(F_{\mathfrak{m}},G_{\mathfrak{m}};s),
\end{equation}
where 
\begin{equation}\label{ideal-class-zeta}
    \zeta_{\mf{m}}(s) := \sum_{\substack{0\neq \mf{a} \subset \mathcal{O}_K\\\mf{a} \in \mf{m}}} N (\mathfrak{a})^{-s}
\end{equation}
denotes the zeta function attached to the ideal class $\mathfrak{m}$. Here $\mf{m}$ runs over a set of representatives of the class group of $K$. This is also absolutely convergent and holomorphic for $\textup{Re}(s)>k+1$. We now consider the case $\bm{G} = \bm{P}$, the Poincar\'{e} series defined in Proposition \ref{adelic-poincare}. We then have that for each $\mf{m}$ as above,

\begin{equation*}
    D_{\mathcal{Q}}(F_{\mathfrak{m}}, \mathcal{P}_{\mathfrak{m}};s) := \sum_{\substack{N=1\\ (N,p)=1 \ \forall p \in \mathcal{Q}}}^{\infty} \langle \phi_{\mathfrak{m}, N}, V_NP_{\mathfrak{m}}\rangle N^{-s}.
\end{equation*}

Before we go any further, we make some remarks on how the Dirichlet series $D_{\mathcal{Q}}(F_{\mathfrak{m}}, \mathcal{P}_{\mathfrak{m}};s)$ depends on the isometry class of the lattice $L_{\mf{m}}$. In particular, we will show that it is independent of the choice of the representatives $\mf{m}$.\\

Let us pick two lattices $L_{\mf{m}_1}$ and $L_{\mf{m}_2}$ and assume that there is exists an isometry $\sigma : L_{\mf{m}_1} \longrightarrow L_{\mf{m}_2}$. It is well known (see for example \cite[Section 4.1]{jacobi_lattice}) that if we extend $\sigma$ to obtain a bijection $\sigma :L_{\mf{m}_1} \otimes_{\mathbb{Z}} \mathbb{C} \longrightarrow L_{\mf{m}_2} \otimes_{\mathbb{Z}} \mathbb{C}$, then we have an isomorphism $S_{k}(L_{\mf{m}_2},N\phi) \longrightarrow S_{k}(L_{\mf{m}_1},N\phi)$, for any $N \geq 1$ given by $\psi(\tau,z) \longmapsto \sigma(\psi):=\psi(\tau, \sigma z)$. Moreover, we have that this map preserves the inner product, that is $\langle \psi_1, \psi_2\rangle = \langle \sigma(\psi_1), \sigma(\psi_2) \rangle$. This can be seen from the definition of the inner product as in \cite[Definition 1.33]{jacobi_lattice}. If we now write $$\psi(\tau,z) = \sum_{n \in \mathbb{Z}, r \in L^*_{\mf{m}_2}} c(n,r) e(n \tau + \phi(r,z)),$$ then 
\[
\sigma(\psi)(\tau,z) = \psi(\tau, \sigma(z))=\sum_{n \in \mathbb{Z},\ r \in L^*_{\mf{m}_{1}}} c(n,\sigma(r)) e(n \tau + \phi(r,z)).
\]
The isometry $\sigma$ induces a map $\sigma: L^*_{\mf{m}_1}/L_{\mf{m}_1} \longrightarrow L^*_{\mf{m}_2}/ L_{\mf{m}_2}$, and hence a similar argument as in the proof of Lemma \ref{discriminant-modules} gives that $P_{\mf{m}_1} = \sigma(P_{\mf{m}_2})$. Moreover, from the definition of the operator $V_N$, we see that $\sigma(V_N \psi) = V_N (\sigma (\psi))$ (see for example \cite[Theorem 4.26]{jacobi_lattice}, for any Jacobi form $\phi$, and hence in particular $\sigma(V_N(P_{\mf{m}_2})) = V_N(P_{\mf{m}_1})$, for all $N \geq 1$.\newline

Now we look at $\phi_{\mf{m}_i,N}$, for $i=1,2$,  the $N$-th Fourier-Jacobi of the modular form $F_{\mf{m}_i}$. We set $\sigma_0:= \textup{diag}(1, \sigma ,1) \in G^*_{\mathbb{Q}}$. We then observe that if $\gamma:=h_{\mf{m}_1} h^{-1}_{\mf{m}_2}\sigma_0$, then  $\gamma L_{\mf{m}_1, 0} = L_{\mf{m}_1, 0}$. In particular, $\gamma \in \Gamma(L_{\mf{m}_1,0})$, and so $\gamma_1:=\textup{diag}(1 ,\gamma,1) = \widetilde{h}_{\mf{m}_1}\widetilde{h}_{\mf{m}_2}^{-1} \textup{diag}(1,\sigma_0,1) \in \Gamma(L_{\mf{m}_1,1})$. Hence, we have that $F_{\mf{m}_1}|_k \gamma_1  = F_{\mf{m}_1}$,
or equivalently $F_{\mf{m}_1} (Z) = F_{\mf{m}_1}(\gamma Z) = F_{\mf{m}_1}(h_{\mf{m}_1} h_{\mf{m}_2}^{-1}\sigma_0(Z))$.\\

By Lemma \ref{fourier-coeffs-different-m}, we have that
\[
F_{\mf{m}_1} (Z)=\sum_{\eta \in L_{\mathfrak{m}_1,0}^{*}} A(h^{-1}_{\mf{m}_1}\eta) e(2\phi_0(\eta,h_{\mf{m}_1} h_{\mf{m}_2}^{-1}\sigma_0(Z)))=\sum_{\eta \in L_{\mathfrak{m}_1,0}^{*}} A(h^{-1}_{\mf{m}_1}\eta) e(2\phi_0(h_{\mf{m}_2}h_{\mf{m}_1}^{-1}\eta, \sigma_0(Z)))
\]

Setting $\eta \longmapsto h_{\mf{m}_1} h^{-1}_{\mf{m}_2}\eta'$ with $\eta' \in L^*_{\mf{m},2}$ we obtain,
\[
F_{\mf{m}_1} (Z)=F_{\mf{m}_2}(\sigma_0(Z))
\]
But $\sigma_0(Z) = \sigma_0(\tau',z,\tau) = (\tau',\sigma(z),\tau)$, and hence we obtain that $\phi_{\mf{m}_1,N}(\tau,z) = \phi_{\mf{m}_2,N}(\tau,\sigma_0(z))$ or equivalently $\sigma(\phi_{\mf{m}_2,N}) = \phi_{\mf{m}_1,N}$, for all $N \geq 1$. \newline

That is, if $L_{\mf{m}_1}$ is isometric to $L_{\mf{m}_2}$, we then obtain
\[
D_{\mathcal{Q}}(F_{\mathfrak{m}_1}, \mathcal{P}_{\mathfrak{m}_1};s) = D_{\mathcal{Q}}(F_{\mathfrak{m}_2}, \mathcal{P}_{\mathfrak{m}_2};s).
\]
In particular, from the discussion in Section \ref{dirichlet series start}, we know that $L_{\mf{m}_1}$ is isometric to $L_{\mf{m}_2}$ if and only if the fraction ideal $\mf{m}_1\mf{m}_2^{-1} \in \textup{Cl}(K/\mathbb{Q})$. In particular, this is the case for when $\mf{m}_1$ is in the same class as in $\mf{m}_2$ in $\textup{Cl}(K)$. Since the zeta function $\zeta_{\mf{m}}(s)$ is independent of the choice of representatives by definition, the Dirichlet series \eqref{general-dirichlet-series} is also independent of the choice of representatives. Moreover, we may write
\begin{equation*}
    D_{\mathcal{Q}}(\bm{F}, \bm{P};s)= \sum_{\mathfrak{m} \in \textup{Cl}(K)/\textup{Cl}(K/\mathbb{Q})}\left(\sum_{\mf{n}\in \textup{Cl}(K/\mathbb{Q})} \zeta_{\mf{m}\mf{n},\mathcal{Q}}(s-k+3) \right)D_{\mathcal{Q}}(F_{\mathfrak{m}},P_{\mathfrak{m}};s).
\end{equation*}
We now rewrite $ D_{\mathcal{Q}}(F_{\mathfrak{m}}, P_{\mathfrak{m}};s)$  as in \cite[Proposition 7.3]{psyroukis_orthogonal}, but with a small difference. Let us make this explicit.\\

First, for each $\mf{m}$, it is clear that $\widehat{\Gamma}(L_{\mf{m},0})$ acts on $L_{\mf{m},0}\left[\ell, 2^{-1}\mathbb{Z}\right]$ (see equation \eqref{shimura set}) via left multiplication. By Lemma \ref{isomorphisms}, we see that the map $v \longmapsto h_{\mf{m}}v$, for $v \in L_0[\ell,2^{-1}\mathbb{Z}]$ induces a bijection $L_0\left[\ell, 2^{-1}\mathbb{Z}\right]/\widehat{\Gamma}(L_0) \longrightarrow L_{\mf{m},0}\left[\ell, 2^{-1}\mathbb{Z}\right]/\widehat{\Gamma}(L_{\mf{m},0})$. We now have the following Lemma.

\begin{lemma}\label{claim on action of discriminant kernel}
For each $\mf{m}$, there is a finite number of orbits of the action of $\widehat{\Gamma}(L_{\mf{m},0})$ on $L_{\mf{m},0}\left[\ell, 2^{-1}\mathbb{Z}\right]$.
\end{lemma}
\begin{proof}
    From the above, it suffices to only show this for $\mf{m} = \mathcal{O}_K$. We defer the proof of this just after Proposition \ref{finite-quotients}.
\end{proof}
Hence, we may set $h_1 := \#\left(L_{\mf{m},0}\left[\ell, 2^{-1}\mathbb{Z}\right]/\widehat{\Gamma}(L_{\mf{m},0})\right)$, independent of $\mf{m}$. Moreover, in a similar fashion to the discussion before \cite[Proposition 7.3]{psyroukis_orthogonal}, we fix a set of representatives $\{\xi_i\}_{i=1}^{h_1}$ for the representatives of $L_0\left[\ell, 2^{-1}\mathbb{Z}\right]/\widehat{\Gamma}(L_0)$, such that $\xi_i \in \mathcal{P}_S$ for all $i$. We then set $\xi_{\mf{m},i} := h_{\mf{m}} \xi_i \in L_{\mf{m},0}\left[\ell, 2^{-1}\mathbb{Z}\right]$ for $i=1,\ldots h_1$. The set $\{\xi_{\mf{m},i}\}$ is then a set of representatives for $L_{\mf{m},0}\left[\ell, 2^{-1}\mathbb{Z}\right]/\widehat{\Gamma}(L_{\mf{m},0})$ and since we are taking $h_{\mf{m}} \in G^{*,+}_{\mathbb{R}}$, we have that also $\xi_{\mf{m},i} \in \mathcal{P}_S$. \newline

We now define
\begin{multline}\label{n-set}
    n(\xi_{\mathfrak{m},i};d) := \#\left\{s \in \mathbb{Z}^{2}/dS_{\mathfrak{m}}\mathbb{Z}^{2} \mid D_{\mf{m}} \equiv \frac{1}{2}q_{\mathfrak{m}}S_{\mathfrak{m}}^{-1}[s] \pmod{q_{\mathfrak{m}}d}, \right.\\\left.\left(\frac{\frac{1}{2}q_{\mathfrak{m}}S_{\mathfrak{m}}^{-1}[s]-D_{\mf{m}}}{q_{\mathfrak{m}}d}, A_{\mathfrak{m}}S_{\mathfrak{m}}^{-1}s, d\right)^{t} = \gamma \xi_{\mathfrak{m},i}, \gamma \in \widehat{\Gamma}^{+}(L_{\mathfrak{m},0})\right\}.
\end{multline}

We now set $\mathcal{P} := \bigcup_{\mf{m}} \mathcal{P}_{\mf{m}}$, where $\mathcal{P}_{\mf{m}}$ are finite sets containing all the prime factors of $D_{\mf{m}} = -\ell q_{\mf{m}}$. Using the expansions in \eqref{expansion-of-jacobi}, we then have the following Proposition, as in \cite[Proposition 7.3]{psyroukis_orthogonal}.
\begin{proposition}\label{first form of result}
    Let $(-\ell q_{\mf{m}},r_{\mf{m}}) \in \widetilde{\textup{supp}}(L_{\mf{m}}, \phi)$. We then have 
\begin{equation*}
    D_{\mathcal{Q}}(F_{\mathfrak{m}}, P_{\mathfrak{m}};s) = \sum_{i=1}^{h_{1}} \zeta_{\mathcal{Q}}(\xi_{\mathfrak{m},i}; s-k+3)D_{\mathcal{Q}}(F_{\mathfrak{m}}, \xi_{\mathfrak{m},i};s),
\end{equation*}
where 
\begin{equation*}
    D_{\mathcal{Q}}(F_{\mathfrak{m}}, \xi_{\mathfrak{m},i};s) := \sum_{\substack{N=1\\ (N,p)=1 \ \forall p \in \mathcal{Q}}}^{\infty} a(\widetilde{h}_{\mathfrak{m}}; N\xi_{\mathfrak{m},i})N^{-s},
\end{equation*}
and 
\begin{equation}\label{zeta xi i m definition}
    \zeta_{\mathcal{Q}}(\xi_{\mathfrak{m},i}; s) := \sum_{\substack{d=1\\ (d,p)=1 \ \forall p \in \mathcal{Q}}}^{\infty} n(\xi_{\mathfrak{m},i}; d)d^{-s}.
\end{equation}
\end{proposition}

Using now Lemma \ref{fourier-coeffs-different-m}, we can write $a(\widetilde{h}_{\mathfrak{m}}; N\xi_{\mathfrak{m},i}) = a(id, Nh_{\mathfrak{m}}^{-1}\xi_{\mathfrak{m},i}) = a(id, N\xi_{i})=A(N\xi_i)$ for all $N \not\in \mathcal{Q}$. In particular, we obtain that  $D_{\mathcal{Q}}(F_{\mathfrak{m}}, \xi_{\mathfrak{m},i};s) = D_{\mathcal{Q}}(F, \xi_i;s)$ for all $\mathfrak{m}$. \newline

We now select a set of representatives for $\textup{Cl}(K)$ to be relatively prime to any ramified ideals. That is, if we write $\mathfrak{d}_K$ for the different of $K$ over $\mathbb{Q}$, we assume that $(\mf{m},\mathfrak{d}_K) = 1$. Therefore,
\begin{multline*}
    D_{\mathcal{Q}}(\bm{F}, \bm{P};s) = \sum_{\mathfrak{m} \in \textup{Cl}(K)} \zeta_{\mf{m}, \mathcal{Q}}(s-k+3) D_{\mathcal{Q}}(F_{\mathfrak{m}},P_{\mathfrak{m}};s)= \\=\sum_{\mathfrak{m} \in \textup{Cl}(K)}\zeta_{\mf{m}, \mathcal{Q}}(s-k+3)  \sum_{i=1}^{h_{1}} \zeta_{\mathcal{Q}}(\xi_{\mathfrak{m},i}; s-k+3)D_{\mathcal{Q}}(F, \xi_{i};s)=\\=\sum_{i=1}^{h_{1}}\left(\sum_{\mathfrak{m} \in \textup{Cl}(K)}\zeta_{\mf{m}, \mathcal{Q}}(s-k+3)\zeta_{\mathcal{Q}}(\xi_{\mathfrak{m},i}; s-k+3) \right)D_{\mathcal{Q}}(F,\xi_i;s),
\end{multline*}
where 
\begin{equation}\label{original-dirichlet}
   D_{\mathcal{Q}}(F,\xi_i;s)=\sum_{\substack{N=1\\ (N,p)=1 \ \forall p \in \mathcal{Q}}}^{\infty} A(N\xi_i)N^{-s}. 
\end{equation}

\section{Equivalence of Quadratic Spaces}\label{equivalence of quadratic spaces}
Our aim is now to understand the Dirichlet series $D_{\mathcal{Q}}(F, \xi_i ;s)$. To that end, in this Section, we will introduce an isomorphism of the quadratic space $(V_0, \phi_0)$ with a quadratic space of Hermitian matrices over an imaginary quadratic field.

We consider the four-dimensional (over $\mathbb{Q}$) quadratic space $(U, \psi)$, where 
\begin{equation*}
    U := \left\{h \in M_{2}(K) \mid \overline{h}^t = h\right\},
\end{equation*}
and $\psi[h] = \det(h), \ \forall h \in U$. One has (see for example \cite[(2.1d)]{shimura2006quadratic}) that the corresponding bilinear form is given by $2\psi(x,y) = \textup{tr}(xy^{\iota})$, where $\iota$ here is the main involution of $M_2(K)$. We define the following elements of $U$
\begin{equation*}
    e := \m{1&0\\0&0}, \,\,f :=\m{0&0\\0&1}, \,\ g := \m{0&1\\1&0},
\end{equation*}
and set $\Lambda := \mathbb{Z}e + \mathbb{Z}f + \mathcal{O}_Kg$. By \cite[p. 495]{shimura2006quadratic} we have that this is a $\mathbb{Z}$-maximal lattice in $U$.\newline

Let also $U^{+} := \{u \in U \mid u > 0\}$, the set of positive definite matrices in $U$. We then have the following Lemma.

\begin{lemma}\label{isomorphism} 
    There is an isomorphism of quadratic spaces
    $f_{\omega} : (U, \psi) \longrightarrow (V_0, \phi_0)$, given as
    \begin{align}\label{linear isomorphism}
    \begin{split}
        f_{\omega} : U &\longrightarrow V_0 \\ 
            \m{x & y+\omega z\\y+ \overline{\omega}z & w} &\longmapsto \m{x\\y\\z\\w},
    \end{split}
    \end{align}
    where $x,y, z, w \in \mathbb{Q}$. Moreover, the map $f_{\omega}$ is a bijection between $U^{+}$ and $\mathcal{P}_{S}$.
\end{lemma}

\begin{proof} One sees easily that the map $f_\omega$ is well defined linear isomorphism of the $\mathbb{Q}$-vector spaces $U$ and $V_0$ (compare also with \cite[Lemma 1.24 and Lemma 1.25]{wernz}). Moreover, an easy computation gives $\phi_0[f_{\omega}(u)] = \psi[u]$, for $u \in U$, and therefore $f_{\omega}$ is indeed an isometry. The last statement of the Lemma follows from Sylvester's criterion.  
\end{proof}

It is now clear that $f_{\omega}(\Lambda) = L_0$. Moreover, $f_{\omega}$ induces the following isomorphism of orthogonal groups:
\begin{align}\label{g_omega}
\begin{split}
    g_{\omega}: \textup{SO}^{\psi}(U) &\xrightarrow[]{\sim} \textup{SO}^{\phi_0}(V_0)\\
    \alpha &\longmapsto f_{\omega} \circ \alpha \circ f_{\omega}^{-1}
\end{split}
\end{align}
This also extends to an isomorphism of the integral groups $\textup{SO}^{\psi}(\Lambda) \xrightarrow[]{\sim} \textup{SO}^{\phi_0}(L_0)$. Indeed, if $\alpha \in \textup{SO}^{\psi}(\Lambda)$, we have
\begin{equation*}
    g_\omega(\alpha)(L_0) = f_{\omega}\circ\alpha\circ f_{\omega}^{-1}(L_0) = f_{\omega} \circ \alpha(\Lambda) = f_{\omega}(\Lambda) = L_0,
\end{equation*}
using the fact that $\alpha(\Lambda) = \Lambda$ and that $f_\omega$ is an isomorphism.\\

We now have the following Lemma, for the elements of $L_0\left[\ell, 2^{-1}\mathbb{Z}\right]$, where $\ell$ is a non-zero integer.
\begin{lemma}\label{equivalence of solution sets}
    With the notation as above, we have
    \begin{equation*}
        h \in \Lambda\left[\ell, 2^{-1}\mathbb{Z}\right] \iff f_{\omega}(h) \in L_0\left[\ell, 2^{-1}\mathbb{Z}\right].
    \end{equation*}
\end{lemma}
\begin{proof}
    From Lemma \ref{isomorphism}, we have $\psi[h] = \ell \iff \phi_0[f_{\omega}(h)] = \ell$. Therefore, we need to check only the condition on the ideals. Let 
    \begin{equation*}
        h = \m{a & b+\omega c \\ b+\overline{\omega}c & d} \in U
    \end{equation*}
    with $a,b,c,d \in \mathbb{Q}$. From \cite[(2.18)]{shimura2006quadratic}, we have that the condition on the ideals is equivalent to 
    \begin{equation}\label{trace condition on matrices}
        a\mathbb{Z} + d\mathbb{Z} + \textup{Tr}_{K/\mathbb{Q}}((b+\overline{\omega}c)\mathcal{O}_K) = \mathbb{Z}.
    \end{equation}
    We will now consider two cases:\\
    
    \textbf{Case 1}: $m \equiv 3 \pmod{4}$. Then
    \begin{equation}\label{trace condition on vectors1}
        \phi_0(f_{\omega}(h), L_0) = 2^{-1}\mathbb{Z} \iff a\mathbb{Z}+d\mathbb{Z} -(2b+c)\mathbb{Z}-\left(b + \frac{m+1}{2}c\right)\mathbb{Z} = \mathbb{Z}.
    \end{equation}
    
    But $\mathcal{O}_K$ has basis $\{1,\omega\}$. In order to compute $\textup{Tr}_{K/\mathbb{Q}}((b+\overline{\omega}c)\mathcal{O}_K)$, we will compute the traces of the linear maps $x \longmapsto \alpha x$ for $\alpha = b+\overline{\omega}c$ and $\alpha = (b+\overline{\omega}c)\omega = \frac{m+1}{4}c+b\omega$.

    \begin{equation*}
        \textup{Tr}_{K/\mathbb{Q}}(b+ \overline{\omega}c) = \textup{Tr} \left(\m{b+c & (m+1)c/4\\ -c & b}\right) = 2b+c.
    \end{equation*}
    \begin{equation*}
        \textup{Tr}_{K/\mathbb{Q}}\left((m+1)c/4+b\omega\right) = \textup{Tr} \left(\m{(m+1)c/4 & -(m+1)b/4 \\ b & b + (m+1)c/4}\right) = b + \frac{m+1}{2}c.
    \end{equation*}
    We now see that \eqref{trace condition on matrices} and \eqref{trace condition on vectors1} are equivalent.\\
    
    \textbf{Case 2}: $m \equiv 1, 2 \pmod{4}$. In this case,
    \begin{equation}\label{trace condition on vectors2}
        \phi_0(f_{\omega}(h), L_0) = 2^{-1}\mathbb{Z} \iff a\mathbb{Z}+d\mathbb{Z} -2b\mathbb{Z}-2mc\mathbb{Z} = \mathbb{Z}.
    \end{equation}
    In a similar fashion to before, we compute 
    \begin{equation*}
        \textup{Tr}_{K/\mathbb{Q}}(b+ \overline{\omega}c) = \textup{Tr} \left( \m{b & mc\\ -c & b}\right) = 2b.
    \end{equation*}
    \begin{equation*}
        \textup{Tr}_{K/\mathbb{Q}}((b+ \overline{\omega}c)\omega) = \textup{Tr} \left( \m{mc & -mb\\ b & mc}\right) = 2mc.
    \end{equation*}
    Again, it is now evident that \eqref{trace condition on matrices} and \eqref{trace condition on vectors2} are equivalent.
\end{proof}
The next step is to see how the action of $\widehat{\Gamma}(L_0)$ on $L_0\left[\ell, 2^{-1}\mathbb{Z}\right]$ translates in the Hermitian setting. We will first recall some basic definitions related to unitary groups of degree 2. We first define the Hermitian upper half plane
\begin{equation*}
    \mathbb{H}_2 := \{Z \in M_2(\mathbb{C}) \mid -i(Z-\overline{Z}^t) > 0\}.
\end{equation*}
We fix an embedding $K \xhookrightarrow{} \mathbb{C}$. The special unitary group of degree 2 over $K$ is given by 
\begin{equation*}
    \textup{SU}(2,2;K) := \{g \in \textup{SL}_{4}(K) \mid \overline{g}^{t}J g = J\},
\end{equation*}
where $J := \m{0 & -1\\1 & 0}$. Now, $M = \m{A & B \\ C & D} \in \textup{SU}(2,2;K)$ acts on $\mathbb{H}_2$ via
\begin{equation*}
    Z \longmapsto M\langle Z\rangle := (AZ+B)(CZ+D)^{-1}.
\end{equation*}
Finally, we set $\textup{SU}(2,2; \mathcal{O}_K) := \textup{SU}(2,2;K) \cap M_{4}(\mathcal{O}_K)$. We then have the following Proposition.
\begin{proposition}\label{relation of SL_2 with discriminant kernel}
    There is an isomorphism $\Psi : \textup{SL}_2(\mathcal{O}_K)/\{\pm 1\} \longrightarrow \widehat{\Gamma}^{+}(L_0)$ such that
    \begin{equation}\label{equivalence of classes}
        f_{\omega}(\gamma u \gamma^{*}) = \Psi(\gamma)f_{\omega}(u), \ \forall \gamma \in \textup{SL}_2(\mathcal{O}_K), \ u \in U_{+},
    \end{equation}
    where $\gamma^{*}:= \overline{\gamma}^{t}$ and $U^{+}$ is the set of positive definite matrices in $U$.
\end{proposition}
\begin{proof}
    From \cite[Theorem 3.17]{wernz}, we have that there is an isomorphism $\Phi : \textup{SU}(2,2; \mathcal{O}_K)/\{\pm 1\} \longrightarrow \widehat{\Gamma}_S$. Moreover, there is an embedding $\epsilon : \textup{SL}_2(\mathcal{O}_K) \xhookrightarrow{} \textup{SU}(2,2; \mathcal{O}_K)$ by sending $\gamma \longrightarrow \textup{diag}(\gamma, \overline{\gamma}^{-t})$, which is compatible with the embedding $\eta: \widehat\Gamma^{+}(L_0) \xhookrightarrow{}\widehat{\Gamma}_S$ via $\gamma \longmapsto \textup{diag}(1, \gamma, 1)$, because of the explicit form of $\Phi$ given in \cite[Theorem 3.13]{wernz}. We define a map $\Psi: \textup{SL}_2(\mathcal{O}_K)/\{\pm 1\}\longrightarrow \widehat{\Gamma}^{+}(L_0)$ by setting $\Psi(\gamma)$ the element such that $\Phi(\epsilon(\gamma)) = \eta(\Psi(\gamma))$, and we claim this is an isomorphism with the required property. First of all, it is a homomorphism from the way it is defined. To check $\Psi$ is injective, assume $\Psi(\gamma_1) = \Psi(\gamma_2)$. Then $\Phi(\epsilon(\gamma_1)) = \Phi(\epsilon(\gamma_2))$ and therefore $\epsilon(\gamma_1) = \pm \epsilon(\gamma_2)$, which in turn implies $\gamma_1 = \pm \gamma_2$, as wanted. Now, if $M \in \widehat{\Gamma}^{+}(L_0)$, we have $\eta(M) \in \widehat{\Gamma}_S$ and so there exists $\widetilde{\gamma} \in \textup{SU}(2,2;\mathcal{O}_K)$ such that $\Phi(\widetilde{\gamma}) = \eta(M)$. Now, let 
    \begin{equation*}
        \widetilde{\gamma} = \m{A& B \\ C&D},
    \end{equation*}
    with $A,B,C,D \in M_2(\mathcal{O}_K)$. From \cite[Theorem 3.13]{wernz}, we must have $\det A = \det D = 1$ and $f_{\omega}(A^{\sharp}B) = f_{\omega}(AC^{\sharp}) = f_{\omega}(BD^{\sharp}) = f_{\omega}(C^{\sharp}D) = 0$, where $^{\sharp}$ denotes the adjugate of a matrix (see \cite[(3.9)]{wernz}). But $f_{\omega}$ is injective and so $B = C = 0$ and then $\widetilde{\gamma} = \epsilon(A)$, with $A \in \textup{SL}_2(\mathcal{O}_K)$. Therefore, $\Psi$ is also surjective and hence an isomorphism.\\

    Let now $u \in U^{+}$. We then have that $iu \in \mathbb{H}_2$. From \cite[Theorem 3.5]{wernz}, we have $f_{\omega}(M \langle Z\rangle) = \Phi(M)\langle f_{\omega}(Z)\rangle$ for all $Z \in \mathbb{H}_2$ and $M \in \textup{SU}(2,2;\mathcal{O}_K)$. By taking now $M = \epsilon(\gamma)$ with $\gamma \in \textup{SL}_2(\mathcal{O}_K)$ and using that $\Phi(\epsilon(\gamma)) = \eta(\Psi(\gamma))$, \eqref{equivalence of classes} follows by unwinding the definitions of the actions.\qedhere
\end{proof}
We end this Section with the following identification of the orthogonal and Clifford groups of $(U,\psi)$, as is given by Shimura in \cite[Proposition 2.2]{shimura2006quadratic}.
\begin{proposition}\label{Xi proposition}
    Let $\mathcal{G} := \{\xi \in \textup{GL}_2(K) \mid \det(\xi) \in \mathbb{Q}^{\times}\}$. There is an isomorphism $\Xi$ of $G^{+}(U)$ onto $\mathcal{G}$, such that if $\alpha \in G^{+}(U)$ and $\xi = \Xi(\alpha)$, then $\nu(\alpha) = \det (\xi)$ and $\tau(\alpha) h = \det(\xi)^{-1}\xi h \overline{\xi}^{t}$ for all $h \in U$. In particular, $\textup{SO}^{\psi}(U)$ is isomorphic to $\mathcal{G}/\mathbb{Q}^{\times}$.
\end{proposition}
\section{The Dirichlet Series $D_{\mathcal{Q}}(F,\xi_i;s)$}\label{D F xi i}
In this Section, we will relate $D_{\mathcal{Q}}(F,\xi_i;s)$ of \eqref{original-dirichlet} to $L$-functions. We will first establish Lemma \ref{claim on action of discriminant kernel}, namely that $\widehat{\Gamma}(L_0)$ acts on $L_0\left[\ell, 2^{-1}\mathbb{Z}\right]$ and there is a finite number of orbits under this action. For $\ell \in \mathbb{Z}_{>0}$, we define
\begin{equation*}
    L_0^{+}[\ell, 2^{-1}\mathbb{Z}] := \{x=(x_1,x_2,x_3,x_4) \in L_0[\ell, 2^{-1}\mathbb{Z}] \mid x_1 > 0\}.
\end{equation*}
\begin{equation*}
    L_0^{-}[\ell, 2^{-1}\mathbb{Z}] := \{x=(x_1,x_2,x_3,x_4) \in L_0[\ell, 2^{-1}\mathbb{Z}] \mid x_1 < 0\}.
\end{equation*}
Note that $\phi_0[x] = \ell \implies 2x_1x_{4} - S[x] = 2\ell \implies x_1x_{4}>0$, so $x_4$ has the same sign as $x_1$.\\

We claim that there is an equivalence between $L_0[\ell, 2^{-1}\mathbb{Z}]/\widehat{\Gamma}(L_0)$ and $L_{0}^{+}[\ell, 2^{-1}\mathbb{Z}]/\widehat{\Gamma}^{+}(L_0)$, where we recall that $\widehat{\Gamma}^{+}(L_0)$ is the connected component of the identity inside $\widehat{\Gamma}(L_0)$. First of all
\begin{equation*}
    L_0[\ell, 2^{-1}\mathbb{Z}] = L_0^{+}[\ell, 2^{-1}\mathbb{Z}] \sqcup L_0^{-}[\ell, 2^{-1}\mathbb{Z}].
\end{equation*}
If $\sigma := \textup{diag}(-1, 1,1,-1) \in \widehat{\Gamma}^{+}(L_0)\backslash\widehat{\Gamma}(L_0)$, we have that $\widehat{\Gamma}(L_0) = \widehat{\Gamma}^{+}(L_0) \sqcup \widehat{\Gamma}^{+}(L_0)\sigma$. Hence, the element $\sigma$ interchanges the two sets $L_0^{\pm}[\ell,2^{-1}\mathbb{Z}]$, and therefore, the claim follows.\newline

Moreover, using Lemma \ref{equivalence of solution sets}, Proposition \ref{relation of SL_2 with discriminant kernel}, and the fact that the action of $\textup{SL}_2(\mathcal{O}_K)$ on $\Lambda\left[\ell, 2^{-1}\mathbb{Z}\right]$ factors through $\textup{SL}_2(\mathcal{O}_K)/ \{\pm 1\}$, we have further the identification of $L_0^{+}[\ell, 2^{-1}\mathbb{Z}]/\widehat{\Gamma}^{+}(L_0)$ and $\Lambda^+\left[\ell, 2^{-1}\mathbb{Z}\right]/ \textup{SL}_2(\mathcal{O}_K)$, where $\Lambda^+\left[\ell, 2^{-1}\mathbb{Z}\right] := \Lambda\left[\ell, 2^{-1}\mathbb{Z}\right] \cap U^+$. \newline

The quotient $\Lambda^+\left[\ell, 2^{-1}\mathbb{Z}\right]/ \textup{SL}_2(\mathcal{O}_K)$ is studied in detail by Shimura in \cite{shimura2006quadratic}. In particular, we consider $\xi := \textup{diag}(1, \ell) \in \Lambda^+\left[\ell, 2^{-1}\mathbb{Z}\right]$. We define
\begin{equation}\label{w space}
    W:= \{ u \in U \,\,|\,\, \psi(u,\xi) = 0\},
\end{equation}
and put $H(\xi) := \textup{SO}^{\psi}(W)$, where we understand that $\psi$ is restricted to $W$.\\

We will now introduce localised and adelised groups and lattices, as usual. The reader can refer to \cite[Chapter 8]{Shimura_Euler_Product} to recall the necessary definitions.\\

We have the inclusion of quadratic spaces $(W,\psi) \subset (U,\psi)$, hence the embeddings $A^+(W) \subset A^+(U)$, and $G^+(W) \subset G^+(U)$. For each rational prime $p$, let
\begin{equation}\label{c in u space}
    C_p := \{ g \in \textup{SO}^{\psi}(U)_p \,\,| \,\, g \Lambda_{p}  = \Lambda_{p}\},
\end{equation}
\begin{equation}\label{j}
    J_p := \{ \gamma \in G^+(U)_p \,\,|\,\, \tau(\gamma) \in C_p, \,\, \nu(\gamma) \in \mathbb{Z}_p^{\times}\},
\end{equation}\\
and put $J := G^+(U)_{\infty} \prod_p J_p$ and $C := \textup{SO}^{\psi}(U)_{\infty}\prod_pC_p$.\\

Shimura's work in \cite{shimura2006quadratic} gives the following Proposition.

\begin{proposition}\label{finite-quotients}
There is a bijection
\[
\Lambda^+\left[\ell, 2^{-1}\mathbb{Z}\right]/ \textup{SL}_2(\mathcal{O}_K) \longleftrightarrow G^{+}(W)_{\mathbb{Q}}\backslash G^{+}(W)_{\mathbb{A}}/(G^{+}(W)_{\mathbb{A}} \cap J).
\]    
\end{proposition}
In particular, the set on the right-hand side is finite. 
\begin{proof} 
First of all, there is an equivalence of the sets $\Lambda^+\left[\ell, 2^{-1}\mathbb{Z}\right]/ \textup{SL}_2(\mathcal{O}_K)$ and $\Lambda\left[\ell, 2^{-1}\mathbb{Z}\right]/ \Delta$, where $\Delta := \{\alpha \in \textup{GL}_2(\mathcal{O}_K) \mid \det (\alpha) = \pm 1\}$, as is shown in \cite[p. 500]{shimura2006quadratic}. Moreover, from the proof of \cite[Theorem 2.10]{shimura2006quadratic}, the set $\Lambda[\ell, 2^{-1}\mathbb{Z}]/\Delta$ is in bijection with $G^{+}(W)_{\mathbb{Q}}\backslash G^{+}(W)_{\mathbb{A}}/(G^{+}(W)_{\mathbb{A}} \cap J)$, because the class number of $\mathbb{Q}$ is $1$.
\end{proof}
Therefore, Lemma \ref{claim on action of discriminant kernel} follows from the above discussion and this Proposition.\\

In order to now understand the Dirichlet series $D_{F,\xi_i, \mathcal{P}}(s)$, we will work back in the quadratic space $(V_0, \phi_0)$. Note that $f_{\omega}$ of Lemma \ref{isomorphism} induces an isomorphism between the Clifford algebras $A(U)$ and $A(V_0)$, which then gives an isomorphism of the Clifford groups $G^{+}(U)$ and $G^{+}(V_0)$. Moreover, $g_{\omega}$ of \eqref{g_omega} is naturally extended to an isomorphism between $\textup{SO}^{\psi}(U)_{\mathbb{A}}$ and $\textup{SO}^{\phi_0}(V_0)_{\mathbb{A}}$. We set 
\begin{equation*}
    \widetilde{W}:= f_{\omega}(W) =  \{x \in V_0 \mid \phi_0(x, f_{\omega}(\xi)) = 0\},
\end{equation*}
\begin{equation*}
    \widetilde{H}(\xi) := g_{\omega}(H(\xi)) = \textup{SO}^{\phi_0}(\widetilde{W}),
\end{equation*}
\begin{equation*}
    \widetilde{C}_p := g_{\omega}(C_p) = \{g \in \textup{SO}^{\phi_0}(V_0)_p \mid gL_{0,p} = L_{0,p}\},
\end{equation*}
\begin{equation*}
    \widetilde{J}_p := f_{\omega}(J_p) = \{\gamma \in G^{+}(V_0)_p \mid \tau(\gamma) \in \widetilde{C}_p, \ \nu(\gamma) \in \mathbb{Z}_p^{\times}\}, 
\end{equation*}
\\
and put $\widetilde{J} := G^+(V_0)_{\infty} \prod_p \widetilde{J}_p$ and $\widetilde{C} := \textup{SO}^{\phi_0}(V_0)_{\infty}\prod_p\widetilde{C}_p$.\\

We now have the following Proposition.
\begin{proposition}\label{commutative diagram 1}
    We work in the quadratic space $(V_0, \phi_0)$. We have the following commutative diagram, where the vertical arrows are bijections.
    \begin{equation*}
    \begin{tikzcd}
    L_0[\ell, 2^{-1}\mathbb{Z}]/\widehat{\Gamma}(L_0) \arrow{r}{\iota} & L_0[\ell, 2^{-1}\mathbb{Z}]/\Gamma(L_0) \\%
     G^{+}(\widetilde{W})_{\mathbb{Q}}\backslash G^{+}(\widetilde{W})_{\mathbb{A}}/(G^{+}(\widetilde{W})_{\mathbb{A}} \cap \widetilde{J})  \arrow{u}{\kappa} \arrow{r}{\tau}& \widetilde{H}(\xi)_{\mathbb{Q}}\backslash \widetilde{H}(\xi)_{\mathbb{A}}/(\widetilde{H}(\xi)_{\mathbb{A}} \cap \widetilde{C}) \arrow{u}{\lambda}
    \end{tikzcd}
    \end{equation*}
    Here, $\iota$ maps $\widehat{\Gamma}(L_0)v \longmapsto \Gamma(L_0)v$ for all $v \in L_0[\ell, 2^{-1}\mathbb{Z}]$ and $\kappa, \lambda$ are the bijections defined by Shimura in \cite[Theorem 1.6, (ii)]{shimura2006quadratic} and \cite[Theorem 2.2, (ii)]{shimura_diophantine_2006}, respectively.
\end{proposition}

\begin{proof} 
    Let us first recall how $\kappa$ and $\lambda$ are defined. \\
    
    From \cite[Theorem 1.6, (ii)]{shimura2006quadratic}, given $\epsilon \in G^{+}(\widetilde{W})_{\mathbb{A}}$, we take $\alpha \in G^{+}(V_0)$ such that $\epsilon \in \alpha\widetilde{J}$ and set $\kappa(\epsilon) := \widehat{\Gamma}(L_0) \tau(\alpha)^{-1} \xi$.\\
    
    Similarly, from \cite[Theorem 2.2, (ii)]{shimura_diophantine_2006}, given $\epsilon' \in \widetilde{H}(\xi)_{\mathbb{A}}$, we take $\alpha' \in \textup{SO}^{\phi_0}(V_0)$ such that $\epsilon' \in \alpha'\widetilde{C} $ and set $\lambda(\epsilon') := \Gamma(L_0) (\alpha')^{-1} \xi$. \\

    Let now $\epsilon \in G^{+}(\widetilde{W})_{\mathbb{A}}$. On the one hand, we have $\lambda (\tau(\epsilon)) = \Gamma(L_0)\alpha^{-1} \xi$, where $\alpha \in \textup{SO}^{\phi_0}(V_0)$ such that $\tau(\epsilon) \in \alpha \widetilde{C}$. On the other hand, $\iota (\kappa (\epsilon)) = \iota \left(\widehat{\Gamma}(L_0) \tau(\beta)^{-1} \xi\right) = \Gamma(L_0)\tau(\beta)^{-1}\xi$, where $\beta \in G^{+}(V_0)$ such that $\epsilon \in \beta \widetilde{J}$. \\

    We now write $\tau(\epsilon) = \alpha c$, with $c \in \widetilde{C}$ and $\epsilon = \beta j$, with $j \in \widetilde{J}$. Then $\tau(\epsilon) = \tau(\beta)\tau(j)$ and so $\tau(\beta)^{-1}\alpha = \tau(j)c^{-1} \in \widetilde{C}$. Hence $\tau(\beta)^{-1}\alpha \in \textup{SO}^{\phi_0}(V_0) \cap \widetilde{C} = \Gamma(L_0)$, i.e. $\tau(\beta)^{-1} \in \Gamma(L_0)\alpha^{-1}$. It then follows that $\lambda(\tau (\epsilon)) = \iota (\kappa (\epsilon))$, as wanted.
\end{proof}
We are now ready to deal with $D_{\mathcal{Q}}(F,\xi_i;s)$. This is done in \cite[Section 8]{psyroukis_orthogonal} but we need to modify it slightly. \\

We observe that for each $i=1,\cdots, h_1$, the series $D_{\mathcal{Q}}(F,\xi_i;s)$ depends only on the $\Gamma^{+}(L_0)$-equivalence class of $\xi_i$ (because of the invariance of $F$ under $\Gamma_S$, see also \cite[p. 26]{eisenstein_thesis}). Since $\widehat{\Gamma}^+(L_0) \leq \Gamma^{+}(L_0)$, we may think the map $\iota$ of Proposition \ref{commutative diagram 1} as
\[
\iota: L_{0}^{+}\left[\ell, 2^{-1}\mathbb{Z}\right]/\widehat{\Gamma}^+(L_0) \longrightarrow L_{0}^{+}\left[\ell, 2^{-1}\mathbb{Z}\right]/\Gamma^+(L_0).
\]
Therefore, we can write $D(\bm{F}, \bm{P};s)$ as
\begin{equation}\label{expression with iota}
    D(\bm{F}, \bm{P};s) = \sum_{i=1}^{h_{1}}\left(\sum_{\mathfrak{m} \in \textup{Cl}(K)}\zeta_{\mf{m}, \mathcal{Q}}(s-k+3)\zeta_{\mathcal{Q}}(\xi_{\mathfrak{m},i}; s-k+3) \right)D_{\mathcal{Q}}(F,\iota(\xi_i);s),
\end{equation}
Let $\{u_i\}_{i=1}^{h_2}$ denote a set of representatives for $\widetilde{H}(\xi)_{\mathbb{Q}}\backslash \widetilde{H}(\xi)_{\mathbb{A}}/(\widetilde{H}(\xi)_{\mathbb{A}}\cap \widetilde{C})$ such that $u_{i, \infty}=1$ for all $i$. Note that from Proposition \ref{commutative diagram 1}, $h_2 := \#\left(L_0\left[\ell, 2^{-1}\mathbb{Z}\right]/\Gamma(L_0)\right)$ and $h_1 \geq h_2$. Without loss of generality, we re-arrange the $\xi_i's$, so that the $u_i$'s correspond to $\{\xi_i\}_{i=1}^{h_2}$ under the map $\lambda$ of Proposition \ref{commutative diagram 1}.\\

Let now $V(\xi)$ denote the space of $\mathbb{C}$-valued functions on $\widetilde{H}(\xi)_{\mathbb{A}}$, which are left $\widetilde{H}(\xi)_{\mathbb{Q}}$ and right $(\widetilde{H}(\xi)_{\mathbb{A}} \cap \widetilde{C})$-invariant. For each prime $p$ and $i=1,\cdots, h_2$, let 
\begin{equation*}
    M(u_i;\xi)_p := \widetilde{H}(\xi)_p \cap u_i\widetilde{C}u_i^{-1} \textup{ and } M(u_i;\xi)_{f} := \prod_{p}M(u_i;\xi)_p.
\end{equation*}
Define then $e(\xi)_i := \#\{\widetilde{H}(\xi)_{\mathbb{Q}}\cap M(u_i;\xi)_{f}\}$ for $1\leq i\leq h_2$, and $\displaystyle{\mu(\xi) := \sum_{i=1}^{h_2}e(\xi)^{-1}_{i}}$. Furthermore, for any $f \in V(\xi)$, we set
\begin{equation}\label{A_f s}
    A_{f} := \mu(\xi)^{-1}\sum_{i=1}^{h_2} \frac{\overline{f}(u_i)}{e(\xi)_i}A(\xi_i),
\end{equation}
where $A(-)$ denote the Fourier coefficients of $F = F_{\mathcal{O}_K}$, as we have defined them just before Lemma \ref{fourier-coeffs-different-m}. We now let $\{f_{j}\}_{j=1}^{h_2}$ be an orthonormal basis for the space $V(\xi)$, consisting of Hecke eigenforms for the Hecke pairs $(\widetilde{H}(\xi)_p, \widetilde{H}(\xi)_p \cap \widetilde{C})$ for all $p \not \in \mathcal{Q}$. Let also $L_{\mathcal{Q}}(-, s)$ denote the standard $L$-function attached to $F$ or any $f_j$, as defined in \cite[Section 4-1]{sugano} (omitting the $p$-factors for $p \in \mathcal{Q}$). Let us now enlarge the set of primes $\mathcal{Q}$ to include the primes described in \cite[Theorem 8.1]{psyroukis_orthogonal}. 
Following the discussion right after \cite[Theorem 8.1]{psyroukis_orthogonal}, and  taking into account the different class numbers here, we obtain for each $i=1, \cdots, h_1$,
\begin{multline}\label{expression}
D_{\mathcal{Q}}(F,\iota(\xi_i);s)= \zeta_{\mathcal{Q}}\left(2s-2k+4\right)^{-1} L_{\mathcal{Q}}\left(F; s-k+2\right)\times\\\times\sum_{j=1}^{h_2}f_j(\lambda^{-1}(\iota(\xi_i)))L_{\mathcal{Q}}\left(\overline{f_j};s-k+5/2\right)^{-1}A_{f_{j}},
\end{multline}
using the notation of Proposition \ref{commutative diagram 1}. This is \cite[(16)]{psyroukis_orthogonal} tailored to our situation.
\\

If now $\{\zeta_i\}_{i=1}^{h_1}$ are representatives for $G^{+}(\widetilde{W})_{\mathbb{Q}}\backslash G^{+}(\widetilde{W})_{\mathbb{A}}/(G^{+}(\widetilde{W})_{\mathbb{A}} \cap \widetilde{J})$, such that $\kappa^{-1}(\xi_i) = \zeta_i$, we have again from Proposition \ref{commutative diagram 1}, that $\lambda^{-1}(\iota(\xi_i)) = \tau(\zeta_i)$. Therefore, we can write \eqref{expression} as 
\begin{multline}\label{expression3}
D_{\mathcal{Q}}(F,\iota(\xi_i);s)= \zeta_{\mathcal{Q}}\left(2s-2k+4\right)^{-1} L_{\mathcal{Q}}\left(F; s-k+2\right)\times\\\times \sum_{j=1}^{h_2}f_j(\tau(\zeta_i))L_{\mathcal{Q}}\left(\overline{f_j};s-k+5/2\right)^{-1}A_{f_{j}}.
\end{multline}
\section{Quaternion Algebras}\label{quaternion algebras}
Now that we have an understanding of the Dirichlet series $D_{\mathcal{Q}}(F,\iota(\xi_i);s)$, our next step is to understand the zeta functions $\zeta_{\mathcal{Q}}(\xi_{\mf{m},i};s)$ of \eqref{zeta xi i m definition}. Our goal is to relate them to a partial zeta function in a quaternion algebra. To that end, we need several preparations.  The quadratic space $(W,\psi)$ is three-dimensional and, as a result, its even Clifford algebra can be identified with a quaternion algebra, which in our case is a division algebra. Before making this explicit, let us recall some basic facts on quaternion algebras. We mainly follow \cite{voight}. \\

Let $B$ be a quaternion algebra over $\mathbb{Q}$ with a standard involution $a \longmapsto \overline {a}$. A $\mathbb{Z}$-order $\mathcal{O} \subset B$ is a $\mathbb{Z}$-lattice that is also a subring of $\mathbb{Z}$. For any $\mathbb{Z}$-lattice $I \subset B$, we define
\begin{equation*}
    \mathcal{O}_{L}(I) := \{\alpha \in B \mid \alpha I \subseteq I\}, \,\, \mathcal{O}_{R}(I):= \{\alpha \in B \mid I \alpha \subseteq \alpha\}.
\end{equation*}
These are orders in $B$ (\cite[Lemma 10.2.7]{voight}), and are called the left and right order of $I$, respectively. We say $I$ is right-invertible if there exists a lattice $I' \subset B$ such that $II' = \mathcal{O}_L(I)$. Similarly, we define the notion of a left-invertible lattice, and we say $I$ is invertible if there is a two-sided inverse $I' \subset B$, so that
\begin{equation*}
    II' = \mathcal{O}_L(I) = \mathcal{O}_{R}(I'),\,\, I'I = \mathcal{O}_L(I') = \mathcal{O}_{R}(I). 
\end{equation*}
Now, if $\alpha \in B$, we define its reduced norm as $\textup{nrd}(\alpha):= \alpha\overline{\alpha}$. We define the reduced norm $\textup{nrd}(I)$ of a $\mathbb{Z}$-lattice $I$, as the $\mathbb{Z}$-submodule of $\mathbb{Q}$ generated by $\{\textup{nrd}(\alpha) \mid \alpha \in I\}$. This is a fractional ideal of $\mathbb{Q}$ (\cite[Lemma 16.3.2]{voight}).\\

Given now a $\mathbb{Z}$-order $\mathcal{O} \subset B$, we say that $I$ is a left-fractional $\mathcal{O}$-ideal if $\mathcal{O} \subseteq \mathcal{O}_L(I)$; similarly on the right.\\

Two $\mathbb{Z}$-lattices $I$ and $J$ are said to be in the same right class and write $I \sim_{R} J$, if $\exists \alpha \in B^{\times}$ such that $\alpha I = J$; similarly for the left. The relation $\sim_{R}$ defines an equivalence relation on the set of $\mathbb{Z}$-lattices in $B$, and we denote the equivalence class of a $\mathbb{Z}$-lattice $I$ by $[I]_{R}$. The (right) class set of $\mathcal{O}$ is
\begin{equation*}
    \textup{Cls}_{R}(\mathcal{O}) := \left\{[I]_{R} : I \textup{ an invertible right fractional-}\mathcal{O} \textup{ ideal}\right\}.
\end{equation*}
We say two orders $\mathcal{O}, \mathcal{O'}$ are of the same type, if $\exists \alpha \in B^{\times}$ such that $\mathcal{O}' = \alpha^{-1}\mathcal{O}\alpha$. Note that this is true if and only if $\mathcal{O}$ and $\mathcal{O}'$ are isomorphic as $\mathbb{Z}$-algebras. Moreover, we say that $\mathcal{O}$ and $\mathcal{O}'$ are locally isomorphic if the localisations $\mathcal{O}_p$ and $\mathcal{O}'_p$ are of the same type for all rational primes $p$. Here, we define the localisation of a $\mathbb{Z}$-lattice in $B$ as usual (see, for example, \cite[Section 9.4]{voight}).\\

We define the genus of a $\mathbb{Z}$-order $\mathcal{O} \subset B$ to be the set of all $\mathbb{Z}$-orders in $B$ which are locally isomorphic to $\mathcal{O}$. The type set of $\mathcal{O}$ is the set of isomorphism classes of orders in the genus of $\mathcal{O}$. If \begin{equation*}
    X := B_{\infty}^{\times}\prod_{p}\mathcal{O}_p^{\times},
\end{equation*}
we can then identify the $B^{\times}$-classes in the genus of $\mathcal{O}$ with the set $B^{\times} \backslash B^{\times}_{\mathbb{A}}/ X$ (\cite[Section 1.8]{shimura_diophantine_2006}). Moreover, if $N_{B}(X) := \{\alpha \in B_{\mathbb{A}}^{\times} \mid \alpha \mathcal{O}\alpha^{-1} = \mathcal{O}\}$, the type set of $\mathcal{O}$ is identified with $B^{\times} \backslash B^{\times}_{\mathbb{A}}/ N_B(X)$ (see \cite[(2.9.2)]{thesis_hein}).\\

Let us now define the quaternion algebra $B:= K + \epsilon K$ with $\epsilon^2 = -\ell$ and $\epsilon a = \overline{a}\epsilon$ for all $a \in K$. Let also $\mathfrak{o} := \mathcal{O}_K + \epsilon \mathcal{O}_K$ denote an order in $B$. From the proof of \cite[Theorem 2.10]{shimura2006quadratic}, we have $\Xi(A^{+}(W)) = B$ and $\Xi(G^{+}(W)) = B^{\times}$, where $\Xi$ is the map of Proposition \ref{Xi proposition}. Moreover, from the same Theorem, $\Xi(G^{+}(W)_p \cap J_p) = \mathfrak{o}_p^{\times}$ for all $p$. \\

In this case, $X := B_{\infty}^{\times}\prod_{p}\mathfrak{o}_p^{\times}$. We can then identify $G^{+}(W)_{\mathbb{Q}}\backslash G^{+}(W)_{\mathbb{A}}/(G^{+}(W)_{\mathbb{A}} \cap J)$ with $B^{\times} \backslash B^{\times}_{\mathbb{A}}/ X$ via the map $\Xi$.

Because of Proposition \ref{commutative diagram 1}, we also want to identify the the quotient $H(\xi)_{\mathbb{Q}}\backslash H(\xi)_{\mathbb{A}}/(H(\xi)_{\mathbb{A}} \cap C)$ in the quaternion algebra $B$. To this end, we will turn back to the quadratic space $(V_0, \phi_0)$. We have the following Proposition.
\begin{proposition}\label{stabilisers of lattices}
    Let $D := \{x \in \widetilde{H}(\xi)_\mathbb{A}\mid x(L_0 \cap \widetilde{W}) = L_0 \cap \widetilde{W}\}$. Then,  $\widetilde{H}(\xi)_{\mathbb{A}} \cap \widetilde{C} = D$, at least in the following cases:
    \begin{itemize}
        \item $m \equiv 3 \pmod 4$, $\ell$ is square-free, and $\textup{gcd}(\ell,m)=1$.
        \item $m \equiv 1 \pmod 4$, $\ell$ is square-free, $\textup{gcd}(\ell, m)=1$, and $\ell \equiv 1 \pmod4$.
    \end{itemize}
\end{proposition}
\begin{proof}
    Because the lattice $L_0$ is taken maximal in $(V_0, \phi_0)$, we have from \cite[Paragraph 6.3]{shimura_classification_quadratic_forms} that $[L_0^{*}/L_0] = \det (S_0) \mathbb{Z}$. Moreover, as in \cite[Lemma 9.1]{psyroukis_orthogonal}, $\widetilde{H}(\xi)$ is represented via the matrix
    \begin{equation*}
        T:= \frac{1}{2}\m{-2\ell \\ & -S},
    \end{equation*}
    with respect to a basis of $M:=\mathbb{Z}^3$. We will now show that under our assumptions on $\ell$ and $m$, $M$ is a maximal $\mathbb{Z}$--lattice. \\
    
    If $m \equiv 3 \pmod 4$, then $\det (2T) = 2\ell \det (S) = 2\ell m$ and because of the conditions on $\ell$ and $m$, $\det (2T)/2$ is square-free. Hence, from \cite[Proposition 1.6.12]{eisenstein_thesis}, the lattice $\mathbb{Z}^3$ is maximal with respect to $(\widetilde{W},\phi_0)$.\\

    If $m \equiv 1 \pmod4$, we will use \cite[Lemma 1.6.5 (ii)]{eisenstein_thesis}. Assume we pick $x = (x_1, x_2, x_3) \in M^{*}$ such that $T[x] \in \mathbb{Z}$. We will show that $x \in M$, which will give us that $M$ is $\mathbb{Z}$--maximal. We write $x_1 = y_1/(2\ell),\ x_2 = y_2/2,\ x_3 = y_3/(2m)$, where $y_i \in \mathbb{Z}$. Then, the condition $T[x] \in \mathbb{Z}$ is equivalent to $4\ell m \mid my_1^2+\ell my_2^2+\ell y_3^2$. Now, by the assumptions, $\ell$ and $4m$ are coprime. Hence, we can look at each divisibility separately.
    \begin{itemize}
        \item $\ell \mid my_1^2+\ell my_2^2+\ell y_3^2$. This implies $\ell \mid my_1^2$, hence $\ell \mid y_1$, because $\gcd (\ell,m)=1$ and $\ell$ square--free.
        \item $4m \mid my_1^2+\ell my_2^2+\ell y_3^2$. First, this implies $m \mid \ell y_3^2$, so $m \mid y_3$, because $\gcd (\ell,m)=1$ and $m$ square--free. By now writing $y_3 = mz_3$, the divisibility is equivalent to $4 \mid y_1^2+\ell y_2^2+\ell mz_3^2$. But since $\ell, m \equiv 1 \pmod 4$, we have $4 \mid y_1^2+y_2^2+z_3^2$, which implies all of them must be even.
    \end{itemize}
    These conditions imply that $x_i \in \mathbb{Z}$ for all $i=1,2,3$, as wanted.\\
    
    Now, again from \cite[Paragraph 6.3]{shimura_classification_quadratic_forms}, we have $[M^{*}/M] = \det (2T) \mathbb{Z} = 2\ell\det (S)\mathbb{Z}$. Since $\phi_0[f_{\omega}(\xi)] = \ell$ and $2\phi_0(f_{\omega}(\xi), L_0) = \mathbb{Z}$, we have $\phi_0[f_{\omega}(\xi)]\phi_0(f_{\omega}(\xi), L_0)^{-2} = 4\ell\mathbb{Z}$ and this equals $2[M^{*}/M][L_0^{*}/L_0]^{-1}$. Hence, from \cite[Corollary 4.3]{murata_preprint}, we obtain $L_0 \cap \widetilde{W}$ is $\mathbb{Z}$-maximal in $\widetilde{W}$. \\

    This now means that $(L_0 \cap \widetilde{W})_{p}$ is maximal in $\widetilde{W}_p$ for all primes $p$. Hence, from \cite[Proposition 11.12, (iv)]{shimura2004arithmetic}, we obtain $\widetilde{H}(\xi)_{p} \cap \widetilde{C} = D_p$ for all primes $p$, as $t_p$ is always even in our case. Here, $D_p := D \cap \widetilde{H}(\xi)_p$. Hence, similarly to \cite[Proposition 9.3]{psyroukis_orthogonal}, we obtain $D = \widetilde{H}(\xi)_{\mathbb{A}} \cap \widetilde{C}$, as required.
\end{proof}
\begin{remark}\label{m 2 mod 4}
    We believe that the case $m \equiv 2 \pmod 4$ can also be included, for appropriate choices of $\ell$. We have obtained computational evidence which suggests this, and we hope to relax this restriction in the near future.
\end{remark}

We now have a surjective map (see \cite[Lemma 17.4.13]{voight})
\begin{align*}
    \rho: \textup{Cls} \mathfrak{o} &\longrightarrow \textup{Typ} \mathfrak o\\
    [I] &\longmapsto [\mathcal{O}_L(I)].
\end{align*}
Because of the identifications we mentioned at the beginning of the Section, we can also think of this map as $\rho: B^{\times} \backslash B^{\times}_{\mathbb{A}}/ X \longrightarrow B^{\times} \backslash B^{\times}_{\mathbb{A}}/ N_B(X)$. We then have the following Proposition.
\begin{proposition}\label{diagram quaternion algebra}
Assume $m$ and $\ell$ are as in Proposition \ref{stabilisers of lattices}. With the notation as above, we have the following commutative diagram, where the horizontal arrows are bijections:
\[ \begin{tikzcd}
G^{+}(W)_{\mathbb{Q}}\backslash G^{+}(W)_{\mathbb{A}}/(G^{+}(W)_{\mathbb{A}} \cap J)  \arrow{d}{\tau}\arrow{r}{\Xi}   & B^{\times} \backslash B^{\times}_{\mathbb{A}}/ X  \arrow[swap]{d}{\rho}\\
H(\xi)_{\mathbb{Q}}\backslash H(\xi)_{\mathbb{A}}/(H(\xi)_{\mathbb{A}} \cap C) \arrow{r}{A^{+}}& B^{\times} \backslash B^{\times}_{\mathbb{A}}/ N_B(X)
\end{tikzcd}
\]
Here, for any lattice $N$, $A^{+}(N)$ is given in as in Section \ref{preliminaries}.
\end{proposition}
\begin{proof}
    
    We recall that we can identify $G^{+}(W)$ with $B^{\times}$ through the map $\Xi$. Moreover, through $\Xi$, we can identify $G^{+}(W)_{\mathbb{Q}}\backslash G^{+}(W)_{\mathbb{A}}/(G^{+}(W)_{\mathbb{A}} \cap J)$ with $B^{\times} \backslash B^{\times}_{\mathbb{A}}/ X$.\\
    
    As we have mentioned, this last set is in bijection with the classes in the right genus of $\mathfrak{o}$ (i.e., all the right invertible $\mathfrak{o}$ ideals which are locally isometric to $\mathfrak{o}$). Therefore, if we start with an element $a \in G^{+}(W)_{\mathbb{A}}$, its class in $G^{+}(W)_{\mathbb{Q}}\backslash G^{+}(W)_{\mathbb{A}}/(G^{+}(W)_{\mathbb{A}} \cap J)$ will correspond to the class $[a \mathfrak{o}]$ in $B^{\times} \backslash B^{\times}_{\mathbb{A}}/ X$, where we view $a \in B_{\mathbb{A}}^{\times}$ via $\Xi$. \\
    
    Moreover, the set $B^{\times} \backslash B^{\times}_{\mathbb{A}}/ N_B(X)$ is in bijection with the classes of orders in $B$ which are locally conjugate to $\mathfrak{o}$, i.e., of the same type. Therefore, the map $\rho$ maps $[a\mathfrak{o}] \longmapsto [a\mathfrak{o}a^{-1}]$. \\

    On the other hand, the map $\tau$ will send the class of $a \in G^{+}(W)_{\mathbb{A}}$ to the class of the lattice $[\tau(a)(L\cap W)]$ in $H(\xi)_{\mathbb{Q}}\backslash H(\xi)_{\mathbb{A}}/(H(\xi)_{\mathbb{A}} \cap C)$. Then, via the map $N \longmapsto A^{+}(N)$, we have that 
    \begin{equation*}
        A^{+}(\tau(a)(L\cap W)) = A^{+}(a(L\cap W)a^{-1}) = aA^{+}(L\cap W)a^{-1} = a\mathfrak{o}a^{-1},
    \end{equation*}
    because $\mathfrak{o} = A^{+}(L\cap W)$. This shows that the diagram is indeed commutative.
\end{proof}

\section{Representation Numbers of Binary Hermitian Forms}

The aim of this Section is to relate the expressions $n(\xi_{\mf{m},i};d)$ of equation \eqref{n-set} with $\xi_{\mf{m},i} \in L_{\mf{m},0}^{+} [\ell,2^{-1}\mathbb{Z}]/\widehat{\Gamma}(L_{\mf{m},0})$ to representation numbers of binary Hermitian forms.\\

We recall the map $f_{\omega}: (U,\psi) \longrightarrow (V_0,\phi_0)$ of Lemma \ref{isomorphism}. Since $h_{\mathfrak{m}} \in G_{\mathbb{Q}}^{*}$ and the map $f_{\omega}$ is an isometry, we have that there is an element $ \beta_{\mathfrak{m}} \in \textup{SO}^{\psi}(U)$ such that $g_{\omega}(\beta_{\mathfrak{m}}) = h_{\mathfrak{m}}$, where $g_{\omega}$ is the map \eqref{g_omega}. This means $f_{\omega}\beta_{\mathfrak{{m}}}f_{\omega}^{-1} =h_{\mathfrak{m}}$, or equivalently $f_{\omega}^{-1}h_{\mf{m}}^{-1}v = \beta_{\mathfrak{m}}^{-1}f_{\omega}^{-1}v$ for all $v \in V_0$. Now, since the map $\tau$ is surjective, there is an element $\delta_{\mathfrak{m}} \in G^{+}(U)$ such that $\tau(\delta_{\mathfrak{m}}) = \beta_{\mathfrak{m}}$. From the definition of the spinor norm, we have $\nu(\delta_{\mf{m}})(\mathbb{Q}^{\times})^2 = \sigma(\beta_{\mf{m}}) = \sigma(h_{\mf{m}}) = N(\mf{m})(\mathbb{Q}^{\times})^2$, hence $\nu(\delta_{\mf{m}}) = N(\mf{m})a^2$, with $a \in \mathbb{Q}^{\times}$. Let then $\delta_{\mf{m}}':= a^{-1} \delta_{\mf{m}} \in G^{+}(U)$. Then, we still have $\tau(\delta_{\mf{m}}') = \beta_{\mf{m}}$

and $\nu(\delta_{\mf{m}}') = N(\mf{m})$. If now $\alpha_{\mathfrak{m}} := \Xi(\delta_\mathfrak{m}')$, where $\Xi$ is the map of Proposition \ref{Xi proposition}, with $\alpha_{\mathfrak{m}} \in \mathcal{G}$, then $\Xi((\delta_\mathfrak{m}')^{-1}) = \alpha_{\mathfrak{m}}^{-1}$, and we have
\begin{equation*}
    \beta_{\mathfrak{m}}^{-1} u = \tau((\delta_{\mathfrak{m}}')^{-1}) u = \det (\alpha_{\mathfrak{m}}) \alpha_{\mathfrak{m}}^{-1} u \overline{\alpha_{\mathfrak{m}}}^{-t},
\end{equation*}
for all $u \in U$, using Proposition \ref{Xi proposition}. We have the following Lemma.
\begin{lemma}
    The matrix $\alpha_{\mf{m}}$ above satisfies $\det(\alpha_\mf{m}) = N(\mf{m})$ and 
    \begin{equation*}
        \alpha_{\mf{m}}^{-1} \in \m{\overline{\mf{m}}^{-1} & \mf{m}^{-1} \\ \overline{\mf{m}}^{-1} & \mf{m}^{-1}}.
    \end{equation*}
\end{lemma}
\begin{proof}
The condition on the determinant follows from the fact that $\det (\alpha_{\mf{m}}) = \nu(\delta'_{\mf{m}}) = N(\mf{m})$ where $\delta'_{\mf{m}} \in G^{+}(U)$ is defined above. We now show the claim regarding the entries of the matrix.  \newline
We write $\alpha_{\mathfrak{m}}^{-1} = \m{u & v \\ r & s} \in \textup{GL}_2(K)$.
Since 
\begin{equation*}
    \Lambda_{\mf{m},0} := f^{-1}_{\omega}(L_{\mf{m},0}) = \m{ \mathbb{Z} & L_{\mf{m}} \\ L_{\mf{m}}& \mathbb{Z}},
\end{equation*}
and $L_0 = h_\mf{m}^{-1}L_{\mf{m},0}$, by translating the actions as above, this means $\det (\alpha_{\mf{m}})\alpha_{\mf{m}}^{-1}\Lambda_{\mf{m},0}\overline{\alpha_{m}}^{-t} = \Lambda$. Hence, we must have
\[
\m{u & v \\ r & s} \m{N(\mathfrak{m}) & 0 \\ 0 & 0}\m{\bar{u} & \bar{r} \\ \bar{v} & \bar{s}} = \m{N(u)N(\mathfrak{m}) & u \bar{r}N(\mathfrak{m}) \\ \bar{u} r N(\mathfrak{m}) & N(r)N(\mathfrak{m}) } \in \m{\mathbb{Z} & \mathcal{O}_K \\ \mathcal{O}_K & \mathbb{Z}}.
\]
That is $N(u) \in N(\mathfrak{m})^{-1} \mathbb{Z}$, $u\bar{r} \in N(\mathfrak{m})^{-1} \mathcal{O}_K$ and $N(r) \in N(\mathfrak{m})^{-1} \mathbb{Z}$. \newline

A similar calculation with the matrix $\m{ 0 & 0 \\ 0 & N(\mathfrak{m})}$ gives $N(v),N(s) \in N(\mathfrak{m})^{-1}\mathbb{Z}$ and $v\bar{s} \in N(\mathfrak{m})^{-1} \mathcal{O}_K$

Moreover we calculate (for an $h \in L_{\mf{m}} = \frac{\bar{\mathfrak{m}}^2}{N(\mathfrak{m})}$),
\[
\m{u & v \\ r & s} \m{0 & N(\mathfrak{m})h \\ N(\mathfrak{m})\bar{h} & 0}\m{\bar{u} & \bar{r} \\ \bar{v} & \bar{s}}= N(\mathfrak{m})\m{v \bar{u} \bar{h} + u \bar{v}h & rv\bar{h} + u \bar{s} h \\ * & \bar{r}s \bar{h} + r \bar{s} h} \in \m{\mathbb{Z} & \mathcal{O}_K \\ \mathcal{O}_K & \mathbb{Z}}
\]
That is, $N(\mathfrak{m})Tr(\bar{u}v \bar{h}) \in \mathbb{Z}$, $N(\mathfrak{m}) Tr(r\bar{s} h) \in \mathbb{Z}$
Equivalently, since $N(\mathfrak{m})h \in \bar{\mathfrak{m}}^2$ we have $\bar{u}v \in \mathfrak{d}^{-1}_K \mathfrak{m}^{-2}$ and $r\bar{s} \in \mathfrak{d}_K^{-1} \bar{\mathfrak{m}}^{-2}$. Since we are taking that $\mathfrak{d}_K$ and $\mathfrak{m}$ share no common factors, from the conditions $N(u), N(v) \in N(\mathfrak{m})^{-1} \mathbb{Z}$ and $\bar{u}v \in \mathfrak{d}^{-1}_K \mathfrak{m}^{-2}$ we have that $u \in \bar{\mathfrak{m}}^{-1}$ and $v \in \mathfrak{m}^{-1}$. Similarly we conclude that $r \in \bar{\mathfrak{m}}^{-1}$ and $s \in \mathfrak{m}^{-1}$. That is, the matrix 
\begin{equation*}
    \alpha^{-1}_{\mathfrak{m}} \in \m{\bar{\mathfrak{m}}^{-1} & \mathfrak{m}^{-1} \\ \bar{\mathfrak{m}}^{-1} & \mathfrak{m}^{-1}}.\qedhere
\end{equation*}
\end{proof}

Back to the expressions $n(\xi_{\mf{m},i};d)$ of equation \eqref{n-set}, we recall
\begin{multline*}
n(\xi_{\mathfrak{m},i};d) := \#\left\{s \in \mathbb{Z}^{2}/dS_{\mathfrak{m}}\mathbb{Z}^{2} \mid D_{\mf{m}} \equiv \frac{1}{2}q_{\mathfrak{m}}S_{\mathfrak{m}}^{-1}[s] \pmod{q_{\mathfrak{m}}d}, \right.\\\left.\left(\frac{\frac{1}{2}q_{\mathfrak{m}}S_{\mathfrak{m}}^{-1}[s]-D_{\mf{m}}}{q_{\mathfrak{m}}d}, A_{\mathfrak{m}}S_{\mathfrak{m}}^{-1}s, d\right)^{t} = \gamma \xi_{\mathfrak{m},i}, \gamma \in \widehat{\Gamma}^{+}(L_{\mathfrak{m},0})\right\},
\end{multline*}

where $D_{\mf{m}} = -\ell q_{\mf{m}}$, with $\ell$ a positive integer and $q_{\mf{m}}$ the level of the lattice $L_{\mf{m}}$.\newline

Using Lemma \ref{isomorphisms}, for any $\gamma \in \widehat{\Gamma}^{+}(L_{\mf{m},0})$, we may write $\gamma = h_{\mathfrak{m}}\gamma' h_{\mathfrak{m}}^{-1}$ with $\gamma' \in \widehat{\Gamma}^{+}(L_{0})$,  and so $\gamma \xi_{\mathfrak{m},i} = h_{\mathfrak{m}} \gamma' \xi_i$, since $\xi_{\mf{m},i} = h_{\mf{m}}\xi_i$. Therefore, the above can be written as
\begin{multline}\label{alternative n(xi m i)}
    n(\xi_{\mathfrak{m},i};d) = \#\left\{s \in \mathbb{Z}^{2}/dS_{\mathfrak{m}}\mathbb{Z}^{2} \mid D_{\mf{m}} \equiv \frac{1}{2}q_{\mathfrak{m}}S_{\mathfrak{m}}^{-1}[s] \pmod{q_{\mathfrak{m}}d}, \right.\\\left. h_{\mathfrak{m}}^{-1}\left(\frac{\frac{1}{2}q_{\mathfrak{m}}S_{\mathfrak{m}}^{-1}[s]-D_{\mf{m}}}{q_{\mathfrak{m}}d}, A_{\mathfrak{m}}S_{\mathfrak{m}}^{-1}s, d\right)^{t} = \gamma \xi_{i}, \ \gamma \in \widehat{\Gamma}^{+}(L_{0})\right\}.
\end{multline}

Putting together all the above, and with the matrix $\alpha_\mathfrak{m} \in \textup{GL}_2(K)$ with $\textup{det}(\alpha_{\mathfrak{m}}) = N(\mf{m})$ as above, we have that 
\begin{equation}\label{translation-hermitian-matrices}
f_{\omega}^{-1} \left( h_{\mathfrak{m}}^{-1}\left(\frac{\frac{1}{2}q_{\mathfrak{m}}S_{\mathfrak{m}}^{-1}[s]-D_{\mf{m{}}}}{q_{\mathfrak{m}}d}, A_{\mathfrak{m}}S_{\mathfrak{m}}^{-1}s, d\right)^{t}\right)
=\det(\alpha_{\mathfrak{m}})\alpha^{-1}_{\mathfrak{m}} \m{\frac{\frac{1}{2}q_{\mathfrak{m}}S_{\mathfrak{m}}^{-1}[s]-D_{\mf{m{}}}}{q_{\mathfrak{m}}d} & h \\ \bar{h} & d & } \overline{\alpha_{\mathfrak{m}}}^{-t},
\end{equation}
where $h \in K $ corresponding to the vector $A_{\mathfrak{m}}S_{\mathfrak{m}}^{-1}s$, that is $h=\theta(A_{\mf{m}}S_{\mf{m}}^{-1}s)$, where $\theta:V \longrightarrow K$ is the isometry \eqref{theta map}.\\

From Lemma \ref{equivalence of solution sets} and Proposition \ref{relation of SL_2 with discriminant kernel}, we have seen the identification
\[
\Lambda^+[\ell,2^{-1}\mathbb{Z}]/\textup{SL}_2(\mathcal{O}_K)\longleftrightarrow L_0^+[\ell,2^{-1}\mathbb{Z}]/\widehat{\Gamma}^+(L_0) 
\]
induced by the map $f_\omega$. We set $f_i := f_{\omega}^{-1} (\xi_i) \in \Lambda^+[\ell,2^{-1}\mathbb{Z}]$, for $i=1,\ldots,h_1$. We have the following Lemma.
\begin{lemma} \label{integrality of hermitian matrices}
    Let $\Lambda = M_2(\mathcal{O}_K) \cap U$. We have the following two statements:
    \begin{itemize}
        \item If $\ell \in \mathbb{Z}$, then $h \in \Lambda\left[\ell, 2^{-1}\mathbb{Z}\right] \implies h \in \Lambda$.
        \item If $\ell \in \mathbb{Z}$ square-free and $\gcd(\ell, 2m)=1$, then
        \begin{equation*}
            \{h \in \Lambda \mid \det(h)=\ell\} = \Lambda[\ell, \mathbb{Z}] \cup \Lambda\left[\ell, 2^{-1}\mathbb{Z}\right].
        \end{equation*}
    \end{itemize}
\end{lemma}
\begin{proof}
    The first statement follows from \cite[Lemma 6.2, (3)]{shimura2004arithmetic}, and the second statement is \cite[Lemma 4.5]{shimura2006quadratic}.
    \end{proof}

By this Lemma, we have in particular that $f_i \in \Lambda$ since $\ell \in \mathbb{Z}$. Using \eqref{translation-hermitian-matrices} and the isomorphism $\Psi$ of Proposition \ref{relation of SL_2 with discriminant kernel}, we then see that the equation 
\[
h_{\mathfrak{m}}^{-1}\left(\frac{\frac{1}{2}q_{\mathfrak{m}}S_{\mathfrak{m}}^{-1}[s]-D_{\mf{m}}}{q_{\mathfrak{m}}d}, A_{\mathfrak{m}}S_{\mathfrak{m}}^{-1}s, d\right)^{t} = \gamma \xi_{i}, \ \gamma \in \widehat{\Gamma}^{+}(L_{0})
\] in \eqref{alternative n(xi m i)} can be replaced with the condition that the matrix 
\[
\det(\alpha_{\mathfrak{m}})\alpha^{-1}_{\mathfrak{m}} \m{\frac{\frac{1}{2}q_{\mathfrak{m}}S_{\mathfrak{m}}^{-1}[s]-D_{\mf{m}}}{q_{\mathfrak{m}}d} & h \\ \bar{h} & d & } \overline{\alpha_{\mathfrak{m}}}^{-t} \sim f_i,
\]
where the notation $X \sim Y$ for two $2\times 2$ matrices $X$ and $Y$ means that there is a matrix $A \in \textup{SL}_2(\mathcal{O}_K)$ such that $AX\overline{A}^t = Y$. Let now
\begin{multline}
    \mathcal{N}(\xi_{\mathfrak{m},i};d) := \left\{s \in \mathbb{Z}^{2}/dS_{\mathfrak{m}}\mathbb{Z}^{2} \mid D \equiv \frac{1}{2}q_{\mathfrak{m}}S_{\mathfrak{m}}^{-1}[s] \pmod{q_{\mathfrak{m}}d}, \,\, \alpha^{-1}_{\mathfrak{m}} A_s\overline{\alpha_{\mathfrak{m}}}^{-t} \sim f_i\right\},
\end{multline}
where 
\begin{equation}\label{A s}
    A_s= \m{\frac{\frac{1}{2}q_{\mathfrak{m}}S_{\mathfrak{m}}^{-1}[s]-D}{q_{\mathfrak{m}}d}N(\mathfrak{m}) & h_s N(\mathfrak{m}) \\ \bar{h_s}N(\mathfrak{m}) & dN(\mathfrak{m}) & }, 
\end{equation}
and $h_s:=\theta(A_{\mf{m}}S_{\mf{m}}^{-1}s) \in K$. We then have $n(\xi_{\mathfrak{m},i};d) := \# \mathcal{N}(\xi_{\mathfrak{m},i};d)$.\\

For a $d \in \mathbb{N}$, $f_i \in \Lambda$ with $\textup{det}(f_i) = \ell$, we set 
 \begin{multline}\label{representation number hermitian}
         \mathcal{R}(f_i, \mathfrak{m}, d) := \left\{h \in \mathfrak{m}^2/d\mathfrak{m}^2 \mid h\overline{h} \equiv -\ell N(\mathfrak{m})^2 \pmod {d N(\mathfrak{m})^2} \right. \\
    \left.
    \textup{ and } \exists \sigma  \in \m{\mathfrak{m} & \mathfrak{m}\\\overline{\mathfrak{m}}&\overline{\mathfrak{m}}} \textup{ with } \det \sigma = N(\mathfrak{m})\textup{ such that }
    \right. \\
    \left.
    \sigma f_i\sigma^{*} \sim \m{\frac{(h\overline{h}+\ell N(\mathfrak{m})^2)}{dN(\mathfrak{m})^2} N(\mathfrak{m}) & h\\ \overline{h} & dN(\mathfrak{m})} \right\},
    \end{multline}
    Let then $r(f_i, \mf{m}, d):= \#\mathcal{R}(f_i, \mf{m}, d)$. We now prove the following Proposition.

    \begin{proposition}\label{proposition r=n}
        Let $\ell \in \mathbb{Z}_{>0}$ and $\xi_i \in L_0^+[\ell,2^{-1}\mathbb{Z}]$. Recall $f_i := f_{\omega}^{-1}(\xi_i) \in \Lambda^+[\ell,2^{-1}\mathbb{Z}]$. Then, for any $d \in \mathbb{N}$, we have $n(\xi_{\mf{m}, i};d) = r(f_i, \mf{m}, d)$.
    \end{proposition}
\begin{proof}

We consider the map $\mathcal{N}(\xi_{\mathfrak{m},i}, d) \longrightarrow \mathcal{R}(f_i, \overline{\mathfrak{m}}, d)$ given by $s \longmapsto h_{s}:= \theta(A_{\mathfrak{m}}S_{\mathfrak{m}}^{-1}s) N(\mathfrak{m})$. \\

\underline{Well-defined:} First of all, if $s \in \mathcal{N}(\xi_{\mathfrak{m},i}, d)$, then we consider the vector 
\begin{equation*}
    v := \left(\frac{\frac{1}{2}q_{\mathfrak{m}}S_{\mathfrak{m}}^{-1}[s]-D}{q_{\mathfrak{m}}d}, A_{\mathfrak{m}}S_{\mathfrak{m}}^{-1}s, d\right).
\end{equation*}
We observe that $\phi_0[v] = -\ell \in \mathbb{Z}$ and because of the conditions, we see that $v \in L_{\mathfrak{m},0}^{*}$, and $L_{\mathfrak{m},0}$ is $\mathbb{Z}$-maximal. This, by \cite[Lemma 6.2, (3)]{shimura2004arithmetic}, in turn implies $v \in L_{\mathfrak{m},0}$, which shows $A_{\mathfrak{m}}S_{\mathfrak{m}}^{-1}s \in L_{\mathfrak{m}}$.\\

This now shows that $\theta(A_{\mathfrak{m}}S_{\mathfrak{m}}^{-1}s) \in \overline{\mathfrak{m}}^2/ N(\mathfrak{m})$, and hence $h_s \in \overline{\mathfrak{m}}^2$. Moreover, since $\theta$ is an isometry,
\begin{equation*}
    h_s \overline{h}_s = N_{K/\mathbb{Q}}(h_s) = N_{K/\mathbb{Q}} (\theta(A_{\mathfrak{m}}S_{\mathfrak{m}}^{-1}s)N(\mathfrak{m})) = N(\mathfrak{m})^2\frac{1}{2}S[A_{\mathfrak{m}}S_{\mathfrak{m}}^{-1}s] = N(\mathfrak{m})^2\frac{1}{2}S_{\mathfrak{m}}^{-1}[s]. 
\end{equation*}
Hence $h_s \overline{h}_s + \ell N(\mathfrak{m})^2 = N(\mathfrak{m})^2\left(\frac{1}{2}S_{\mathfrak{m}}^{-1}[s]+\ell\right)\equiv 0 \pmod {d N(\mathfrak{m})^2}$, because $D \equiv \frac{1}{2}q_{\mathfrak{m}}S_{\mathfrak{m}}^{-1}[s] \pmod{q_{\mathfrak{m}}d}$ implies $\frac{1}{2}S_{\mathfrak{m}}^{-1}[s]+\ell \equiv 0\pmod d$, after replacing $D = -q_\mathfrak{m}\ell$.\\

Assume now $s \equiv t \pmod{d S_{\mathfrak{m}}\mathbb{Z}^2}$ and write $s-t = d S_{\mathfrak{m}}u$, with $u \in \mathbb{Z}^2$. We then have 
\begin{equation*}
    h_s - h_t =  N(\mathfrak{m})\theta(A_{\mathfrak{m}}S_{\mathfrak{m}}^{-1}(s-t)) = dN(\mathfrak{m})\theta(A_{\mathfrak{m}}u),
\end{equation*}
and since $L_{\mathfrak{m}} = A_{\mathfrak{m}}L$, we have $\theta(A_{\mathfrak{m}}u) \in \overline{\mathfrak{m}}^2/N(\mathfrak{m})$. This shows $h_s - h_t  \in d \overline{\mathfrak{m}}^2$. \\

Finally, $\m{\frac{(h_s\overline{h_s}+\ell N(\mathfrak{m})^2)}{dN(\mathfrak{m})^2} N(\mathfrak{m}) & h_s\\ \overline{h_s} & dN(\mathfrak{m})} = A_s$ (see \eqref{A s}), and the last property follows from the property of $A_s$ and the properties of $\alpha_{\mathfrak{m}}$. \newline

\underline{Injective}: Assume now $h_{s} = h_{t}$ in $\mathcal{R}(f_i, \overline{\mathfrak{m}},d)$. This means that there is some $a \in \overline{\mathfrak{m}}^2$ such that $h_s-h_t = da$. By the calculations above, we have
\begin{equation*}
    h_s - h_t = (N\mathfrak{m} )\theta(A_{\mathfrak{m}}S_{\mathfrak{m}}^{-1}(s-t)),
\end{equation*}
which shows 
\begin{equation*}
    s - t = N(\mathfrak{m})^{-1}S_{\mathfrak{m}}A_{\mathfrak{m}}^{-1}\theta^{-1}(h_s-h_t) = dS_{\mathfrak{m}}A_{\mathfrak{m}}^{-1}\theta^{-1}(N(\mathfrak{m})^{-1} a) \in dS_{\mathfrak{m}}\mathbb{Z}^2,
\end{equation*}
because $\theta(N(\mathfrak{m})^{-1} a) \in L_{\mathfrak{m}}$, since $N(\mathfrak{m})^{-1} a \in \overline{\mathfrak{m}}^2/N(\mathfrak{m})$ and therefore $A_{\mathfrak{m}}^{-1}L_{\mathfrak{m}} \in L =\mathbb{Z}^2$.\\

\underline{Surjective}: If now $h \in \mathcal{R}(f_i, \overline{\mathfrak{m}},d)$, we can choose $s:= S_{\mathfrak{m}}A_{\mathfrak{m}}^{-1}\theta^{-1}(N(\mathfrak{m})^{-1}h) \in \mathbb{Z}^2$, since $N(\mathfrak{m})^{-1}h \in \overline{\mathfrak{m}}^2/N(\mathfrak{m})$ and hence $\theta^{-1}(N(\mathfrak{m})^{-1}h) \in L_{\mathfrak{m}}$, hence $s \in S_\mathfrak{m} \mathbb{Z}^2 \subset \mathbb{Z}^2$.
\end{proof}

Our next step is to relate the numbers $r(f_i, \overline{\mf{m}},d)$ to some representation numbers attached to the binary Hermitian form associated with the Hermitian matrix $f_i$.\\

Let $I_{\ell}$ denote the number of $\textup{SL}_2(\mathcal{O}_K)$-equivalence classes in $\Lambda$ for the set of Hermitian binary forms of fixed discriminant $\ell$. We have $I_{\ell} \geq h_1$, and we extend the $\textup{SL}_2(\mathcal{O}_K)$-equivalence classes $\{f_i\}_{i=1}^{h_1}$, with $f_i \in \Lambda[\ell,2^{-1}\mathbb{Z}]$, to a complete set of $\textup{SL}_2(\mathcal{O}_K)$-equivalence classes $\{f_i\}_{i=1}^{I_{\ell}}$ (note that by the first statement of Lemma \ref{integrality of hermitian matrices}, we have $\Lambda[\ell,2^{-1}\mathbb{Z}] \subset \Lambda$ and hence our original $f_i$'s belong in $\Lambda$). Here, we will abuse the notation and write $f_i$ for both the binary Hermitian form and the Hermitian matrix representing the form. That is, for $u,v \in K$, we write $f_i(u,v):= (u,v) f_i (\bar{u},\bar{v})^t$. \\

For each $i=1,\cdots, I_{\ell}$, we introduce the set of $\mf{m}$-primitive representations of $d$ by $f_i$ (see \cite[Definition 3.1]{Elstrodt})
\[
P(f_i,\mf{m},d):= \{ (u,v) \in \mf{m} \times \mf{m} \,\,|\,\,f_i(u,v) = N(\mf{m})d,\,\,\,<u,v>=\mf{m} \}, 
\]
where the notation $<u,v>$ is the fractional ideal generated by $u$ and $v$ in $K$. Moreover, we define the $\mf{m}$-primitive representation numbers $p(f_i,\mf{m},d) := \# P(f_i,\mf{m},d)$. Finally, in addition to the set \ref{representation number hermitian}, we define

\[
\mathcal{R}(\mathfrak{m}, d, -\ell) := \left\{h \in \mathfrak{m}^2/d\mathfrak{m}^2 \mid h\overline{h} \equiv -\ell N(\mathfrak{m})^2 \pmod {d N(\mathfrak{m})^2} \right).
\]
Following \cite[Definition 3.4]{Elstrodt}, we define a map
\begin{equation}\label{phi-map}
\varphi_{f_i}: P(f_i, \mf{m}, d) \rightarrow \mathcal{R}(\mf{m}, d,-\ell),\,\,\,\varphi_{f_i}\left((u,v)\right) := h,
\end{equation}
where for any $(u,v) \in P(f_i, \mf{m}, d)$, we pick $r, s \in \overline{\mf{m}}$ such that $\det \m{r & s \\ u & v}  = N(\mf{m})$ and set $h$ to be the entry such that
\[
\m{r & s \\ u & v} f_i \m{\bar{r} & \bar{u} \\ \bar{s} & \bar{v}} = \m{ * & h \\ \bar{h} & dN(\mf{m})}.
\]
It is shown in \cite[Definition 3.4]{Elstrodt} that this map is well-defined and from \cite[Proposition 3.5]{Elstrodt} induces a map
\[
\varphi_{f_i} : P(f_i, \mf{m}, d)/E(f_i) \rightarrow \mathcal{R}(\mf{m}, d,-\ell),
\]
where $E(f_i) :=\left\{\gamma \in \textup{SL}_2(\mathcal{O}_K)\,\,|\,\,\gamma f_i \gamma^* = f_i \right\}$ acting linearly on the set $P(f_i,\mf{m}, d)$. We then have the following Lemma.
\begin{lemma}
    The map $\phi_{f_i}$ gives a bijection between $P(f_i, \overline{\mathfrak{m}}, d)/E(f_i)$ and $\mathcal{R}(f_i, \mathfrak{m}, d)$ of \eqref{representation number hermitian}.
\end{lemma}
\begin{proof}
    First, from the definition above, the image $\phi_{f_i}$ is in $\mathcal{R}(f_i, \mathfrak{m}, d)$. We now form the disjoint union of the domains $P(f_i,\overline{\mf{m}},d)/E(f_i)$ for all $i =1, \cdots, I_{\ell}$, and construct the map
\begin{equation}\label{composition-of-phi-maps}
    \varphi_{\ell} : \bigsqcup_{i \in I_{\ell}} P(f_i, \overline{\mf{m}}, d)/E(f_i) \longrightarrow \mathcal{R}(\overline{\mf{m}},d, -\ell)
\end{equation}
such that $\varphi_{\ell}{\mid (P(f_i,\overline{\mf{m}}, d)/E(f_i))} := \varphi_{f_i}$ for all $i=1, \cdots, I_{\ell}$. From \cite[Theorem 3.7]{Elstrodt}, the map $\varphi_{\ell}$ is bijective.\\

    Since $\phi_{\ell}$ of \eqref{composition-of-phi-maps} is bijective and $\phi_{f_i} = \phi_{\ell} |_{P(f_i, \mathfrak{m},d)/E(f_{i})}$, we have that $\phi_{f_i}$ is injective. We now only need to show that the image is the whole of $\mathcal{R}(f_i, \mathfrak{m},d)$. Let $h \in \mathcal{R}(f_i, \mathfrak{m},d)$. Then, since $\phi_{\ell}$ is surjective, there is $(u,v) \in P(f_j, \overline{\mathfrak{m}}, d)$ with $\phi_{\ell}((u,v)) = h$ for some $j$. This means that there is $\sigma \in \m{\mathfrak{m} & \mathfrak{m}\\\overline{\mathfrak{m}}&\overline{\mathfrak{m}}}$ with $\det \sigma = N(\mathfrak{m})$ such that
    \begin{equation*}
        \sigma f_j \sigma ^{*} = \m{\star & h\\\overline{h} & d N(\mathfrak{m})}.
    \end{equation*}
    But since $h \in R(f_i, \mathfrak{m},d)$, we also have that $\exists \tau $ with the same properties such that
    \begin{equation*}
        \tau f_j \tau ^{*} = \m{\star & h\\\overline{h} & d N(\mathfrak{m})}.
    \end{equation*}
    The $\star$ entries are the same because the matrix has determinant $\ell N (\mathfrak{m})^2$. Hence $\tau^{-1}\sigma f_j (\tau^{-1}\sigma)^{*} = f_i$. But we can now verify $\tau^{-1}\sigma \in \textup{SL}_2(\mathcal{O}_K)$, which forces $i=j$. Hence, the Lemma follows.
\end{proof}

\section{Partial Zeta functions and Latimer's work}
In this Section, we will establish a link between the partial zeta functions $\zeta_{\mathcal{Q}}(\xi_{\mf{m}i};s)$ of \eqref{zeta xi i m definition} and partial zeta functions attached to the quaternion algebra $B$. From Proposition \ref{proposition r=n} we have, for any $d \geq 1$, $n(\xi_{\mf{m},i}, d) = r(f_i, \mf{m},d)$, and from the previous discussion $r(f_i, \mf{m},d) = p(f_i,\mf{m}, d)/e(f_i)$. Hence,
\begin{multline*}
\zeta_{\mathcal{Q}}(\xi_{\mf{m},i};s)= \sum_{\substack{N=1\\(N,p)=1 \forall p\in \mathcal{Q}}}^{\infty} n(\xi_{\mf{m},i}; N)N^{-s}=\sum_{\substack{N=1\\(N,p)=1 \forall p\in \mathcal{Q}}}^{\infty} r(f_i, \mf{m}, N)N^{-s}=\\= \frac{1}{e(f_i)} \sum_{\substack{N=1\\(N,p)=1 \forall p\in \mathcal{Q}}}^{\infty} p(f_i,\mf{m}, N)N^{-s}=\frac{1}{e(f_i)} N(\mf{m})^{s} \sum_{\substack{(u,v) \in \mf{m}^2 \setminus \mathcal{Q},\\ <u,v>=\mf{m}}}  \frac{1}{f_i(u,v)^{s}},
\end{multline*}

where the notation $(u,v) \in \mf{m}^2 \setminus \mathcal{Q}$ means that we exclude the pairs $(u,v)$ which give $\gcd(f_i(u,v),p) \neq 1$ for $p \in \mathcal{Q}$. 
We now set

\begin{equation}\label{zeta hat}
    \widehat{\zeta}_{\mathcal{Q}}(\xi_i;s) :=\sum_{\mf{m} \in \textup{Cl}(K)} \zeta_{\mf{m}, \mathcal{Q}}(s) \zeta_{\mathcal{Q}}(\xi_{\mf{m},i};s).
\end{equation}
In particular, equation (\ref{expression with iota}) now reads 
\begin{equation}\label{expression with zeta hat}
D(\bm{F}, \bm{P};s) = \sum_{i=1}^{h_1} \widehat{\zeta}_{\mathcal{Q}}(\xi_i;s-k+3)D_{\mathcal{Q}}(F,\iota(\xi_i);s).
\end{equation}

We now have the following Lemma.
\begin{lemma}\label{zeta hat class number one}
    
    With notation as above we have,
    \[
   \widehat{\zeta}_{\mathcal{Q}}(\xi_i;s) =  \frac{1}{e(f_i)} \sum_{(u,v) \in \mathcal{O}_K^2\setminus \mathcal{Q}} \frac{1}{f_i(u,v)^{s}}.
    \]
\end{lemma}
\begin{proof}

This is \cite[Proposition 4.5]{Elstrodt}.
\end{proof}
To proceed further, we rely on the two papers of Latimer \cite{Latimer1,Latimer2} on the relation between representation numbers of binary Hermitian forms and ideals in quaternion algebras.\newline

Recall that in Section \ref{quaternion algebras}, we defined $B = K + \epsilon K$ and fixed the order $\mathfrak{o} = \mathcal{O}_K + \epsilon \mathcal{O}_K$ in $B$. For a left ideal $\mathfrak{U}$ in $\mathfrak{o}$, we may choose elements $\omega_1,\omega_2 \in \mathfrak{o}$ such that 
\[
\mathfrak{U} = \mathcal{O}_K \omega_1 + \mathcal{O}_K \omega_2.
\]
We call $\omega_1, \omega_2$ a basis of $\mathfrak{U}$ with respect to $\mathcal{O}_K$ and write $\mathfrak{U} = [\omega_1,\omega_2]$. If we write $\omega_i = g_{i1} + \epsilon \,\overline{g_{i2}}$, with $g_{ij} \in \mathcal{O}_K$, we define 
\begin{equation}\label{norm_latimer}
    N(\mathfrak{U}) := \det \m{g_{11} & g_{12} \\ g_{21} & g_{22}}.
\end{equation}
(see \cite[Section 4]{Latimer1}). In the case when $N(\mathfrak{U})$ is a positive integer, we call $\omega_1, \omega_2$ a proper basis. 
Following \cite[Theorem 1]{Latimer1}, we will call $\mathfrak{U}$ a regular ideal if it possesses a proper basis. From \cite[Lemma 4]{Latimer1}, if $\mathfrak{U}$ is a regular ideal, every ideal in its (left) class is also regular. Hence, the term regular class is meaningful.\\

The main theorem in \cite[Theorem 3]{Latimer1} reads as follows:
\begin{theorem}
    There is one-to-one correspondence between regular classes of left ideals in $\mathfrak{o}$ and $\textup{SL}_2(\mathcal{O}_K)$-classes of positive definite binary hermitian forms of discriminant $\ell$.
\end{theorem}
We now write $\{\mathfrak{U}_i\}_{i=1}^{h_1}$ for representatives of regular classes of ideals in $\mathfrak{o}$ corresponding to the binary Hermitian forms $\{f_i\}_{i=1}^{h_1}$, where we recall that $f_i \in \Lambda[\ell,2^{-1}\mathbb{Z}] \subseteq \Lambda$. We first establish the following Proposition.

\begin{proposition}
   The ideals $\mathfrak{U}_i$, $i=1,\ldots,h_1$, are invertible. 
\end{proposition}

\begin{proof}
Let us consider a $f_i \in \Lambda^+\left[\ell, 2^{-1}\mathbb{Z}\right] \subseteq \Lambda$, and write $f_i = \m{a & b \\ \bar{b} & c}$, with $a,c \in \mathbb{Z}$ and $b \in \mathcal{O}_K$. By the proof of \cite[Theorem 3]{Latimer1}, a representative of the class corresponding to this form is $\mathcal{L}:= [a, b+ \epsilon]$. It suffices to show $\mathcal{L}$ invertible, as then every ideal in its class will be invertible (see \cite[Chapter 17]{voight}). Recall, $\epsilon^2 = -k = b\bar{b} -ac<0$. 

Let $\mathcal{L}' := [a, \bar{b}-\epsilon]$, which is the ideal obtained by taking the standard involution of $B$. We compute 
\begin{multline*}
\mathcal{L}\mathcal{L}' = a^2\mathcal{O}_K + (a\bar{b} - a\epsilon)\mathcal{O}_K +(ab+\epsilon a)\mathcal{O}_K +(b\bar{b} -b \epsilon + \epsilon\bar{b} -\epsilon^2)\mathcal{O}_K= \\=a(a\mathcal{O}_K + b\mathcal{O}_K+\bar{b}\mathcal{O}_K +c\mathcal{O}_K) + a\epsilon \mathcal{O}_K,
\end{multline*}
since $(b+\epsilon)(\bar{b}-\epsilon) = b\bar{b} -b\epsilon +\epsilon \bar{b} -\epsilon^2 = b\bar{b} +k = ac$. Hence,
\[
\mathcal{L}\mathcal{L}' = a(a\mathcal{O}_K + b\mathcal{O}_K+\bar{b}\mathcal{O}_K +c\mathcal{O}_K + \epsilon\mathcal{O}_K) 
\]
We now claim that in order to establish that $\mathcal{L}$ is invertible with inverse $\frac{1}{a} \mathcal{L}'$, it is enough to show that 
\begin{equation}\label{claim}
    a\mathcal{O}_K + b\mathcal{O}_K+\bar{b}\mathcal{O}_K +c\mathcal{O}_K= \mathcal{O}_K.
\end{equation}
Indeed, if this is true, we will have that $\mathcal{L} \left(\frac{1}{a}\mathcal{L}'\right) = \mathfrak{o}$, so it would be enough to show that $\mathfrak{o} = \mathcal{O}_L(\mathcal{L})
$. Following \cite[paragraph 16.5.18]{voight}, 
we have
\[
\mathfrak{o} = \mathcal{L} \left(\frac{1}{a}\mathcal{L}'\right) = \mathcal{O}_L(\mathcal{L})\mathcal{L} \left(\frac{1}{a}\mathcal{L}'\right) = \mathcal{O}_L(\mathcal{L}) \mathfrak{o} = \mathcal{O}_L(\mathcal{L}),
\]
because $\mathfrak{o} \subseteq \mathcal{O}_L(\mathcal{L})$, as $\mathcal{L}$ is an $\mathfrak{o}$-ideal. That is, $\mathfrak{o} =\mathcal{O}_L(\mathcal{L})$, which shows that $\mathcal{L}$ is invertible. \newline

We now establish our claim in \eqref{claim}, i.e.,
\[
a\mathcal{O}_K + b\mathcal{O}_K+\bar{b}\mathcal{O}_K +c\mathcal{O}_K= \mathcal{O}_K.
\]

We will make use of the assumption that $f_i \in \Lambda^{+}\left[\ell, 2^{-1}\mathbb{Z}\right]$, and $\ell$ is an integer. Note that from \cite[(2.18)]{shimura_diophantine_2006}, the first condition is equivalent to  
\[
a\mathbb{Z} + c\mathbb{Z} + \textup{Tr}_{K/\mathbb{Q}}(b \mathcal{O}_K) = \mathbb{Z},
\]
and the second implies that $b \in \mathcal{O}_K$ (see Lemma \ref{integrality of hermitian matrices}), which in turn gives 
\[
a\mathcal{O}_K + b\mathcal{O}_K+\bar{b}\mathcal{O}_K +c\mathcal{O}_K \subseteq \mathcal{O}_K.
\]
For the other direction, we write $\mathcal{O}_K = \mathbb{Z} + \omega\mathbb{Z} = \mathbb{Z} + \overline{\omega}\mathbb{Z}$, with $\omega$ as in \eqref{omega}. We then have
\begin{multline*}
    \mathcal{O}_K = \mathbb{Z} + \omega\mathbb{Z} = (a\mathbb{Z} + c\mathbb{Z} + \textup{Tr}_{K/\mathbb{Q}}(b \mathcal{O}_K)) + \omega (a\mathbb{Z} + c\mathbb{Z} + \textup{Tr}_{K/\mathbb{Q}}(b \mathcal{O}_K)) =\\=
a\mathcal{O}_K + c\mathcal{O}_K + \textup{Tr}_{K/\mathbb{Q}}(b \mathcal{O}_K)+ \omega \textup{Tr}_{K/\mathbb{Q}}(b \mathcal{O}_K)\subseteq a\mathcal{O}_K + c\mathcal{O}_K +b\mathcal{O}_K+\bar{b}\mathcal{O}_K,
\end{multline*}

since $\textup{Tr}_{K/\mathbb{Q}}(b \mathcal{O}_K) \subset b \mathcal{O}_K + \bar{b}\mathcal{O}_K$ and similarly $\omega \textup{Tr}_{K/\mathbb{Q}}(b \mathcal{O}_K) \subset b \mathcal{O}_K + \bar{b}\mathcal{O}_K$.
\end{proof}

For a (left) ideal $I$ of $\mathfrak{o}$, we set $\mathcal{N}(I) := \# (\mathfrak{o}/I)$. We will now prove the following Lemma, which is given without proof as a footnote in \cite[page 441]{Latimer1}.
\begin{lemma}\label{relation of norms}
    If $I$ is a regular left ideal in $\mathfrak{o}$, we have $N(I)^2 = \mathcal{N}(I)$, where $N(I)$ is the norm of equation \eqref{norm_latimer}.
\end{lemma}
\begin{proof}
    First of all, if $\{a_i\}, \{b_i\}$ are $\mathbb{Z}$-bases for $\mathfrak{o}$ and $I$ respectively, related by a base change matrix $C$, we have $\mathcal{N}(I) = |\det(C)|$.  We write $\mathcal{O}_K = \mathbb{Z} + \mathbb{Z}\omega$ and then we have that a $\mathbb{Z}$-basis for $\mathfrak{o}$ is $\{1, \omega, \epsilon, \omega \epsilon\}$. Moreover, assume we take an $\mathcal{O}_K$-basis for $I$, given by $\theta_1 = x+y\epsilon $ and $\theta_2 = z+w\epsilon $, with $x,y,z,w \in \mathcal{O}_K$. A $\mathbb{Z}$-basis for $I$ is then given by $\{\theta_1, \omega \theta_1, \theta_2, \omega\theta_2\}$. We write $x=a+b\omega$, $y = c+d\omega$, $z = e+f\omega$, and $w = g+h\omega$, with $a,b,c,d,e,f,g,h \in \mathbb{Z}$.\\

    Consider first the case $K=\mathbb{Q}(\sqrt{-m})$ with $m \equiv 1, 2 \pmod 4$. Then, using $\omega^2=-m$, $\overline{\omega} = -\omega$ and the fact that $\epsilon \omega = \overline{\omega}\epsilon = -\omega \epsilon$, we have that the matrix $C$ can be written as
    \begin{equation*}
        C = \m{a&b&c&d \\ -mb &a &-md& c\\e&f&g&h\\-mf&e&-mh&g}, \ m\equiv 1,2 \pmod 4.
    \end{equation*}
    If $m \equiv 3 \pmod 4$ and using the fact that $\overline{\omega} = 1-\omega$ and $\omega^2 = \omega - (1+m)/4$, we get
    \begin{equation*}
        C = \m{a&b&c&d \\ -(1+m)b/4 &b+a &-d(1+m)/4& c+d\\e&f&g&h\\-(1+m)f/4 &f+e &-h(1+m)/4& g+h}, \ m\equiv 3 \pmod 4.
    \end{equation*}

In both cases, we have $N(I) = \det \m{x &y\\z&w}$. Using the fact that $I$ is a regular ideal and hence $N(I)$ is a positive integer, we can compute $N(I)^2 = N(I)\overline{N(I)} = |\det (C)| = \mathcal{N}(I)$ for both cases, as claimed.
\end{proof}

\begin{proposition}

For each $i=1, \cdots, h_1$, we have
\[
    \widehat{\zeta}_{\mathcal{Q}}(\xi_i;s) = \sum_{\substack{I \in [\mathfrak{U}_i]\\(I,{\mathcal{Q}})=1 }} N(I)^{-s},
\]
    where the summation is over all left ideals $I$ in $\mathfrak{o}$ in the class of $\mathfrak{U}_i$ that have norm not divisible by any prime in the set $\mathcal{P}$.  
\end{proposition}
\begin{proof}
    For this, we rely on \cite[Theorem 1]{Latimer2}. 
    On the one hand, we have from Lemma \ref{zeta hat class number one} that
    \begin{equation*}
        \widehat{\zeta}_{\mathcal{Q}}(\xi_i;s) =  \frac{1}{e(f_i)} \sum_{(u,v) \in \mathcal{O}_K^2\setminus \mathcal{Q}} \frac{1}{f_i(u,v)^{s}} = \frac{1}{e(f_i)}\sum_{\substack{d=1 \\ (d,p)=1 \ \forall p \in \mathcal{Q}}}^{\infty} q(f_i,d)d^{-s},
    \end{equation*}
    where $q(f_i, d) := \#\{(u,v) \in \mathcal{O}_K^2 \backslash \mathcal{Q} \mid f_i(u,v) = d\}$. From \cite[Corollary, p. 443]{Latimer2} and because each $f_i$ is taken positive definite, we have that for each $d \geq 1$ with $\gcd (d, p) =1$ for all $p \in \mathcal{Q}$
    \begin{equation*}
        q(f_i, d) = e(f_i)\cdot  \#\{J \in [\mathfrak{U}_i] \textup{ with }N(J)=d\}.
    \end{equation*}
    From this, the Proposition follows.
\end{proof}

We have stated the Proposition above in terms of left ideals in $\mathfrak{o}$, but the presence of the main involution of $B$ allows us to transfer the results also for right-ideals. Given a left ideal $I$ in $\mathfrak{o}$, we can obtain a right ideal $\bar{I} := \{ \bar{a} \,\,|\,\, a \in I\}$, where $\,\bar{}\,$ here denotes the main involution of $B$. Then one obtains a bijection between left and right ideals which preserves norms. Indeed, let us show that $N(\bar{I}) = N(I)$. From \cite[Section 16.6.14]{voight}, for any invertible ideal $J$, we have $\bar{J} J = N(J) \mathcal{O}_R(J)$, and $J \bar{J} = N(J)\mathcal{O}_L(J)$. By setting $J=\bar{I}$ in the first relation and $J=I$ in the second, we obtain 
\begin{equation*}
    I\bar{I} = \textup{nrd}(\bar{I}) \mathcal{O}_R(\bar{I}) \ \textup{ and } \ I\bar{I} =\textup{nrd}(I) \mathcal{O}_L(I).
\end{equation*}
From \cite[Lemma 16.6.7]{voight}, we have $\mathcal{O}_R(\bar{I}) = \mathcal{O}_L(I)$, and hence we obtain that $\textup{nrd}(I) = \textup{nrd}(\bar{I})$. But $\textup{nrd}(I) = N(I)$, from our Lemma \ref{relation of norms} and \cite[Theorem 16.1.3]{voight}.\\

In particular, we may restate the above Proposition as follows.
\begin{proposition}\label{hat zeta functions}

For each $i=1, \cdots, h_1$, we have
\[
    \widehat{\zeta}_{\mathcal{Q}}(\xi_i;s) = \sum_{\substack{I \in [\overline{\mathfrak{U}_i}]\\(I,{\mathcal{Q}})=1 }} N(I)^{-s},
\]
where the summation is over all right ideals $I$ in $\mathfrak{o}$ in the class of $\overline{\mathfrak{U}_i}$ that have norm not divisible by any prime in the set $\mathcal{P}$.  
\end{proposition}
\section{Main Theorem}\label{section-main-theorem}
    In this Section, we are finally ready to state and prove our main Theorem. The first step is to identify everything in the quaternion algebra $B$. We have already done that for the zeta functions $\zeta_{\mathcal{Q}}(\xi_{\mf{m},i};s)$ in the previous Section. As we have already mentioned in Section \ref{quaternion algebras}, $\Xi(A^{+}(W)) = B$. Moreover, $\Xi(A^{+}(\Lambda) \cap A^{+}(W)) = \mathfrak{o}$ from \cite[(2.19b)]{shimura2006quadratic}. Therefore, $\Xi(A^{+}(\Lambda \cap W)) = \mathfrak{o}$ from \cite[Lemma 2.1]{murata2014}. \\

    By substituting \eqref{expression3} to \eqref{expression with zeta hat}, we obtain
    \begin{multline}\label{what we have from Dirichlet}
        \zeta_{\mathcal{Q}}(2s-2k+4)D_{\mathcal{Q}}(\bm{F}, \bm{P};s) \\=L_{\mathcal{Q}}(F;s-k+2)\sum_{i=1}^{h_1}\widehat{\zeta}_{\mathcal{Q}}(\xi_i;s-k+3)\sum_{j=1}^{h_2}A_{f_j} L_{\mathcal{Q}}(\overline{f_j}; s-k+5/2)^{-1} f_{j}(\tau(\zeta_i))=\\ = L_{\mathcal{Q}}(F;s-k+2)\sum_{j=1}^{h_2}A_{f_j} L_{\mathcal{Q}}(\overline{f_j}; s-k+5/2)^{-1}\sum_{i=1}^{h_1} f_{j}(\tau(\zeta_i))\widehat{\zeta}_{\mathcal{Q}}(\xi_i;s-k+3).
    \end{multline}
    
    Note that as has been observed in \cite[Theorem 9.3]{psyroukis_orthogonal}, we have $L_{\textup{Sug}}\left(\textbf{1}; s-1/2\right) = L_{\textup{Shi}}(s)$. Here, the notation means the $L$-function we encounter in \cite{sugano} and \cite{shimura_orthogonal} respectively, as these are normalised differently. From now on, we use the normalisation of Shimura. Therefore, from now on, we will write $L_{\mathcal{Q}}(f_j, s-k+3)$ instead of $L_{\mathcal{Q}}(f_j, s-k+5/2)$. However, we still keep the normalisation of Sugano for $L_{\mathcal{Q}}(F,s)$.\\

    Let us now set up some notation. We remind the reader that we have already chosen a set of representatives $\{\zeta_i\}_{i=1}^{h_1}$ for $G^{+}(\widetilde{W})_{\mathbb{Q}}\backslash G^{+}(\widetilde{W})_{\mathbb{A}}/(G^{+}(\widetilde{W})_{\mathbb{A}} \cap \widetilde{J})$. In the following, we will slightly abuse notation and consider $\{\zeta_i\}_{i=1}^{h_1}$ as elements in $(U,\psi)$ via $f_{\omega}^{-1}$. Hence, we will regard them as representatives for $G^{+}(W)_{\mathbb{Q}}\backslash G^{+}(W)_{\mathbb{A}}/(G^{+}(W)_{\mathbb{A}} \cap J)$. Under the commutative diagram of Proposition \ref{diagram quaternion algebra}, we will write $u_i$ for the element of $B^{\times} \backslash B^{\times}_{\mathbb{A}}/ X$ which corresponds to $\zeta_i$, for all $i=1,\cdots, h_1$. Then $\{u_i\}_{i=1}^{h_1}$ is a complete set of representatives for $B^{\times} \backslash B^{\times}_{\mathbb{A}}/ X$. \\

    Similarly, we choose $\{v_i\}_{i=1}^{h_2}$ as a set of representatives for $H(\xi)_{\mathbb{Q}}\backslash H(\xi)_{\mathbb{A}}/(H(\xi)_{\mathbb{A}} \cap C)$ and write $\{w_i\}_{i=1}^{h_2}$ for the corresponding representatives in $B^{\times} \backslash B^{\times}_{\mathbb{A}}/ N_B(X)$, under Proposition \ref{diagram quaternion algebra}.\\

    We set $S:=A^+(\Lambda \cap W)$ and $\mathcal{J}:= G^+(W)_{\mathbb{A}} \cap J$, and we select the finite set of primes $\mathcal{Q}$ large enough so that $\mathcal{J}_p = (S_p)^{\times} \cap G^+(W)_p$ for all $p \not \in \mathcal{Q}$. We then consider the set $\mathcal{Y} := \{ y \in G^+(W)_{\mathbb{A}} \,\,|\,\, y_p \in \mathcal{Y}_p \textup{ for all } p\}$, where 
    \[
    \mathcal{Y}_p = \begin{cases} G^+(W)_p \cap S_p \,\,\,\text{for}\,\,\,p \not \in \mathcal{Q} \\   \mathcal{J}_p \,\,\text{otherwise}\end{cases}.
    \]
    We further define $\mathcal{Y}^0:= \{ y \in \mathcal{Y} \,\,|\,\, \kappa(y) = \mathbb{Z}\}$, where $\kappa$ is defined in \cite[Section 22.8]{shimura2004arithmetic}, and is called the capacity ideal. Moreover, we fix a set $\mathcal{A} \subset G^{+}(W)_{\mathbb{Q}}\backslash G^{+}(W)_{\mathbb{A}}/(G^{+}(W)_{\mathbb{A}} \cap J)$ of pre-images for the elements $\{v_l\}_{l=1}^{h_2}$ under $\tau$. Assume $\{a_l\}_{l=1}^{h_2}$ are these elements.
    
    We then define 
    \[
    \mathcal{Y}_{a_{\ell}} := \{ y \in \mathcal{Y} \,\,|\,\, \tau(y) \in H(\xi)_{\mathbb{Q}} \tau(a_{\ell})(H(\xi)_{\mathbb{A}} \cap C)\}, 
    \]
and
\[
\mathcal{Y}_{a_{\ell}}^0 := \mathcal{Y}^0 \cap \mathcal{Y}_{a_{\ell}}.
\]
Moreover, we set $Q := \left\{ q \in \prod_{p} \mathbb{Z}_p \,\, |\,\, q_p \neq 0,\,\, q-1 \in p\mathbb{Z}_p,\, \forall p \in \mathcal{Q} \right\}$, and $R := \mathcal{J} \cap \mathbb{Q}^{\times}_{\mathbf{h}}$, where $\mathbb{Q}^{\times}_{\mathbf{h}}$ denotes the finite adeles of $\mathbb{Q}$. We then have the following Lemma.

\begin{lemma} \label{decomposition}

We have a bijection
\[
\mathcal{Y}_{a_{\ell}}/ \mathcal{J} \longleftrightarrow Q/R \times \mathcal{Y}^0_{a_{\ell}}/\mathcal{J}.
\]
Moreover,
\begin{equation}\label{y mod j decomposition}
    \mathcal{Y}/ \mathcal{J} = \bigsqcup_{\ell = 1}^{h_2} \mathcal{Y}_{a_{\ell}}/ \mathcal{J} 
\end{equation}
\end{lemma}
\begin{proof}
First of all, from \cite[Lemma 22.9, (iii)]{shimura2004arithmetic}, we have that there is a bijection between the set $\mathcal{Y}/ \mathcal{J}$ and the set $Q/R \times \mathcal{Y}^0/\mathcal{J}$ by mapping a $y \in \mathcal{Y}$ to $(q,y_0)$, where $y = q y_0$ with $q \in Q$ and $y_0 \in \mathcal{Y}^0$.
Moreover, from \cite[Lemma 22.10]{shimura2004arithmetic} we have that
\begin{equation*}
    \mathcal{Y}^0/ \mathcal{J} = \bigsqcup_{\ell=1}^{h_2} \mathcal{Y}^0_{a_{\ell}}/ \mathcal{J}.
\end{equation*}
The Lemma then follows by observing that $\tau(q) = 1$ for all $q \in Q$, and hence $\tau(y) = \tau(y_0)$ for $y = qy_0$ as before.  
\end{proof}

 We now state and prove the following Proposition.

\begin{proposition} \label{Equality of partial zeta}

We have the equality
    \[
    \sum_{\substack{i \textup{ such that } \\\rho(u_i)=w_{\ell}}}\widehat{\zeta}_{\mathcal{Q}}(\xi_i;s)=\sum_{y \in \mathcal{Y}_{a_{\ell}}/\mathcal{J}} N(\nu(y)\mathbb{Z})^{-s},
    \]
where $\nu$ is the map of \eqref{nu map} and $\rho$ the map of Proposition \ref{diagram quaternion algebra}.    
\end{proposition}
\begin{proof} In order to establish the equality, we first use the fact that the map $\Xi$ of Proposition \ref{Xi proposition} identifies $A^+(W)$ with $B$. Using this identification, we have that the map $\nu$ is the reduced norm of the quaternion algebra $B$ (see, for example, \cite[Section 7.3]{shimura2004arithmetic}). In particular, we may view the right-hand side of the claimed equality as a partial zeta series of $B$. Under this identification, we have that $\mathcal{Y}_p = \mathfrak{o}_p \cap B_p^{\times}$ for all $ p \not \in \mathcal{Q}$. Moreover, we have $\mathcal{J}_p = \mathfrak{o}_p^{\times}$ for all $ p$ and hence in particular $\mathcal{Y}_p = \mathfrak{o}_p^{\times}$ for all $p \in \mathcal{Q}$. Moreover the set $ \mathcal{Y}_{a_{\ell}}/\mathcal{J}$ can be identified with locally principal right ideals in the order $\mathfrak{o}$ with the property that their image with respect to the map $\rho$ is the order corresponding to $w_{\ell}$ and locally are equal to $\mathfrak{o}_p$ for all $ p \in \mathcal{Q}$. Indeed, for an element $y=(y_p) \in \mathcal{Y}$ we associate the locally principal ideal $I$ in $\mathfrak{o}$ such that $I_p = y_p \mathfrak{o}_p $ for all $ p \not \in \mathcal{Q}$, and $I_p = \mathfrak{o}_p$ for all $ p \in \mathcal{Q}$. By \cite[Main Theorem 16.7.7]{voight} such an ideal is invertible, and moreover every invertible (right) ideal in $\mathfrak{o}$ is locally principal. That is, the set $ \mathcal{Y}_{a_{\ell}}/\mathcal{J}$ may be identified with the set of invertible right ideals in $\mathfrak{o}$ with the property that they are locally trivial for $p \in \mathcal{Q}$, and their image with respect to the map $\rho$ is the order $w_{\ell}$. We now note that given an element $y=(y_p) \in \mathcal{Y}$ corresponding to an ideal $I$, we have that $N(\nu(y)\mathbb{Z}) = N(I)$. Indeed, by \cite[Main Theorem 16.7.7 (iv')]{voight} and Lemma \ref{relation of norms} above, we have that $N(I) = \textup{nrd}(I)$, where $\textup{nrd}(I)$ now equals $\textup{nrd}(I) = \textup{gcd}\{\nu(a) \,\,|\,\, a \in I\}$. Then, by \cite[Equation 16.3.4]{voight}, we have that 
\begin{equation*}
    \textup{nrd}(I) = \prod_p N(\nu(y_p)\mathbb{Z}) = N(\nu(y)\mathbb{Z}),
\end{equation*}
which concludes our claim. Finally, we also observe that the last equation gives that an invertible ideal $I$ of $\mathfrak{o}$ is locally trivial at a prime $p$ if and only if $\textup{gcd}(N(I),p) = 1$. Putting all these together, we have that
    \[
    \sum_{ \mathcal{Y}_{a_{\ell}}/\mathcal{J}} N(\nu(y)\mathbb{Z})^{-s}= \sum_{\rho(\overline{\mathfrak{U}_i})=w_{\ell}}\sum_{\substack{I \in [\overline{\mathfrak{U}_i}]\\(I,{\mathcal{Q}})=1 }} N(I)^{-s},
    \]
    where the first sum runs over a set of representatives of right ideals in $\mathfrak{o}$ which map to $w_{\ell}$ under the map $\rho$. Then Proposition \ref{hat zeta functions} allows us to conclude the proof.
    \end{proof}
    We now define the space of automorphic forms on $G^{+}(W)_{\mathbb{A}}$ of weight $k$ with repsect to $\mathcal{J}$ as in \cite[(22.4a), (22.4b)]{shimura2004arithmetic}. We denote this space by $\mathcal{S}_k(\mathcal{J})$. Moreover, if $\psi$ is a Hecke character of $\mathbb{Q}$, we define
    \begin{equation*}
        S_k(\mathcal{J}, \psi) = \left\{\mathbf{g} \in S_k(\mathcal{J}) \mid \mathbf{g}(cx) = \psi(c)^{-1}\mathbf{g}(x), \ \forall c \in \mathbb{Q}_{\mathbb{A}}^{\times}\right\},
    \end{equation*}
    as in \cite[(22.6)]{shimura2004arithmetic}. We also define the Hecke algebra $\mathfrak{R}(\mathcal{J}, \mathcal{Y})$ in the sense of \cite[Section 8.16]{shimura2004arithmetic} and its action on $S_k(\mathcal{J, \psi})$ as follows. For an element $\tau \in \mathcal{Y}$, we take a finite subset $Y$ of $G^+(W)_{\mathbf{h}}$, so that 
    \begin{equation}\label{decomposition hecke operator}
        \mathcal{J}\tau \mathcal{J} = \bigsqcup_{y \in Y} \mathcal{J}y.
    \end{equation}
    For $\mathbf{g} \in S_k(\mathcal{J}, \psi)$ and $x \in G^{+}(W)_{\mathbb{A}}$, we then set
    \begin{equation}\label{action of hecke operators}
        (\mathbf{g} \mid \mathcal{J}\tau\mathcal{J})(x) := \sum_{y \in Y} \mathbf{g}(xy^{-1}),
    \end{equation}
    and extend the action linearly to $\mathfrak{R}(\mathcal{J}, \mathcal{Y})$. Finally, for any $\mathbf{g} \in S_k(\mathcal{J}, \psi)$, which is a Hecke eigenform, we define the Dirichlet series
    \begin{equation*}
        \mathfrak{T}(s, \mathbf{g}) := \sum_{\tau \in \mathcal{J}\backslash \mathcal{Y}/\mathcal{J}}N(\nu(\tau)\mathbb{{Z}})^{-s}\mu(\tau),
    \end{equation*}
    where $\mu(\tau) \in \mathbb{C}$ is given by $\mathbf{g} \mid \mathcal{J}\tau \mathcal{J} = \mu(\tau)\mathbf{g}$. \\

    We now give the following Lemma, which is similar to \cite[(22.13)]{shimura2004arithmetic}, but here we consider left instead of right coset representatives. 
    
\begin{lemma}\label{L-function}

Let $\mathbf{g} \in S_0(\mathcal{J}, \mathbf{1})$, where $\mathbf{1}$ denotes the trivial character. We then have
\[
\mathfrak{T}(s,\mathbf{g}) \mathbf{g}(x) = \sum_{y \in \mathcal{Y}/ \mathcal{J}} N(\nu(y)\mathbb{Z})^{-s} \mathbf{g}(xy_{\mathbf{h}}), \,\, \forall x \in G^{+}(W)_{\mathbb{A}}.
\]
\end{lemma}

\begin{proof}

First, as in \cite[p. 208]{shimura2004arithmetic}, for each $\tau \in \mathcal{Y}$, we may find a set $Y_{\tau}$ such that
\[
\mathcal{J}\tau \mathcal{J} = \bigsqcup_{y \in Y_{\tau}} \mathcal{J} y = \bigsqcup_{y \in Y_{\tau}} y \mathcal{J}, 
\]
and so $\mathcal{J}\tau^* \mathcal{J} = \bigsqcup_{y \in Y_{\tau}} \mathcal{J} y^*  = \bigsqcup_{y \in Y_{\tau}}  \mathcal{J}\nu(y) y^{-1}$, where we have used the fact that $\mathcal{J}^* = \mathcal{J}$. Indeed, in our case the set $\mathcal{J}$ can be identified with the set $\prod_p \mathfrak{o}_p^{\times}$.

Using now the definition of the action of \eqref{action of hecke operators}, we have 
\begin{equation*}
   (\mathbf{g} \mid \mathcal{J} \tau^* \mathcal{J})(x) = \sum_{y \in Y_{\tau}} \mathbf{g}(x \nu(y)^{-1}y) = \sum_{y \in Y_{\tau}}\mathbf{g}(xy), 
\end{equation*}
because $\mathbf{g}$ is of trivial character. 

Now we use the fact that $\mathcal{Y}^* = \mathcal{Y}$ since in our case $\mathcal{Y}_p = \mathfrak{o}_p$ for $p \not \in \mathcal{Q}$, and $\mathfrak{o}_p^{\times}$ otherwise. But then we have
that 
\begin{multline*}
    \mathfrak{T}(s,\mathbf{g}) \mathbf{g}(x) = \sum_{\tau \in \mathcal{J}\backslash \mathcal{Y}/\mathcal{J}}N(\nu(\tau)\mathbb{{Z}})^{-s}\mu(\tau)\mathbf{g}(x) = \sum_{\tau \in \mathcal{J}\backslash \mathcal{Y}/\mathcal{J}}N(\nu(\tau^{*})\mathbb{{Z}})^{-s}\mu(\tau^{*})\mathbf{g}(x) = \\=\sum_{\tau \in \mathcal{J}\backslash \mathcal{Y}/\mathcal{J}}N(\nu(\tau)\mathbb{{Z}})^{-s}(\mathbf{g} \mid \mathcal{J}\tau^{*}\mathcal{J})(x) = \sum_{y \in \mathcal{Y}/ \mathcal{J}} N(\nu(y)\mathbb{Z})^{-s} \mathbf{g}(xy_{\mathbf{h}}),
\end{multline*}
because $\mathcal{J}\tau \mathcal{J} = \bigsqcup_{y \in Y_{\tau}} y \mathcal{J}$.
\end{proof}

    We now prove the following auxiliary Lemma.
    
    \begin{lemma}\label{real eigenvalues}
    For an eigenform $\mathbf{g} \in S_0(\mathcal{J},\mathbf{1})$, we have
    \[
    L(\mathbf{g},s) = L(\bar{\mathbf{g}},s).
    \]
\end{lemma}

\begin{proof}First we note that the eigenvalues of $\mathbf{g}$ are real. This follows from the fact that the Hecke operators are self-adjoint with respect to the corresponding inner product (see, for example, \cite[(22.11a)]{shimura2004arithmetic}, with $\psi$ trivial there). Hence, if we write $\mu(\tau) \in \mathbb{R}$ for the Hecke eigenvalue of $g$ with respect to the Hecke operator $\mathcal{J} \tau \mathcal{J}$, with $\tau \in \mathcal{Y}$, and consider a decomposition as in \eqref{decomposition hecke operator}, we have
\[
\left(\overline{\mathbf{g}} | \mathcal{J} \tau \mathcal{J}\right)(x)  = \sum_{y \in Y} \overline{\mathbf{g} (xy^{-1})} = \overline{\left(\mathbf{g} | \mathcal{J} \tau \mathcal{J}\right)(x) } = \overline{\mu(\tau)\mathbf{g}(x)}= \mu(\tau) \overline{\mathbf{g}}(x), 
\]
That is, $\mathbf{g}$ and $\overline{\mathbf{g}}$ have the same system of eigenvalues. But the $L$-function is attached to the system of eigenvalues, and hence they are the same. \qedhere
\end{proof}
    We are now ready to prove the following Theorem. 
    \begin{theorem}\label{sub main theorem}
        Assume $K=\mathbb{Q}(\sqrt{-m})$ with $m$ square free and $m \not \equiv 2 \pmod{4}$. Choose $\ell$ so that it satisfies the conditions of Proposition \ref{stabilisers of lattices}. For each $j=1, \cdots, h_2$, we have
        \begin{equation}\label{sub main equation}
            \sum_{i=1}^{h_1}f_{j}(\tau(\zeta_i))\widehat{\zeta}_{\mathcal{Q}}(\xi_i;s-k+3) = f_j(1) L_{\mathcal{Q}}\left(f_j;s-k+3\right).
        \end{equation}
    \end{theorem}
    \begin{proof}
        We remind the reader that we have chosen $\{w_l\}_{l=1}^{h_2}$ as a set of representatives for $B^{\times }\backslash B_{\mathbb{A}}^{\times}/ N_B(X)$, and these correspond to a set of representatives $\{v_l\}_{l=1}^{h_2}$ for $H(\xi)_{\mathbb{Q}}\backslash H(\xi)_{\mathbb{A}}/(H(\xi)_{\mathbb{A}} \cap C)$ under the diagram of Proposition \ref{diagram quaternion algebra}. Since the map $\tau$ is surjective, and using Lemma \ref{zeta hat class number one}, we can rewrite the above relation as follows:
        \begin{equation*}
           \sum_{i=1}^{h_1}f_{j}(\tau(\zeta_i))\widehat{\zeta}_{\mathcal{Q}}(\xi_i;s-k+3) = \sum_{l=1}^{h_2} f_j(v_l) \sum_{\substack{i \textup{ such that }\\ \tau(\zeta_i) = v_l}} \widehat{\zeta}_{\mathcal{Q}}(\xi_i;s-k+3).
        \end{equation*}
        We will now identify the above in the quaternion algebra $B$. If $\{u_i\}_{i=1}^{h_1}$ is a set of representatives for $B^{\times }\backslash B_{\mathbb{A}}^{\times}/X$ corresponding to the $\zeta_{i}$'s, we have from the commutative diagram of Proposition \ref{diagram quaternion algebra} that the relation $\tau(\zeta_i) = v_l$ translates to $\rho(u_i) = w_l$. \\
        
     Recall we have fixed a set $\mathcal{A} \subset G^{+}(W)_{\mathbb{Q}}\backslash G^{+}(W)_{\mathbb{A}}/(G^{+}(W)_{\mathbb{A}} \cap J)$ of pre-images for the elements $\{v_l\}_{l=1}^{h_2}$ under $\tau$ and called $\{a_l\}_{l=1}^{h_2}$ its elements. 
     If for each $j=1, \cdots, h_2$, we set $g_j:= f_j \circ \tau$, we have that these are eigenforms for the corresponding Hecke algebras. Moreover, from \cite[Remark (I), p. 209]{shimura2004arithmetic}, $L(f_j, s) = L(g_j, s)$. Hence, if we abuse notation and write $\{a_l\}_{l=1}^{h_2}$ for the corresponding elements in $B^{\times }\backslash B_{\mathbb{A}}^{\times}/X$, we have that it suffices to show
    \begin{equation}\label{what we need to show}                   
        \sum_{\ell=1}^{h_2}f_j(\tau(a_{\ell}))\left(\sum_{\substack{i \textup{ such that } \\\rho(u_i)=w_{\ell}}}\widehat{\zeta}_{\mathcal{Q}}(\xi_i;s-k+3)\right) = f_j(1) L_{\mathcal{Q}}(g_j;s-k+3).
    \end{equation}
We now use Lemma \ref{L-function} with $\mathbf{g} = g_j = f_j \circ \tau$ and set $x=1 \in G^{+}(W)_{\mathbb{A}}$. We then have

\[
\mathfrak{T}(s,g_j) g_j(1) = \sum_{y \in \mathcal{Y}/ \mathcal{J}} N(\nu(y)\mathbb{Z})^{-s} g_j(y_{\mathbf{h}}). 
\]
Now, because $\tau$ is a homomorphism, we have $g_j(1) = f_j(\tau(1)) = f_j(1)$. From \eqref{y mod j decomposition} in Lemma \ref{decomposition}, we can rewrite the above as
\[
\mathfrak{T}(s,g_j) f_j(1) = \sum_{\ell=1}^{h_2} \sum_{y \in \mathcal{Y}_{a_{\ell}}/ \mathcal{J}} N(\nu(y)\mathbb{Z})^{-s}g_j(y_\mathbf{h}) = \sum_{\ell=1}^{h_2} f_j (\tau(a_{\ell}))\sum_{y \in \mathcal{Y}_{a_{\ell}}/ \mathcal{J}} N(\nu(y)\mathbb{Z})^{-s},
\]
because if $y \in \mathcal{Y}_{a_l}$, then $\tau(y) \in H(\xi)\tau(a_l)(H(\xi)_{\mathbb{A}}\cap C)$ and hence $f_j(\tau(y_{\mathbf{h}})) = f_{j}(\tau(a_l))$.\\

It is now true from \cite[(22.12b)]{shimura2004arithmetic} that $\mathfrak{T}(s,g_j)=L(g_j,s)$. Proposition \ref{Equality of partial zeta} now gives that 
\[
    \sum_{\substack{i \textup{ such that } \\\rho(u_i)=w_{\ell}}}\widehat{\zeta}_{\mathcal{Q}}(\xi_i;s)=\sum_{ y\in \mathcal{Y}_{a_{\ell}}/\mathcal{J}} N(\nu(y)\mathbb{Z})^{-s},
\]
and hence \eqref{what we need to show} follows.
\end{proof}

We can finally state and prove the Main Theorem of this paper.

\begin{theorem}\label{main theorem}

Assume $K=\mathbb{Q}(\sqrt{-m})$ with $m$ square free and $m \not \equiv 2 \pmod{4}$.
Choose $\ell$ so that it satisfies the conditions of Proposition \ref{stabilisers of lattices}. 
Then we have,
\[
\zeta_{\mathcal{Q}}(2s-2k+4)D_{\mathcal{Q}}(\bm{F},\bm{P};s) = A(\xi)L_{\mathcal{Q}}(F;s-k+2).
\]

\end{theorem}
\begin{proof} 
By substituting \eqref{sub main equation} from Theorem \ref{sub main theorem} in \eqref{what we have from Dirichlet}, using also Lemma \ref{real eigenvalues}, we obtain
\begin{equation*}
   \zeta_{\mathcal{Q}}(2s-2k+4)D_{\mathcal{Q}}(\bm{F},\bm{P}; s) = C \times L_{\mathcal{Q}}(F;s-k+2),
\end{equation*}
where $C:=\sum_{j=1}^{h_2}A_{f_j} f_j(1)$. We only need to show $C = A(\xi)$. To show this, we set $\xi_i = \xi$ in \ref{expression}. From the definition of $\lambda$ in Proposition \ref{commutative diagram 1}, we have that the $\Gamma(L_0)$-class of $\xi$ corresponds to the trivial class in $\widetilde{H}(\xi)_{\mathbb{Q}}\backslash \widetilde{H}(\xi)_{\mathbb{A}}/(\widetilde{H}(\xi)_{\mathbb{A}}\cap \widetilde{C})$. After clearing denominators in \ref{expression}, and considering the coefficient $1^{-s}$ in the series expansion, we obtain $C = A(\xi)$, and hence the Theorem follows.

\end{proof}
\begin{remark}
    The fact that the constant $C$ equals $A(\xi)$ can also be deduced from the fact that $\{f_j\}_{j=1}^{h_2}$ is an orthonormal basis for $V(\xi)$.
\end{remark}

We now write $\mathcal{V}_k \subset S_{k}(\widehat{\Gamma})$ for the span of $\mathcal{P}_{k,D,0}$ for all valid choices of $D$, as in the Theorem above. We write $\bm{V}_k \subset S_k(\widehat{D}_f)$ for the corresponding space of automorphic forms. We then obtain the following Corollary.
\begin{corollary}\label{Theorem for G}
     Let $\bm{G} \in \bm{V}_k$. Then, for a finite set of primes $\mathcal{Q}$ sufficiently large we have that
     \[
     \zeta_{\mathcal{Q}}(2s-2k+4)D_{\mathcal{Q}}(\bm{F}, \bm{G};s) =\widetilde{C} \,L_{\mathcal{Q}}(F;s-k+2),
     \]
     where $\widetilde{C} \in \mathbb{C}$ is given as
     \[
     \widetilde{C} = (\# \textup{Cl}(K/\mathbb{Q}))
     \sum_{\mf{m} \in \textup{Cl}(K)/\textup{Cl}(K/\mathbb{Q})} \langle \phi_{\mf{m},1},\psi_{\mf{m},1}\rangle.
     \]
     Here $\phi_{\mf{m},1}$ and $\psi_{\mf{m},1}$ are the first Fourier-Jacobi coefficients of $F_{\mf{m}}$ and $G_{\mf{m}}$, respectively. 
\end{corollary}
\begin{proof}
    From the Theorem above, it follows that $\zeta_{\mathcal{Q}}(2s-2k+4)D_{\mathcal{Q}}(\bm{F}, \bm{G};s)$ is equal to $L_{\mathcal{Q}}(F;s-k+2)$ up to a constant, when we take $\mathcal{Q}$ large enough to apply our Theorem for each Poincar\'e series in the linear combination that makes up $G$. This constant is equal to $\widetilde{C}$, by comparing the coefficients of $1^{-s}$ in the expansion of the Dirichlet series.
\end{proof}

\renewcommand{\abstractname}{Acknowledgements}
\begin{abstract}
Part of this work was completed during the second-named author's PhD studies. This author was supported by the HIMR/UKRI "Additional Programme for Mathematical Sciences", grant number EP/V521917/1, as well as by a scholarship from the Onassis Foundation.
\end{abstract}

\printbibliography

@article{andrianov,
year = {1974},
publisher = {},
volume = {\textbf{29}},
number = {\textbf{3}},
pages = {45-116},
author = {A. N. Andrianov},
title = {{Euler Products corresponding to Siegel modular forms of genus 2}},
journal = {Russian Mathematical Surveys},
}

@article{gritsenko,
    author = {Gritsenko, V.A.},
    title = {{Jacobi functions and Euler products for Hermitian modular forms.}},
    journal = {J Math Sci},
    volume = {\textbf{63}},
    pages = {2883-2914},
    year = {1992},
}

@article{kohnen_skoruppa,
    author = {Kohnen, W. and Skoruppa, N. P.},
    title = {{A certain Dirichlet series attached to Siegel modular forms of degree two.}},
    journal = {Invent Math},
    volume = {\textbf{95}},
    pages = {541–558},
    year = {1989},
}

@inbook{krieg_integral_orthogonal,
author = {A. Krieg },
title = {Integral Orthogonal Groups},
booktitle = {{Dynamical Systems, Number Theory and Applications}},
chapter = {Chapter 9},
pages = {177-195}
}

@article{krieg_jacobi,
title = {{Jacobi Forms of Several Variables and the Maaß Space}},
journal = {Journal of Number Theory},
volume = {\textbf{56}},
number = {\textbf{2}},
pages = {242-255},
year = {1996},
author = {A. Krieg},
}

@PHDTHESIS{jacobi_lattice,
           title = {{On Jacobi forms of lattice index}},
          author = {A. Mocanu},
          school = {University of Nottingham},
            year = {2019},
             url = {http://eprints.nottingham.ac.uk/56893/},
}

@phdthesis{eisenstein_thesis,
      author       = {Schaps, F.},
      othercontributors = {Krieg, Aloys and Heim, Bernhard and Alfes-Neumann, Claudia},
      title        = {{{E}isenstein series for the orthogonal group $\textup{O}(2,n)$}},
      school       = {RWTH Aachen University},
      type         = {Dissertation},
      publisher    = {RWTH Aachen University},
      reportid     = {RWTH-2022-08607},
      year         = {2022},
      cin          = {114110 / 110000},
      ddc          = {510},
      cid          = {$I:(DE-82)114110_20140620$ / $I:(DE-82)110000_20140620$},
      typ          = {PUB:(DE-HGF)11},
      url = {https://publications.rwth-aachen.de/record/853111/files/853111.pdf}
}

@article{shimura_orthogonal,
author = {Shimura, G.},
title = {{An exact mass formula for orthogonal groups}},
volume = {\textbf{97}},
journal = {Duke Mathematical Journal},
number = {\textbf{1}},
publisher = {Duke University Press},
pages = {1 -- 66},
year = {1999},
}

@book{shimura2004arithmetic,
  title={{Arithmetic and Analytic Theories of Quadratic Forms and Clifford Groups}},
  author={Shimura, G.},
  lccn={2003063826},
  series={Mathematical surveys and monographs},
  year={2004},
  publisher={American Mathematical Society}
}

@article{shimura_classification_quadratic_forms,
 author = {Shimura, G.},
 journal = {American Journal of Mathematics},
 number = {\textbf{6}},
 pages = {1521--1552},
 publisher = {Johns Hopkins University Press},
 title = {{Classification, Construction, and Similitudes of Quadratic Forms}},
 volume = {\textbf{128}},
 year = {2006},
}

@book{Shimura_Euler_Product,
    AUTHOR = {Shimura, G.},
     TITLE = {{Euler products and {E}isenstein series}},
    SERIES = {CBMS Regional Conference Series in Mathematics},
    VOLUME = {\textbf{93}},
      YEAR = {1997},
     PAGES = {xx+259},
}

@article{shimura_diophantine_2006,
  title={{Integer-valued quadratic forms and quadratic Diophantine equations.}},
  author={Shimura, G.},
  journal={Doc. Math.},
  volume={\textbf{11}},
  pages={333-367},
  year={2006},
}

@article{sugano,
author = {Sugano, T.},
title = {{On Dirichlet Series Attached to Holomorphic Cusp Forms on $\hbox{SO}(2,q)$}},
volume = {\textbf{62}},
journal = {Adv. Stud. Pure Math.},
pages = {333-362},
year = {1985},
}

@article{sugano_lifting,
  title={{Jacobi Forms and the Theta Lifting}},
  author={Sugano, T.},
  year={1995},
  journal = {Comment. Math. Univ. St. Paul},
  volume = {\textbf{44}},
  number = {\textbf{1}},
  pages = {1-58},
}

@phdthesis{wernz,
  title={On Hermitian modular groups and modular forms},
  author={Wernz, A.},
  year={2019},
  school={Dissertation, RWTH Aachen University},
  url = {https://publications.rwth-aachen.de/record/764825/files/764825.pdf}
}

@article{shimura2006quadratic,
  title={Quadratic Diophantine equations and orders in quaternion algebras},
  author={Shimura, G.},
  journal={American journal of mathematics},
  volume={\textbf{128}},
  number={\textbf{2}},
  pages={481--518},
  year={2006},
  publisher={Johns Hopkins University Press}
}

@article{murata2014,
  title={On quadratic Diophantine equations in four variables and orders associated with lattices},
  author={Murata, M.},
  journal={Documenta Mathematica},
  volume={\textbf{19}},
  pages={247--284},
  year={2014}
}

@PHDTHESIS{thesis_hein,
           title = {{Orthogonal modular forms: an application to a conjecture of Birch, algorithms and computations}},
          author = {Hein, J.},
          school = {Dartmouth College},
            year = {2016},
             url = {https://math.dartmouth.edu/theses/grad/2016/hein-thesis-abstract.pdf},
}

@article{psyroukis_orthogonal,
      title={A Fourier-Jacobi Dirichlet series for cusp forms on orthogonal groups}, 
      author={R. Psyroukis},
      year={2025},
      journal = {Research in Number Theory},
      volume = {\textbf{90}},
      number = {\textbf{11}},
}

@article{Latimer1,
  title={On Ideals in Generalised Quaternion Algebras and Hermitian Forms},
  author={Latimer, C.},
  journal={Transactions of AMS},
  volume={\textbf{38}},
  pages={436--446},
  year={1935}
}

@article{Latimer2,
  title={On Ideals in a Quaternion Algebra and the Representation of Integers by Hermitian Forms},
  author={Latimer, C.},
  journal={Transactions of AMS},
  volume={\textbf{40}},
  pages={439--449},
  year={1936}
}

@unpublished{murata_preprint,
      title="{Invariants and discriminant ideals of orthogonal complements in a quadratic space}", 
      author={M. Murata},    
      note={Preprint},
      year={2018},
      url={https://arxiv.org/abs/1112.2454},
}

@book{Elstrodt,
  title={Groups Acting on Hyperbolic Space: Harmonic Analysis and Number Theory},
  author={Elstrodt, J. and Grunewald, F. and Mennicke, J.},
  year={1997},
  publisher={Springer Science \& Business Media}
}

@book{voight,
  title={Quaternion algebras},
  author={Voight, J.},
  year={2021},
  publisher={Springer Nature}
}

@article{paper,
title={Integral canonical models for Spin Shimura varieties}, 
volume={\textbf{152}}, 
number={\textbf{4}}, 
journal={Compositio Mathematica}, 
author={Pera, K. M.}, 
year={2016}, 
pages={769–824},
}
\end{document}